\documentclass[10pt]{article} 
\usepackage[utf8]{inputenc} 
\usepackage[english]{babel}
\usepackage[T1]{fontenc}
\usepackage{latexsym}

\usepackage{geometry} 
\usepackage{booktabs} 
\usepackage{array} 
\usepackage{paralist} 
\usepackage{verbatim} 
\usepackage{subfig} 
\usepackage{amsmath}
\usepackage{amsfonts}
\usepackage{hyperref}
\usepackage{bbm}
\usepackage{babel}
\usepackage{amsthm}
\usepackage{amssymb}
\usepackage{mathtools}
\usepackage{eufrak}
\usepackage{bbold}

\makeatletter
\newtheorem*{rep@theorem}{\rep@title}
\newcommand{\newreptheorem}[2]{%
\newenvironment{rep#1}[1]{%
 \def\rep@title{#2 \ref{##1}}%
 \begin{rep@theorem}}%
 {\end{rep@theorem}}}
\makeatother

\newtheorem{thm}{Theorem}[section]
\newtheorem{cor}[thm]{Corollary}
\newtheorem{lem}[thm]{Lemma}
\newtheorem{prop}[thm]{Proposition}
\newtheorem{defn}[thm]{Definition}
\newtheorem{rem}[thm]{Remark}
\newtheorem{esem}[thm]{Example}

\newtheorem*{cor*}{Corollary}

\newtheorem{theorem}{Theorem}
\newreptheorem{theorem}{Theorem}

\newreptheorem{proposition}{Proposition}

\newreptheorem{corollary}{Corollary}

\newtheorem{question}{Question}
\newreptheorem{question}{Question}

\newtheorem{conjecture}{Conjecture}
\newreptheorem{conjecture}{Conjecture}

\usepackage{fancyhdr} 
\usepackage{sectsty}
\allsectionsfont{\sffamily\mdseries\upshape} 

\usepackage[nottoc,notlof,notlot]{tocbibind} 
\usepackage[titles,subfigure]{tocloft} 

\title{Kähler-Einstein metrics with prescribed singularities\\ on Fano manifolds}
\author{Antonio Trusiani\footnote{email: antonio.trusiani@math.univ-toulouse.fr}}
\date{} 

\begin{document}
\maketitle
\begin{abstract}
Given a Fano manifold $(X,\omega)$ we develop a variational approach to characterize analytically the existence of Kähler-Einstein metrics with prescribed singularities, assuming that these singularities can be approximated algebraically.\newline
Moreover we define a function $\alpha_{\omega}$ on the set of prescribed singularities which generalizes Tian's $\alpha$-invariant, showing that its upper lever set $\{\alpha_{\omega}(\cdot)>\frac{n}{n+1}\}$ produces a subset of the Kähler-Einstein locus, i.e. of the locus given by all prescribed singularities that admit Kähler-Einstein metrics. In particular, we prove that many $K$-stable manifolds admit all possible Kähler-Einstein metrics with prescribed singularities. Conversely, we show that enough positivity of the $\alpha$-invariant function at non-trivial prescribed singularities (or other conditions) implies the existence of genuine Kähler-Einstein metrics.\newline
Finally, through a continuity method we also prove the strong continuity of Kähler-Einstein metrics on curves of totally ordered prescribed singularities when the relative automorphism groups are discrete.
\end{abstract}
\vspace{5pt}
{\small \textbf{Keywords:} Kähler-Einstein metrics, Fano manifolds, complex Monge-Ampère equations, alpha-invariants, log canonical thresholds.\newline
}
\section{Introduction}
A Fano manifold $X$ admits a Kähler-Einstein (KE) metric if and only if $(X,-K_{X})$ is \emph{K-stable} (\cite{CDS15}). This is the famous solution to the Yau-Tian-Donaldson conjecture for the anticanonical polarization, and it connects a differential-geometric notion to a purely algebro-geometric notion as predicted by S.T. Yau (\cite{Yau93}).\newline
There are now two natural possible singular versions of this correspondence: when $X$ is singular or when the metric has some \emph{prescribed singularities}. In this article we deal with the second problem.\newline

Letting $\omega$ be a Kähler form with cohomology class $c_{1}(X)$, as any KE metric corresponds to a function $u\in PSH(X,\omega)\cap \mathcal{C}^{\infty}(X)$ such that
\begin{equation}
\label{eqn:Ricci}
\mbox{Ric}(\omega+dd^{c}u)=\omega+dd^{c}u
\end{equation}
where $d^{c}:=\frac{i}{4\pi}(\bar{\partial} -\partial)$ so that $dd^{c}=\frac{i}{2\pi}\partial\bar{\partial}$, the most natural extension to the prescribed singularities setting is to fix $\psi\in PSH(X,\omega)$ and to look for $u\in PSH(X,\omega)$ which satisfies (\ref{eqn:Ricci}) in a singular sense and such that $u-\psi$ is globally bounded. We refer to section \S\ref{sec:KE} for the precise definition of \emph{Kähler-Einstein metrics with prescribed singularities $[\psi]$ ($[\psi]$-KE metrics)}, here we underline its characterization in terms of Monge-Ampère equations: by abuse of language $\omega+dd^{c}u$ is a $[\psi]$-KE metric if and only if $u$ solves
\begin{equation}
\label{eqn:MAINT}
\begin{cases}
MA_{\omega}(u)=e^{-u+C}\mu\\
u\in\mathcal{E}^{1}(X,\omega,\psi).
\end{cases}
\end{equation}
for $C\in\mathbbm{R}$. The measure $\mu$ in (\ref{eqn:MAINT}) is the usual smooth volume form on $X$ given as $\mu=e^{-\rho}\omega^{n}$ for $\rho$ Ricci potential, $MA_{\omega}(u):=\langle (\omega+dd^{c}u)^{n}\rangle$ is the Monge-Ampère measure of $u$ in terms of the \emph{non-pluripolar product} (see \cite{BEGZ10}) while $\mathcal{E}^{1}(X,\omega,\psi)$ is the set of all $u\in PSH(X,\omega)$ more singular than $\psi$, i.e. $u\leq \psi+C$ for $C\in \mathbbm{R}$, such that the $\psi$-relative energy $E_{\psi}(u)$ is finite (see \cite{DDNL17b}, \cite{Tru19}, \cite{Tru20a}). Note that the set $\mathcal{E}^{1}(X,\omega,\psi)$ contains all $u$ such that $u-\psi$ is globally bounded.\newline
Recalling that $PSH(X,\omega)$ is naturally endowed with a partial order $u\preccurlyeq v$ if $u\leq v+C$, the following conditions on $\psi$ are necessary to solve (\ref{eqn:MAINT}):
\begin{itemize}
\item[i)] $\psi=P_{\omega}[\psi]:=\Big(\sup\{u\in PSH(X,\omega)\, : \, u\preccurlyeq \psi,u\leq 0\}\Big)^{*}$ where the star is for the upper semicontinuous regularization;
\item[ii)] $V_{\psi}:=\int_{X}MA_{\omega}(\psi)>0$;
\item[iii)] $\mathcal{I}(\psi)=\mathcal{O}_{X}$ where $\mathcal{I}(\psi)$ is the multiplier ideal sheaf attached to $\psi$.
\end{itemize}
The first condition means that $\psi$ is a \emph{model type envelope}, i.e. $\psi\in\mathcal{M}$ (it is shown to be necessary in \cite{DDNL17b}), while we will say that $(X,\psi)$ is \emph{Kawamata Log Terminal (klt)} when $(iii)$ holds (see Remark \ref{rem:Assumptions} for more details on these necessary conditions). With obvious notations we denote respectively by $\mathcal{M}^{+}, \mathcal{M}_{klt}^{+}$ the sets of model type envelopes that satisfy respectively $(ii)$, $(ii)$ and $(iii)$. Thus $\mathcal{M}_{klt}^{+}$ can be thought as the set of all admissible prescribed singularities and it is natural to define the \emph{Kähler-Einstein locus} as
$$
\mathcal{M}_{KE}:=\{\psi\in\mathcal{M}_{klt}^{+}\, :\, \mbox{there exists a }[\psi]\mbox{-KE metric}\}.
$$
Then, observing that the $[0]$-KE metrics are the genuine KE metrics, it is immediate to ask the following questions.
\begin{question}
\label{Question}
Let $(X,\omega)$ be a Fano manifold. Is it possible to characterize $\mathcal{M}_{KE}$? When $\mathcal{M}_{KE}=\mathcal{M}_{klt}^{+}$? What geometric properties does the set $\mathcal{M}_{KE}$ possess?
\end{question}

To start addressing Question \ref{Question}, we first define the $\alpha$-invariant function $\mathcal{M}\ni \psi\to \alpha_{\omega}(\psi)\in (0,+\infty)$,
\begin{equation}
\label{eqn:AlphaIntro}
\alpha_{\omega}(\psi):=\sup\Big\{\alpha\geq 0\, :\, \sup_{u\preccurlyeq \psi, \sup_{X}u=0}\int_{X}e^{\alpha(\psi- u)}e^{-\psi}d\mu<+\infty \Big\},
\end{equation}
if $\psi\in\mathcal{M}_{klt}$ and $\alpha_{\omega}(\psi):=0$ otherwise, and its \emph{modified version}
\begin{equation}
    \label{eqn:AlphaIntro2}
    \tilde{\alpha}_{\omega}(\psi):=\sup\Big\{\alpha\geq 0\, :\, \sup_{u\preccurlyeq \psi, \sup_{X}u=0}\int_{X}e^{-\alpha u}d\mu<+\infty \Big\}.
\end{equation}
Both these functions generalize to the $\psi$-relative setting the classical Tian's $\alpha$-invariant (\cite{Tian87}), and we have the following result.
\begin{theorem}
\label{thmA2}
Let $(X,\omega)$ be a Fano manifold. Then
$$
\Big\{\psi\in\mathcal{M}_{klt}^{+}\, : \, \alpha_{\omega}(\psi)>\frac{n}{n+1}\Big\}\subset \mathcal{M}_{KE}.
$$
Moreover, letting $\mathrm{lct}(X,\psi):=\sup\big\{t>0\,:\, (X,t\psi)\, \mbox{is klt}\}$,
\begin{itemize}
    \item[i)] $\alpha_{\omega}(0)>\frac{n}{n+1}$ implies
    $$
    \Big\{\psi\in\mathcal{M}_{klt}^{+}\, : \, \mathrm{lct}(X,\psi)>1+\frac{1-\alpha_{\omega}(0)}{\frac{n+1}{n}\alpha_{\omega}(0)-1}\Big\}\subset \mathcal{M}_{KE}.
    $$
    In particular $\alpha_{\omega}(0)\geq 1$ implies $\mathcal{M}_{KE}=\mathcal{M}_{klt}^{+}$.
\end{itemize}
On the other hand, each of the followings implies $0\in\mathcal{M}_{KE}$:
\begin{itemize}
    \item[ii)] there exists $\psi\in\mathcal{M}$, $t\in(0,1]$ such that $ \tilde{\alpha}_{\omega}(\psi_{t})>\frac{n}{(n+1)t}$ for $\psi_{t}:=P_{\omega}[(1-t)\psi]$;
    \item[iii)] given $\alpha<\alpha_{\omega}(0)$ there exists $\psi\in\mathcal{M}_{klt}^{+}$ such that $\begin{cases}
        \tilde{\alpha}_{\omega}(\psi)>C_{1,\psi}=C_{1,\psi}(\alpha,V_{\psi}/V_{0}) \\ V_{\psi}/V_{0}> D_{1}=D_{1}(\alpha);
    \end{cases}$
    \item[iv)] given $\alpha<\alpha_{\omega}(0)$ there exists $\psi\in\mathcal{M}_{klt}^{+}$ such that $\begin{cases}
        \alpha_{\omega}(\psi)>C_{2,\psi}=C_{2,\psi}\big(\alpha,V_{\psi}/V_{0},\mathrm{lct}(X,\psi)\big) \\
        V_{\psi}/V_{0}> D_{2,\psi}=D_{2,\psi}\big(\alpha,\mathrm{lct}(X,\psi)\big).
    \end{cases}$
\end{itemize}
\end{theorem}
We refer to section \S\ref{sec:PSS} for the explicit formulas of the constants $C_{i,\psi}, D_{1},D_{2,\psi}$ in Theorem \ref{thmA2}, but the upshot of $(ii), (iii)$ is that the values of $\alpha_{\omega}(\cdot), \tilde{\alpha}_{\omega}(\cdot)$ at singular model type envelopes can help to understand if $X$ admits a genuine Kähler-Einstein metrics.\newline
Let us stress that the advantage of the relative setting is that we can \emph{choose} $\psi\in\mathcal{M}_{klt}^{+}$. Indeed, the more singular $\psi$ is, the less functions are involved in the definitions in (\ref{eqn:AlphaIntro}) and in (\ref{eqn:AlphaIntro2}), and hence the more calculable $\alpha_{\omega}(\psi),  \tilde{\alpha}_{\omega}(\psi)$ are. See for instance subsection \S \ref{ssec:Estimate} where, for $\psi\in\mathcal{M}_{klt}^{+}$ having isolated singularities at $N$ points, we produce lower bounds for the $\psi$-relative $\alpha$-invariant in terms of the \emph{multipoint Seshadri constants} and \emph{pseudoeffective thresholds}.\newline

Theorem \ref{thmA2}$.(i)$ instead shows that $\mathcal{M}_{KE}=\mathcal{M}_{klt}^{+}$ as soon as $\alpha_{\omega}(0)\geq 1$, which together with the discussion below about a new continuity method with \emph{movable singularities} leads to the following natural conjecture.
\begin{conjecture}
\label{conjA}
Let $(X,\omega)$ be a Fano manifold such that $\mbox{Aut}(X)^{\circ}=\{\mbox{Id}\}$. Then
$$
0\in\mathcal{M}_{KE} \Longleftrightarrow \mathcal{M}_{KE}=\mathcal{M}_{klt}^{+}.
$$
\end{conjecture}
By $\mbox{Aut}(X)^{\circ}$ we denote the connected component of the identity map.  Note that the assumption $\mbox{Aut}(X)^{\circ}=\{\mathrm{Id}\}$ is necessary as $X=\mathbbm{P}^{2}$ shows (Example \ref{esem:1}).\newline

Moreover, when $\alpha_{\omega}(0)>\frac{n}{n+1}$, Theorem \ref{thmA2}$.(i)$ implies the existence of many log-KE metrics for weak log Fano pairs $(Y,\Delta)$ given by resolutions of integrally closed coherent analytic sheaves. Indeed, among $\mathcal{M}_{klt}^{+}$ there are particular model type envelopes associated to analytic singularities. Namely, we say that $\psi$ has \emph{analytic singularities type} if $\psi=P_{\omega}[\varphi]$ for $\varphi\in PSH(X,\omega)$ with analytic singularities formally encoded in $(\mathcal{I},c)$. In this case, taking $p:Y\to X$ resolution of $\mathcal{I}$, the set of $[\psi]$-KE metrics is in correspondence with the set of log-KE metrics for the log pair $(Y,\Delta)$ where $\Delta:=cD-K_{Y/X}$ and $p^{*}\mathcal{I}=\mathcal{O}_{Y}(-D)$ (see \cite{Tru20b}). If the singularities are \emph{algebraic} (i.e. $c\in\mathbbm{Q}$), we prove that
$$
\alpha_{\omega}(\psi)=\alpha(Y,\Delta),
$$
i.e. $\alpha_{\omega}(\psi)$ coincides with the usual log $\alpha$-invariant $\alpha(Y,\Delta)$ (see Proposition \ref{prop:AlphaAlg}).\newline

In particular, if $D$ is a smooth irreducible divisor in $|-rK_{X}|$, $\varphi_{D}\in PSH(X,\omega)$ such that $\omega+dd^{c}\varphi_{D}=\frac{1}{r}[D]$ and $\psi_{t}:=P_{\omega}[t\varphi_{D}]$ for any $t\in[0,1]$, finding a $[\psi_{t}]$-KE metric is equivalent to find a KE metric $\omega_{u_{t}}:=\omega+dd^{c}u_{t}$ with cone singularities along $D$ of angle $2\pi(1-t)/r$, i.e. $u_{t}\in PSH(X,\omega)$ locally bounded such that
$$
Ric(\omega_{u_{t}})=t\omega_{u_{t}}+\frac{(1-t)}{r}[D].
$$
This is the path considered in \cite{CDS15} to solve the Yau-Tian-Donaldson Conjecture and it is well-known that there exists $t_{0}\in(0,1]$ such that $\psi_{t}\in\mathcal{M}_{KE}$ for any $t\in (0,t_{0})$ (\cite{Berm13}, \cite{JMR11}, see also Remark \ref{rem:CDS}) and
\begin{equation*}
\{\psi_{t}\}_{t\in(0,1]}\subset \mathcal{M}_{KE}\quad \Longleftrightarrow  \quad 0\in \mathcal{M}_{KE}.
\end{equation*}
In particular Theorem \ref{thmA2} gives a valuative criterion to detect if the curve $\{\psi_{t}\}_{t\in(0,1]}$ is entirely contained in $\mathcal{M}_{KE}$.\newline

Since for any $\psi\in\mathcal{M}_{klt}^{+}$ the curve $[0,1]\ni t\to \psi_{t}=P_{\omega}[(1-t)\psi]$ belongs to $\mathcal{M}_{klt}^{+}$ and it is weakly continuous (see Lemma \ref{lem:Star}), i.e. continuous with respect to the natural $L^{1}$-topology on $\mathcal{M}_{klt}^{+}$, it is then natural to wonder if it is possible to perform a continuity method for this path, generalising the underlying idea in \cite{CDS15}.\newline
In the companion paper \cite{Tru20b} we introduced a continuity method with \emph{movable singularities} based on the \emph{strong topology} of $\omega$-psh functions given as the coarsest refinement of the weak topology such that the energy $E_{\cdot}(\cdot)$ becomes continuous (\cite{Tru19}, \cite{Tru20a}). Thus, denoting with $\mathcal{M}_{D}$ the subset of $\mathcal{M}$ of all model type envelopes that are approximable by decreasing sequences of model type envelopes with algebraic singularities (see section \S \ref{sec:Prelim}), we show the following result.
\begin{theorem}
\label{thmA}
Let $(X,\omega)$ be a Fano manifold and let $\{\psi_{t}\}_{t\in[0,1]}\subset\mathcal{M}_{klt}^{+}$ be a weakly continuous segment such that
\begin{itemize}
\item[i)] $\psi_{0}\in\mathcal{M}_{KE}$;
\item[ii)] $\{\psi_{t}\}_{t\in[0,1]}\subset\mathcal{M}_{D}$;
\item[iii)] $\psi_{t}\preccurlyeq \psi_{s}$ if $t\leq s$;
\item[iv)] $\mbox{Aut}(X,[\psi_{t}])^{\circ}=\{\mbox{Id}\}$ for any $t\in [0,1]$.
\end{itemize}
Then the set
$$
S:=\{t\in[0,1]\, : \, \psi_{t}\in \mathcal{M}_{KE}\}
$$
is open, the unique family of $[\psi_{t}]$-KE metrics $\{\omega_{u_{t}}\}_{t\in S}$ is weakly continuous and the family of potentials $\{u_{t}\}_{t\in S}$ can be chosen so that the curve $S\ni t\to u_{t}\in \mathcal{E}^{1}(X,\omega,\psi_{t})$ is strongly continuous.
\end{theorem}
In Theorem \ref{thmA} $\mbox{Aut}(X,[\psi])^{\circ}:=\mbox{Aut}(X)^{\circ}\cap \mbox{Aut}(X,[\psi])$ where $\mbox{Aut}(X,[\psi])$ is the set of automorphisms $F:X\to X$ such that $F^{*}\psi-\psi$ is globally bounded. $(iv)$ is a necessary hypothesis for the uniqueness of $[\psi]$-KE metrics as explained below in Theorem \ref{thmB}.\newline
The set $\mathcal{M}_{D}$ contains plenty of model type envelopes, but in general $\mathcal{M}_{D}\subsetneq \mathcal{M}$ (see Example \ref{esem:StrangePiu}). However, it is worth to underline that if $\psi\in \mathcal{M}_{D}^{+}:=\mathcal{M}_{D}\cap \mathcal{M}^{+}$ then $\psi_{t}=P_{\omega}[(1-t)\psi]\in\mathcal{M}_{D}^{+}$ for any $t\in[0,1]$ (Proposition \ref{prop:Resc}), thus Theorem \ref{thmA} includes these interesting paths.\newline

To prove Theorems \ref{thmA2}, \ref{thmA} we develop a variational approach to study the existence of $[\psi]$-KE metrics for a fixed $\psi\in\mathcal{M}_{klt}^{+}$ (see \cite{BBGZ09}, \cite{DR15} for the absolute setting). Namely, we define two translation invariant functionals $D_{\psi}$, $M_{\psi}$, called respectively the $\psi$-relative Ding and Mabuchi functional, which generalize the well-known Ding and Mabuchi functionals to the $\psi$-relative setting as our next result shows.
\begin{theorem}
\label{thmB}
Let $\psi\in\mathcal{M}^{+}_{D,klt}:=\mathcal{M}^{+}_{klt}\cap\mathcal{M}_{D}$ and let $u\in \mathcal{E}^{1}(X,\omega,\psi)$. Then the following statements are equivalent:
\begin{itemize}
\item[i)] $\omega_{u}:=\omega+dd^{c}u$ is a $[\psi]$-KE metric;
\item[ii)] $D_{\psi}(u)=\inf_{\mathcal{E}^{1}(X,\omega,\psi)}D_{\psi}$;
\item[iii)] $M_{\psi}(u)=\inf_{\mathcal{E}^{1}(X,\omega,\psi)}M_{\psi}$.
\end{itemize}
Moreover if $\omega_{u}$ is a $[\psi]$-KE metric then $u$ has $\psi$-relative minimal singularities (i.e. $u-\psi$ globally bounded) and if $\omega_{v}$ is another $[\psi]$-KE metric then there exists $F\in\mbox{Aut}(X,[\psi])^{\circ}$ such that $F^{*}\omega_{v}=\omega_{u}$. 
\end{theorem}

Next, when $\mbox{Aut}(X,[\psi])^{\circ}=\{\mbox{Id}\}$, we prove that the existence of the unique $[\psi]$-KE metric is equivalent to the \emph{coercivity} of the $\psi$-relative Ding and Mabuchi functionals as in the absolute setting. We recall that the strong topology on $\mathcal{E}^{1}(X,\omega,\psi)$ is a metric topology given by a complete distance $d$ which generalizes to the $\psi$-relative setting the distance introduced by T. Darvas (\cite{Dar15}) as proved in our previous works \cite{Tru19},\cite{Tru20a}.
\begin{theorem}
\label{thmC}
Let $\psi\in\mathcal{M}_{D,klt}^{+}$ and assume $\mbox{Aut}(X,[\psi])^{\circ}=\{\mbox{Id}\}$. Then the following conditions are equivalent:
\begin{itemize}
\item[i)] the $\psi$-relative Ding functional is $d$-coercive over $\mathcal{E}^{1}_{norm}(X,\omega,\psi):=\{u\in\mathcal{E}^{1}(X,\omega,\psi)\, : \, \sup_{X}u=0\}$;
\item[ii)] the $\psi$-relative Mabuchi functional is $d$-coercive over $\mathcal{E}^{1}_{norm}(X,\omega,\psi)$;
\item[iii)] there exists a unique $[\psi]$-KE metric.
\end{itemize}
\end{theorem}
\subsection{About the assumptions.}
\label{ssec:Hypothesis}
The assumption "$\psi\in\mathcal{M}_{D}$" in Theorems \ref{thmA}, \ref{thmB} and \ref{thmC} is exclusively due to the fact that the linearity of the $\psi$-relative energy along weak geodesic segments holds in the class $\mathcal{M}_{D}^{+}$ (see Theorem \ref{thm:Linear}). In other words, proving Theorem \ref{thm:Linear} for general $\psi\in\mathcal{M}^{+}$ would extend Theorems \ref{thmA}, \ref{thmB} and \ref{thmC} to $\psi\in\mathcal{M}^{+}_{klt}$.\newline
In Theorem \ref{thmA} if we replace $(iv)$ with $\mbox{Aut}(X,[\psi_{t}])^{\circ}=\{\mbox{Id}\}$ for any $t\in[0,1)$, which may be useful when $\mbox{Aut}(X)$ is not discrete, then the openness and the strong continuity result hold in $[0,1)$. However in this situation is unclear if it may happen that $S=[0,1]$ but the family of KE metrics $\{\omega_{u_{t}}\}_{t\in[0,1)}$ does not converge to a $[\psi_{1}]$-KE metric. Indeed the \emph{closedness} of the continuity method depends on a uniform bounds on the supremum of the potentials appropriately chosen (as in other more classical continuity methods), and in the proof of Theorem \ref{thmA} this estimate is basically a consequence of a uniform coercivity.\newline
Finally note that in Theorem \ref{thmA2} there are no assumptions on $\mathcal{M}_{D}$. Indeed this is a consequence of the fact that the arrow $(i)\Rightarrow (iii)$ in Theorem \ref{thmC} holds for general $\psi\in\mathcal{M}_{klt}^{+}$.
\subsection{Related Works}
During the last period of the preparation of this article, T. Darvas and M. Xia in (\cite{DX20}) independently defined the same set $\mathcal{M}_{D}$, exploring deeply its properties and its relations with the algebraic approximations of geodesic rays in $\big(\mathcal{E}^{1}(X,\omega),d\big)$ where $\{\omega\}=c_{1}(L)$ for $L$ ample line bundle. 
\subsection{Structure of the paper}
In the next two sections we work with a general compact Kähler manifold $(X,\omega)$, i.e. $\omega$ is not necessarily integral. In Section \S \ref{sec:Prelim} we collect some preliminaries on model type envelopes and on the strong topologies, while in Section \S \ref{sec:Model} we define the set $\mathcal{M}_{D}$, characterizing it through a version of the Demailly's regularization Theorem (Theorem \ref{thm:Reg}). Moreover in the same section we prove the linearity of the Monge-Ampère energy $E_{\psi}(\cdot)$ along weak geodesic segments for $\psi\in\mathcal{M}^{+}_{D}$, showing also that $\big(\mathcal{E}^{1}(X,\omega,\psi),d\big)$ is a geodesic metric space.\newline
In Section \S \ref{sec:KE} we assume $\{\omega\}=c_{1}(X)$ and we develop the variational approach to study Kähler-Einsten metrics with prescribed singularities. We prove Theorems \ref{thmB} and \ref{thmC}. Furthermore, as the $\psi$-relative $\alpha$-invariant is an important tool to show these two theorems, the subsection \S \ref{ssec:Alpha} is dedicated to explore some of its analytic and algebraic properties. Finally, Section \S \ref{sec:PSS} contains the proof of Theorems \ref{thmA2}, \ref{thmA}.
\subsection{Acknowledgments}
I would like to thank my PhD advisors Stefano Trapani and David Witt Nyström for their comments. The author is supported by a postdoctoral grant of the Knut and Alice Wallenberg Foundation.
\section{Preliminaries}
\label{sec:Prelim}
Letting $(X,\omega)$ be a compact Kähler manifold endowed with a Kähler form $\omega$, we denote by $PSH(X,\omega)$ the set of $\omega$-plurisubharmonic ($\omega$-psh) functions $u$, i.e. all upper semicontinuous function $u\in L^{1}$ such that $\omega+dd^{c}u\geq 0$ in the sense of $(1,1)$-currents. Here $d^{c}:=\frac{i}{4\pi}(\partial-\bar{\partial})$ so that $dd^{c}=\frac{i}{2\pi}\partial \bar{\partial}$.\newline
The maximum of two $\omega$-psh functions $u,v$ still belongs to $PSH(X,\omega)$ but $\min(u,v)$ may not be $\omega$-psh. This is one reason to introduce the $\omega$-psh function
$$
P_{\omega}(u,v):=\Big(\sup\{w\in PSH(X,\omega)\, : \, w\leq \min(u,v)\}\Big)^{*}
$$
(the star is for the upper semicontinuous regularization). It is clearly the largest $\omega$-psh function that is smaller than $u,v$. But sometimes we may want to find the largest function $w\in PSH(X,\omega)$ that is bounded above by $v\in PSH(X,\omega)$ and that is \emph{more singular} than $u\in PSH(X,\omega)$, where $w$ is more singular than $u$ if $w\leq u+C$ for a constant $C\in\mathbbm{R}$ (we denote such partial order by $\preccurlyeq$). Thus we recall the following envelope (\cite{RWN14}):
$$
P_{\omega}[u](v):=\big(\lim_{C\to +\infty}P_{\omega}(u+C,v)\big)^{*}.
$$
If now we take $v=0$ we obtain a projection map $P_{\omega}[\cdot]:=P_{\omega}[\cdot](0):PSH(X,\omega)\to PSH(X,\omega)$. The image of this map is denoted by $\mathcal{M}$ and its elements are called \emph{model type envelopes} (see \cite{RWN14}, \cite{DDNL17b}, \cite{Tru19} and references therein). It is an easy exercise to check that on $\mathcal{M}$ the two partial orders $\leq, \preccurlyeq$ coincides. The set $\mathcal{M}$ is crucial when one tries to solve complex Monge-Ampère equations with prescribed singularities (\cite{DDNL17b}, \cite{DDNL18b}), i.e. equations as
\begin{equation}
\begin{cases}
MA_{\omega}(u)=\nu\\
[u]=[\psi]
\end{cases}
\end{equation}
where $[u]$ is the equivalence class of $u\in PSH(X,\omega)$ under the partial order $\preccurlyeq$ ($[u]=[\psi]$ if and only if $u-\psi\in L^{\infty}$), $\nu$ is a non-pluripolar positive measure on $X$ and
$$
MA_{\omega}(u):=\langle (\omega +dd^{c}u)^{n}\rangle
$$
is the top non-pluripolar product of the closed and positive current $\omega+dd^{c}u$ (see \cite{BEGZ10}). We also need to recall that the \emph{total mass} of the Monge-Ampère operator respects the partial order $\preccurlyeq$ by \cite{WN17}, i.e.
$$
u\preccurlyeq v \Longrightarrow \int_{X}MA_{\omega}(u)\leq \int_{X}MA_{\omega}(v).
$$
Given $\psi\in PSH(X,\omega)$, $\mathcal{E}(X,\omega,\psi):=\{u\preccurlyeq \psi\, : \, \int_{X}MA_{\omega}(u)=\int_{X}MA_{\omega}(\psi)\}$ is the set of all $\omega$-psh functions with $\psi$\emph{-relative full mass}.\newline
Finally we underline that $PSH(X,\omega)$ is naturally endowed with a \emph{weak} topology given by the inclusion $PSH(X,\omega)\subset L^{1}$ (i.e. the $L^{1}$-topology), and that $\mathcal{M}$ is weakly closed. Moreover, setting $\mathcal{M}^{+}:=\{\psi\in\mathcal{M}\, : \, V_{\psi}>0\}$ and given a totally ordered family $\mathcal{A}:=\{\psi_{i}\}_{i\in I}\subset \mathcal{M}^{+}$, the Monge-Ampère operator produces a homeomorphism between $\overline{\mathcal{A}}$ and its image endowed with the weak topology of measures (Lemma $3.12$ in \cite{Tru20a}).
\subsection{Strong topologies}
\label{ssec:StrongTop}
The Monge-Ampère operator may not be continuous with respect to the weak topology on $PSH(X,\omega)$. Here we recall briefly a strengthened of the weak topology for some particular subsets of $PSH(X,\omega)$ which is more efficient when one wants to study complex Monge-Ampère equations. See our previous works \cite{Tru19}, \cite{Tru20a} and references therein.\newline
 
Given $\psi\in\mathcal{M}$, the sets $\mathcal{E}^{1}(X,\omega,\psi)\subset PSH(X,\omega)$ and $\mathcal{P}^{1}(X,\omega,\psi)\subset \mathcal{P}(X):=\{ \mbox{probability measures on} \,X\}$ are defined respectively as
\begin{gather*}
\mathcal{E}^{1}(X,\omega,\psi):=\{u\in \mathcal{E}(X,\omega,\psi)\, : \,E_{\psi}(u)>-\infty\},\\
\mathcal{P}^{1}(X,\omega,\psi):=\{V_{\psi}\nu\, : \, \nu\in\mathcal{P}(X)\, \,\mbox{satisfies} \,\, E_{\psi}^{*}(\mu)<+\infty\}
\end{gather*}
where $E_{\psi}, E_{\psi}^{*}$ are the $\psi$\emph{-relative energies}. More precisely,
$$
E_{\psi}(u):=\frac{1}{n+1}\sum_{j=0}^{n}\int_{X}(u-\psi)\langle (\omega+dd^{c}u)^{j}\wedge (\omega+dd^{c}\psi)^{n-j}\rangle
$$
if $u$ has $\psi$\emph{-relative minimal singularities}, i.e. $[u]=[\psi]$, and $E_{\psi}(u):=\lim_{k\to +\infty}E_{\psi}\big(\max(u,\psi-k)\big)$ otherwise. See \cite{DDNL17b}, \cite{Tru19}, \cite{Tru20a} for many of its properties, noting that the authors in \cite{DDNL17b} assumed $\psi$ with \emph{small unbounded locus}, i.e. locally bounded on the complement of a closed complete pluripolar set, while in \cite{Tru19}, \cite{Tru20a} their results have been extended to the general case.
\begin{prop}[\cite{Tru20a}, Lemma $3.13$, Propositions $3.14$, $3.15$]
\label{prop:Usc}
Let $\{\psi_{k}\}_{k\in\mathbbm{N}}\subset \mathcal{M}^{+}$ be a totally ordered sequence of model type envelopes, and let $\{u_{k}\}_{k\in\mathbbm{N}}\subset PSH(X,\omega)$ such that $u_{k}\in\mathcal{E}^{1}(X,\omega,\psi_{k})$ for any $k\in\mathbbm{N}$. If $u_{k}\to u$ weakly. Then
$$
\limsup_{k\to +\infty}E_{\psi_{k}}(u_{k})\leq E_{P_{\omega}[u]}(u).
$$
Moreover if $u_{k}\to u$ weakly and $E_{\psi_{k}}(u_{k})\geq -C$ uniformly, then $\psi_{k}\to P_{\omega}[u]$ weakly. In particular for any $C\in\mathbbm{N}, \psi\in\mathcal{M}^{+}$ the set
$$
\mathcal{E}^{1}_{C}(X,\omega,\psi):=\{u\in\mathcal{E}^{1}(X,\omega,\psi)\, : \, \sup_{X}u\leq C \, \mbox{and}\, E_{\psi}(u)\geq -C\}
$$
is weakly compact.
\end{prop}

The $\psi$-relative energy $E_{\psi}^{*}$ (\cite{Tru20a}) is instead defined as
$$
E_{\psi}^{*}(\nu):=\sup_{u\in \mathcal{E}^{1}(X,\omega,\psi)}\Big(E_{\psi}(u)-V_{\psi}L_{\nu}(u)\Big)\in [0,+\infty]
$$
where $L_{\nu}(u):=\lim_{k\to +\infty}\int_{X}\big(\max\{u,\psi-k\}-\psi\big)d\nu$ if $\nu$ does not charge $\{\psi=-\infty\}$ and as $L_{\nu}\equiv-\infty$ otherwise. We refer to \cite{Tru20a} for its properties.\newline
It is then natural to endow these sets with \emph{strong topologies} given as the coarsest refinements of the weak topologies such that the $\psi$-relative energies become continuous. The following summarized result holds.
\begin{thm}[\cite{Tru19}, Theorem $A$; \cite{Tru20a}, Theorem $A$, Proposition $5.5$]
\label{thm:Riass}
Let $\psi\in\mathcal{M}^{+}$. Then:
\begin{itemize}
\item[i)] the strong topology on $\mathcal{E}^{1}(X,\omega,\psi)$ is a metric topology given by the complete distance $d(u,v):=E_{\psi}(u)+E_{\psi}(v)-2E_{\psi}\big(P_{\omega}(u,v)\big)$;
\item[ii)] the Monge-Ampère operator $MA_{\omega}(\cdot)$ produces a homeomorphism
\begin{equation}
\label{eqn:Homeom}
MA_{\omega}:\big(\mathcal{E}_{norm}^{1}(X,\omega,\psi),d\big)\to \big(\mathcal{P}^{1}(X,\omega,\psi), strong\big)
\end{equation}
where we set $\mathcal{E}^{1}_{norm}(X,\omega,\psi):=\{u\in\mathcal{E}^{1}(X,\omega,\psi)\, : \, \sup_{X}u=0\}$;
\item[iii)] for any $V_{\psi}\nu=MA_{\omega}(u)\in \mathcal{P}^{1}(X,\omega,\psi)$ the equality $E_{\psi}^{*}(\nu)=E_{\psi}(u)-\int_{X}(u-\psi)MA_{\omega}(u)$ holds.
\end{itemize}
\end{thm}
Theorem $2.2.(i)$ was also independently proved by M. Xia in \cite{Xia19}.\newline
We recall that $E_{\psi}(\cdot)$ is continuous along decreasing sequences (Lemma $4.11$ in \cite{DDNL17b}), thus the distance $d$ satisfies the same property, i.e. decreasing sequences are strongly continuous.\newline
Furthermore it is possible to extend the strong topology of $\mathcal{E}^{1}(X,\omega,\psi)$ considering different model type envelopes. Letting $\{\psi_{k}\}_{k\in\mathbbm{N}}\subset\mathcal{M}^{+}$ totally ordered sequence and letting $\psi\in\mathcal{M}^{+}$ be one of its weak limits, we say that a sequence $\{u_{k}\}_{k\in\mathbbm{N}}$ such that $u_{k}\in\mathcal{E}^{1}(X,\omega,\psi_{k})$ \emph{converges strongly} to $u\in\mathcal{E}^{1}(X,\omega,\psi)$ if $u_{k}\to u$ weakly and $E_{\psi_{k}}(u_{k})\to E_{\psi}(u)$.
\begin{prop}[\cite{Tru20a}, Theorem $6.3$]
\label{prop:Strong}
Let $\{\psi_{k}\}_{k\in\mathbbm{N}}\subset \mathcal{M}^{+}$ be a totally ordered sequence converging weakly to $\psi\in\mathcal{M}^{+}$, and let $u_{k}\in\mathcal{E}^{1}(X,\omega,\psi_{k})$ be a sequence converging strongly to $u\in\mathcal{E}^{1}(X,\omega,\psi)$. Then there exists a subsequence $\{u_{k_{h}}\}_{h\in\mathbbm{N}}$ and two sequences $v_{h}\geq u_{k_{h}}\geq w_{h}$ of $\omega$-psh functions such that $v_{h}\searrow u$, $w_{h}\nearrow u$ almost everywhere. In particular $u_{k}\to u$ in capacity.
\end{prop}
We recall that $\{u_{k}\}_{k\in\mathbbm{N}}\subset PSH(X,\omega)$ converges in capacity to $u\in PSH(X,\omega)$ if for any $\delta>0$
$$
\mbox{Cap}\big(\{|u_{k}-u|\geq \delta\}\big)\to 0 
$$
where for any $B\subset X$ Borel set
\begin{equation*}
\mbox{Cap}(B):=\sup\Big\{\int_{B}MA_{\omega}(u)\, : \, u\in PSH(X,\omega), -1\leq u\leq 0\Big\}
\end{equation*}
(see \cite{Kol98}, \cite{GZ17} and reference therein).
\subsection{Case with analytical singularities.}
\label{ssec:Anal}
In this subsection we assume $\psi:=P_{\omega}[\varphi]\in \mathcal{M}^{+}$ where $\varphi\in PSH(X,\omega)$ has \emph{analytical singularities}, i.e. locally $\varphi_{|U}=g+c\log \big(|f_{1}|^{2}+\cdots +|f_{k}|^{2}\big)$ where $c\in\mathbbm{R}_{\geq 0}$, $g\in C^{\infty}$, and $\{f_{j}\}_{j}^{k}$ are local holomorphic functions. The coherent ideal sheaf $\mathcal{I}$ generated by these functions has integral closure globally defined, hence the singularities of $\varphi$ are formally encoded in $(\mathcal{I},c)$. We also recall that $\varphi$ has $\psi$-relative minimal singularities (see Proposition $4.36$ in \cite{DDNL17b}).\newline
It is well-known in this case that there exists a smooth resolution $p: Y\to X$ given by a sequence of blow-ups of smooth centers such that $p^{*}\mathcal{I}=\mathcal{O}_{Y}(-D)$ for an effective divisor $D$. Moreover the Siu Decomposition (\cite{Siu74}) of $p^{*}(\omega+dd^{c}\varphi)$ is given by
$$
p^{*}(\omega+dd^{c}\varphi)=\eta+c[D]
$$
where $\eta$ is a big and semipositive smooth $(1,1)$-form on $Y$. We also recall that the sets $\mathcal{E}(Y,\eta)$ and $\mathcal{E}^{1}(Y,\eta)$ are defined as in the Kähler case (see \cite{BEGZ10}) and that $\mathcal{E}^{1}(Y,\eta)$ becomes a complete metric space where endowed with the distance
$$
d(u,v):=E(u)+E(v)-2E\big(P_{\eta}(u,v)\big).
$$
The quantities $P_{\eta}(\cdot,\cdot), E(\cdot)$ are defined as in the Kähler case.
\begin{prop}[\cite{Tru20b}, Lemma $4.6$, Proposition $4.7$]
\label{prop:Anal}
The metric spaces $\big(\mathcal{E}^{1}(X,\omega,\psi),d\big)$, $\big(\mathcal{E}^{1}(Y,\eta),d\big)$ are isometric through the map $F:u\to \tilde{u}:=(u-\varphi)\circ p$, and the two energies $E_{\psi}(\cdot)$ and $E(\cdot)$ respectively on $\mathcal{E}^{1}(X,\omega,\psi)$ and on $\mathcal{E}^{1}(Y,\eta)$ satisfy $E_{\psi}(u)-E_{\psi}(\varphi)=E(\tilde{u})$. Moreover $F$ extends to a bijection $F:\{u\in PSH(X,\omega)\, : \,u\preccurlyeq \psi\}\to PSH(Y,\eta)$.
\end{prop}
\section{Some particular model type envelopes.}
\label{sec:Model}
In this section $\nu(u,x)$ will denote the \emph{Lelong number} of $u\in PSH(X,\omega)$ at $x\in X$, i.e., fixing a holomorphic chart $x\in U\subset X$,
$$
\nu(u,x):=\sup\{\gamma\geq 0\, : \, u(z)\leq \gamma \log |z-x|^{2}+O(1) \,\, \mathrm{on}\,\, U\,\}.
$$
We also recall that the \emph{multiplier ideal sheaf} $\mathcal{I}(tu)$, $t\geq 0$, of $u\in PSH(X,\omega)$ is the analytic coherent and integrally closed ideal sheaf whose germs are given by
$$
\mathcal{I}(tu,x):=\Big\{f\in \mathcal{O}_{X,x}\, \mathrm{such}\, \mathrm{that}\, \int_{V}|f|^{2}e^{-tu}\omega^{n}<+\infty \, \mathrm{for}\, \mathrm{some}\, \mathrm{open}\, \mathrm{set}\, x\in V\subset X\Big\}.
$$
From now on, we call \emph{singularity data} associated to a function $u$ the data given by all the germs of the multiplier ideal sheaves $\mathcal{I}(tu)$ for $t>0$.
\begin{prop}[\cite{Tru20b}, Proposition 3.10]
\label{prop:Lelong}
Let $u\in PSH(X,\omega)$. Then $u$ and $\psi:=P_{\omega}[u]$ have the same singularity data, i.e.
$$
\mathcal{I}(t u,x)=\mathcal{I}(t\psi,x)\, \, \mathrm{for}\, \mathrm{any} \,\, t>0, \, x\in X.
$$
\end{prop}
By Theorem $1.2.$ in \cite{DDNL17b} if $P_{\omega}[u]=\psi$ then $u\in\mathcal{E}(X,\omega,\psi)$. The reverse arrow also holds when $\psi\in\mathcal{M}^{+}$ by Theorem $1.3$ in \cite{DDNL17b}.\newline
Note that as a consequence of Theorem $A$ in \cite{BFJ08} and of the solution to the Strong Openness Conjecture (\cite{GZ14}), it follows from Proposition \ref{prop:Lelong} that $\nu(u,x)=\nu\big(P_{\omega}[u],x\big)$ for any $x\in X$ and for any $u\in PSH(X,\omega)$.
\begin{defn}
\label{defn:MD}
We define the subset $\mathcal{M}_{D}\subset \mathcal{M}$ of all model type envelopes $\psi\in\mathcal{M}$ such that $\psi\succcurlyeq \psi'$ for any $\psi'\in\mathcal{M}$ with the same singularity data of $\psi$.
\end{defn}
Observe that by Proposition \ref{prop:Anal} $\mathcal{M}_{D}$ includes any $\psi\in\mathcal{M}$ with analytical singularities type, i.e. $\psi=P_{\omega}[u]$ for $u$ with analytical singularities, as any other $\psi'\in\mathcal{M}$ with the same singularity data of $\psi$ corresponds to a $\eta$-psh function for $\eta$ as in subsection \S \ref{ssec:Anal}.\newline
Moreover, in the next subsection we will prove that for any $\psi\in\mathcal{M}$ there exists a unique $\psi'\in\mathcal{M}_{D}$ with the same singularity data of $\psi$ (see Corollary \ref{cor:SD1}).
\subsection{A regularization process}
\label{ssec:AnSng}
The following key result is a consequence of the renowned Demailly's regularization Theorem (\cite{Dem91}).
\begin{thm}
\label{thm:Reg}
Let $\psi\in \mathcal{M}$. Then there exists a decreasing sequence $\{\psi_{k}\}_{k\in\mathbbm{N}}\subset \mathcal{M}$ such that to any $u\in PSH(X,\omega)$ having the same singularity data of $\psi$ can be associated a sequence $\{u_{k}\}_{k\in\mathbbm{N}}$ with the following properties: 
\begin{itemize}
\item[i)] for any $k\in\mathbbm{N}$, $u_{k}\in\mathcal{E}(X,\omega,\psi_{k})$, $u_{k}$ has algebraic singularities and $u_{k}$ has $\psi_{k}$-relative minimal singularities;
\item[ii)] $u_{k}$ converges to $u$ in capacity.
\end{itemize}
If $|u_{1}-u_{2}|$ is bounded over $X$ then $|u_{1,k}-u_{2,k}|$ is uniformly bounded over $X$.\newline
Furthermore, $\psi_{k}\searrow \tilde{\psi}\in\mathcal{M}_{D}$ where $\tilde{\psi}$ has the same singularity data of $\psi$.
\end{thm}
A function $u\in PSH(X,\omega)$ with analytic singularities formally encoded in $(\mathcal{I},c)$ is said to have \emph{algebraic} singularities when $c\in\mathbbm{Q}$.
\begin{proof}
\textbf{Step 1: sketch of the construction of a sequence of quasi-plurisubharmonic functions with logarithmic poles (Proposition $3.7$ in \cite{Dem91}).}\newline
For $\{W_{\nu}\}_{\nu\in \Lambda}$ fixed finite covering of open coordinate sets, it is possible to choose a finite open covering $\{\Omega_{j}\}_{j\in J}$ of coordinate balls of radius $2\delta$ (if $\delta$ is small enough) such that any $\Omega_{j}$ is contained in at least one $W_{\nu}$ and such that the set of all coordinate balls of radius $\delta$ produces another open covering $\{\Omega'_{j}\}_{j\in J}$. Fix also $\epsilon(\delta)$ be a modulus of continuity related to $\{\Omega_{j}\}_{j}$ for $\omega$ such that $\epsilon(\delta)\to 0$ for $\delta\to 0$ and such that $|\omega_{x'}-\omega_{x}|\leq \epsilon(\delta)\omega_{x}/2$ for all $x,x'\in \Omega_{j}$. Then, it follows that $0\leq -\omega-\tau_{j}^{*}\gamma_{j}\leq 2\epsilon(\delta)\omega$ on $\Omega_{j}$ where $\gamma_{j}$ is a $(1,1)$-form with constant coefficients on $\tau_{j}(\Omega_{j})=B_{2\delta}(a_{j})$ such that $-\omega-\epsilon(\delta)\omega=\tau_{j}^{*}\gamma_{j}$ at $\tau_{j}^{-1}(a_{j})$. We denote by $\tilde{\gamma}_{j}$ the homogeneous quadratic function in $z-a_{j}$ such that $dd^{c}\tilde{\gamma}_{j}=\gamma_{j}$. Thus for any $j\in J$, $m\in\mathbbm{N}$ and $\phi\in PSH(X,\omega)$ we define locally on $\Omega_{j}$
$$
\hat{\varphi}_{j,m}:=\frac{1}{m}\log\Big(\sum_{l}|\sigma_{j,m,l}|^{2}\Big)
$$
where $\{\sigma_{j,m,l}\}_{l\in\mathbbm{N}}$ is an orthonormal base of the Hilbert space $\mathcal{H}_{\Omega_{j}}(m\tilde{\varphi}_{j}):=\{f\in\mathcal{O}_{\Omega_{j}}(\Omega_{j})\,:\, ||f||_{m\tilde{\varphi}_{j},\Omega_{j}}^{2}:=\int_{\Omega_{j}}|f|^{2}e^{-m\tilde{\varphi}_{j}}<+\infty\}$ for $\tilde{\varphi}_{j}:=\phi-\tilde{\gamma}_{j}\circ \tau_{j}$ (which is plurisubharmonic on $\Omega_{j}$).\newline
Next, the functions $\hat{\varphi}_{j,m}$ are glued together thanks to a partition of unity technique based on Lemma $3.5$ in \cite{Dem91}. Namely, defining $w_{j,m}'(x):=m\hat{\varphi}_{j,m}(x)+m\tilde{\gamma}_{j}(z)+\frac{\sqrt{m}}{2}(\delta^{2}-\lVert z\rVert^{2})$ where $z=\tau_{j}(x)$ are coordinates on $\Omega_{j}$, the glued function is given as
$$
\varphi'_{m}:=\frac{1}{m}\log \Big(\sum_{j}\theta_{j}e^{w_{j,m}'}\Big)
$$
where $\{\theta_{j}\}_{j}$ are smooth nonnegative functions with support in $\Omega_{j}$ such that $\theta_{j}\leq 1$ on $\Omega_{j}$ while $\theta_{j}=1$ on $\Omega_{j}'$ ($\theta$ satisfies some other extra conditions as described in Lemma $3.5$ in \cite{Dem91}). Then taking $\delta=\delta_{m}$ that goes very slowly to $0$ as $m\to +\infty$, the sequence $\varphi'_{m}$ satisfies $\omega+dd^{c}\varphi'_{m}\geq -\epsilon_{m}\omega$ for $\epsilon_{m}\searrow 0$ and, for any $m\in\mathbbm{N}$, $\varphi'_{m}$ has \emph{logarithmic poles} along $\big(\mathcal{I}(m\phi),\frac{1}{m}\big)$, i.e. locally $\varphi'_{m|U}=\frac{1}{m}\log \big(\lvert f_{1}\rvert^{2}+\cdots+\lvert f_{N}\rvert^{2}\big) +g$ where $g$ is a bounded function and $\{f_{j}\}_{j=1}^{N}$ are local holomorphic functions generating $\mathcal{I}(m\phi)$.\newline
\textbf{Step 2: construction of a decreasing sequence of quasi-plurisubharmonic functions with algebraic singularities.}\newline
Coming back to the construction of the previous Step, as proved in Theorem $2.2.1.(Step\, 3)$ in \cite{DPS00}, we also have
$$
\hat{\varphi}_{j,m_{1}+m_{2}}\leq \frac{A_{1}}{m_{1}+m_{2}}+\frac{m_{1}}{m_{1}+m_{2}}\hat{\varphi}_{j,m_{1}}+\frac{m_{2}}{m_{1}+m_{2}}\hat{\varphi}_{j,m_{2}}
$$
for any $m_{1},m_{2}\in\mathbbm{N}$ where $A_{1}$ depends only on $n=\dim X$. Therefore in the gluing process described before, considering $\varphi_{j,k}:=\hat{\varphi}_{j,2^{k}}+\frac{A_{1}}{2^{k}}$ instead of $\hat{\varphi}_{j,2^{k}}$ we get a sequence of glued functions
$$
\varphi_{k}=\varphi_{2^{k}}'+\frac{A_{1}}{2^{k}}
$$
which is now \emph{decreasing}, has logarithmic poles along $\big(\mathcal{I}(2^{k}\phi),\frac{1}{2^{k}}\big)$ and satisfies $\omega+dd^{c}\varphi_{k}\geq -\epsilon_{k}\omega$ for $\epsilon_{k}\searrow 0$ (by abuse of notation we denoted by $\epsilon_{k}$ the sequence $\epsilon_{2^{k}}$ of the previous Step). Indeed the unique thing to observe is that the hypothesis to apply the key aforementioned Lemma $3.5$ in \cite{Dem91} are trivially satisfied by the family $w_{j,2^{k}}'+A_{1}$ which gives the glued functions $\varphi_{k}$.\newline
Then by a regularization argument of Richberg (\cite{Ric68}, see also Lemma $2.15$ in \cite{Dem91}) we approximate $\varphi_{k}$ with a smooth quasi-psh function $\phi_{k}$ on $X\setminus V\big(\mathcal{I}(2^{k}\phi)\big)$ such that $|\phi_{k}-\varphi_{k}|\leq 1/k$ and such that it extends to a quasi-psh function on $X$ with
$$
\omega+dd^{c}\phi_{k}\geq -2\epsilon_{k}\omega.
$$
Thus, since $\phi_{k}$ has the same singularity of $\varphi_{k}$, we get that $\phi_{k}$ has algebraic singularities formally encoded in $\big(\mathcal{I}(2^{k}\phi),\frac{1}{2^{k}}\big)$.\newline
\textbf{Step 3: obtaining the sequences $\{\psi_{k}\}_{k\in\mathbbm{N}}\subset \mathcal{M}, \{u_{k}\}_{k\in\mathbbm{N}}$ and proving (i), (ii).}\newline
Next assuming $u$ such that $\psi=P_{\omega}[u]$, we apply the regularization just described in Step $1$ and $2$ to the $\omega$-psh function $\tilde{u}:=u-\sup_{X}u-1$, obtaining a sequence $\tilde{u}_{k}$. Then we define
$$
u_{k}:=\frac{1}{1+2\epsilon_{k}}\tilde{u}_{k}+\sup_{X}u+1.
$$
By construction $u_{k}\in PSH(X,\omega)$, $u_{k}$ has algebraic singularities assuming without loss of generality that $\{\epsilon_{k}\}_{k\in\mathbbm{N}}\in\mathbbm{Q}$ and, as a consequence of Proposition \ref{prop:Lelong}, the singularity type $[u_{k}]$ is constant varying $u$ which satisfies $P_{\omega}[u]=\psi$. Therefore defining $\psi_{k}:=P_{\omega}[u_{k}]$ the first point follows.\newline
About the convergence in capacity, clearly we may assume $\sup_{X}u=-1$. Then we denote by $v_{k}\in\mathcal{E}(X,\omega,\psi)$ the decreasing sequence of quasi-psh function with logarithmic poles converging to $u$ obtained by the process described above (i.e. the $\varphi_{k}$'s of before). By Hartogs' Lemma (see Proposition $8.4$ in \cite{GZ17}) $\sup_{X}v_{k}\to -1$ and it is immediate to check that $\frac{v_{k}}{1+2\epsilon_{k}}$ becomes a decreasing sequence converging to $u$ when $\sup_{X}v_{k}\leq 0$. Thus we get that $\frac{v_{k}}{1+2\epsilon_{k}}\to u$ in capacity. Next we note that for any $\delta>0$
$$
\Big\{|u_{k}-u|\geq \delta\Big\}\subset \Big\{\Big|\frac{v_{k}}{1+2\epsilon_{k}}-u\Big|\geq \delta-\frac{1}{k(1+2\epsilon_{k})}\Big\}
$$
as $|\tilde{u}_{k}-v_{k}|\leq 1/k$ by construction, recalling that $u_{k}:=\frac{\tilde{u}_{k}}{1+2\epsilon_{k}}$. Hence taking $k=k_{\delta}\gg 0$ big enough we get that
$$
\Big\{|u_{k}-u|\geq \delta\Big\}\subset \Big\{|\frac{v_{k}}{1+2\epsilon_{k}}-u|\geq \frac{\delta}{2}\Big\},
$$
which implies that $u_{k}\to u$ in capacity.\newline
Assuming $|u_{1}-u_{2}|\leq C$, to prove that $u_{1,k}-u_{2,k}$ is uniformly bounded it is clearly enough to check that $|v_{1,k}-v_{2,k}|$ is uniformly bounded, where as before we denote by $v_{i,k}$ the sequence of quasi-psh function with logarithmic poles which decreases to $u_{i}$ for $i=1,2$ (i.e. in the process described above we replace $\phi,\varphi_{k},\hat{\varphi}_{j,m},\tilde{\varphi}_{j},\varphi_{j,k}$ respectively with $u_{i},v_{i,k},\hat{v}_{i,j,m},\tilde{v}_{i,j}, v_{i,j,k}$). Thus if $u_{1}\leq u_{2}+C$ and assuming without loss of generality that $\sup_{X}u_{1}=\sup_{X}u_{2}=-1$ then $\hat{v}_{1,j,2^{k}}\leq \hat{v}_{2,j,2^{k}}+C$ for any $j\in J$ and any $k\in\mathbbm{N}$ since
$$
\hat{v}_{1,j,2^{k}}=\sup_{f\in B(1)}\frac{1}{2^{k}}\log |f|^{2}
$$
where $B(1)$ is the unit ball in $\mathcal{H}_{\Omega_{j}}(2^{k}\tilde{v}_{1,j})$ and similarly for $u_{2}$. Hence we get that $|v_{1,j,k}-v_{2,j,k}|$ is uniformly bounded in $j,k$ and by the gluing process described in Step 1 it follows that also $|v_{1,k}-v_{2,k}|$ is uniformly bounded in $k$.\newline
\textbf{Step 4: proving that $\psi_{k}\searrow \tilde{\psi}\in\mathcal{M}_{D}$ having the same singularity data of $\psi$.} \newline
By construction $\psi_{k}\succcurlyeq\psi_{k+1}$, which is equivalent to $\psi_{k}\geq \psi_{k+1}$. Thus $\tilde{\psi}:=\lim_{k\to +\infty}\psi_{k}\in\mathcal{M}$ and $\tilde{\psi}\geq \psi$ as immediate consequence of $\psi_{k}\geq \psi$ for any $k\in\mathbbm{N}$. In particular $\mathcal{I}(t\psi)\subset \mathcal{I}(t\tilde{\psi})$ for any $t>0$. Conversely we set $u:=\psi-1$ and, for fixed $t>0$, we claim that
\begin{equation}
\label{eqn:L2}
\mathcal{I}(t\psi)\supset \mathcal{I}\big((1+\tau_{k})t\tilde{u}_{k}\big)=\mathcal{I}\big((1+\tau_{k})(1+2\epsilon_{k})t u_{k}\big)
\end{equation}
for $\tau_{k}= \frac{t}{2^{k}-t}$ if $k\gg 1$ such that $2^{k}>t$, where $\tilde{u}_{k}=(1+2\epsilon_{k})u_{k}$ is the quasi-psh function with analytic singularities formally encoded in $(\mathcal{I}(2^{k}\psi),\frac{1}{2^{k}})$ constructed in Step 1. The inclusion in (\ref{eqn:L2}) would imply 
$$
\mathcal{I}(t\psi)\supset \mathcal{I}\big((1+\tau_{k})(1+2\epsilon_{k})t\tilde{\psi}\big) 
$$
as $u_{k}\succcurlyeq \tilde{\psi}$ for any $k\in\mathbbm{N}$, and the resolution of the strong openness conjecture (see \cite{GZ14}) would yield the remaining inclusion 
$$
\mathcal{I}(t\psi)\supset \mathcal{I}\big(t\tilde{\psi}\big)
$$
letting $k\to +\infty$.\newline
Let us prove (\ref{eqn:L2}). For any $U\subset X$ open set and for any holomorphic function $f$ over $U$, as $(2^{k}-t)u-2^{k}\tilde{u}_{k}<0$ over $\{u<(1+\tau_{k})\tilde{u}_{k}\}$ by the choice of $\tau_{k}=\frac{t}{2^{k}-t}$, we get
\begin{multline*}
\int_{U}|f|^{2}e^{-tu}\omega^{n}=\int_{U\cap\{u\geq (1+\tau_{k})\tilde{u}_{k}\}}|f|^{2}e^{-tu}\omega^{n}+\int_{U\cap \{u<(1+\tau_{k})\tilde{u}_{k}\}}|f|^{2}e^{2^{k}(\tilde{u}_{k}-u)}e^{(2^{k}-t)u-2^{k}\tilde{u}_{k}}\omega^{n}\leq\\
\leq\int_{U}|f|^{2}e^{-(1+\tau_{k})t\tilde{u}_{k}}\omega^{n}+\int_{U}|f|^{2}e^{2^{k}(\tilde{u}_{k}-u)}\omega^{n}.
\end{multline*}
Moreover, as $\tilde{u}_{k}$ has analytic singularities formally encoded in $\big(\mathcal{I}(2^{k}u),\frac{1}{2^{k}}\big)$, $\int_{U}|f|^{2}e^{2^{k}(\tilde{u}_{k}-u)}<+\infty$ up to restrict the open set $U$, and the inclusion in (\ref{eqn:L2}) follows.\newline
Finally, as by construction $\tilde{\psi}\geq \psi'$ for any $\psi'\in\mathcal{M}$ with the same singularity data of $\psi$ (simply switching $\psi$ and $\psi'$), $\tilde{\psi}$ is a maximal element in $\mathcal{M}$ for fixed singularity data, i.e. $\tilde{\psi}\in\mathcal{M}_{D}$, which concludes the proof.
\end{proof}
We say that $\psi\in\mathcal{M}$ has \emph{analytic (resp. algebraic) singularities type} if $\psi=P_{\omega}[\varphi]$ for $\varphi\in PSH(X,\omega)$ with analytic (resp. algebraic) singularities.
\begin{cor}
\label{cor:SD1}
For any $\psi\in\mathcal{M}$ there exists a unique $\psi'\in\mathcal{M}_{D}$ having the same singularity data of $\psi$. Moreover if $\psi_{1},\psi_{2}\in\mathcal{M}_{D}$ and the singularity data of $\psi_{1}$ are \emph{worse} than the singularity data of $\psi_{2}$ (i.e. $\mathcal{I}(t\psi_{1},x)\subset \mathcal{I}(t\psi_{2},x)$ for any $x\in X$, $t>0$), then $\psi_{1}\preccurlyeq \psi_{2}$.
\end{cor}
\begin{proof}
The first statement is a trivial consequence of Theorem \ref{thm:Reg}.\newline
Next if $\psi_{1,k},\psi_{2,k}\in\mathcal{M}_{D}$ are the sequences with algebraic singularities type converging respectively to $\psi_{1},\psi_{2}$ given by Theorem \ref{thm:Reg} with respect to the same Demailly's regularization, then we have $\psi_{1,k}\leq \psi_{2,k}$ if the singularity data of $\psi_{1}$ are worse than the singularity data of $\psi_{2}$, and letting $k\to +\infty$ concludes the proof.
\end{proof}
Theorem \ref{thm:Reg} implies that the elements in $\mathcal{M}_{D}$ can be approximated by a decreasing sequence of model type envelopes with algebraic singularities type. This property defines the set $\mathcal{M}_{D}$ as immediate consequence of the following result.
\begin{prop}
\label{prop:MDPlus}
Let $\{\psi_{k}\}_{k\in\mathbbm{N}}\subset\mathcal{M}_{D}$ be a decreasing sequence converging to $\psi$. Then $\psi\in\mathcal{M}_{D}$.
\end{prop}
\begin{proof}
Let $\psi'\in\mathcal{M}_{D}$ having the same singularity data of $\psi$. Then by Corollary \ref{cor:SD1} $\psi_{k}\succcurlyeq \psi'$ for any $k\in\mathbbm{N}$, which is equivalent to $\psi_{k}\geq \psi'$ as we are considering model type envelopes. Hence $\psi\geq \psi'$, which implies $\psi=\psi'$ by definition of $\mathcal{M}_{D}$ and concludes the proof.
\end{proof}
The following example shows that $\mathcal{M}_{D}$ is a proper subset of $\mathcal{M}$.
\begin{esem}
\label{esem:StrangePiu}
\emph{Let $K\subset \mathbbm{P}^{1}$ be a polar Cantor set, $\omega=\omega_{FS}$ be the Fubini-Study metric on $\mathbbm{P}^{1}$, and $\mu_{K}$ be the measure on $\mathbbm{C}$ associated to $K$. Then the potential $u(z):=\int_{\mathbbm{C}}\log|z-w|d\mu_{K}(w)$ is a subharmonic function on $\mathbbm{C}$, harmonic on $\mathbbm{C}\setminus \mathrm{Supp}(\mu_{K})=\mathbbm{C}\setminus K$ and $u(z)=\mu_{K}(\mathbbm{C})\log |z|+O(|z|^{-1})$ as $z\to +\infty$ (see Theorem $3.1.2$ in \cite{Rans}). Thus, up to rescaling the Fubini-Study metric, $u$ extends to an $\omega_{FS}$-psh function, i.e. $u\in PSH(\mathbbm{P}^{1},\omega_{FS})$. Moreover as $\mu_{K}$ has no atoms, $\nu(u,z)=0$ for any $z\in\mathbbm{P}^{1}$, which by Skoda's Integrability Theorem (\cite{Sko72}, see also Theorem \ref{thm:Skoda} below) implies that $u$ has trivial singularity data. Therefore by Proposition \ref{prop:Lelong}, the family of model type envelopes $\{\psi_{t}:=P_{\omega}[tu]\}_{t\in [0,1]}\subset\mathcal{M}$ has constant singularity data, but $V_{\psi_{t}}=\int_{X}MA_{\omega}(tu)=(1-t)\int_{X}\omega+t\int_{X}MA_{\omega}(u)= (1-t)\int_{X}\omega$ since $MA_{\omega}(u)$ is concentrated on $K$ which is polar. Hence clearly $\{\psi_{t}\}_{t\in(0,1]}\subset\mathcal{M}\setminus \mathcal{M}_{D}$.}
\end{esem}
Finally it is important to observe that any element in $\mathcal{M}_{D}^{+}$ can be connected to $0$ by a natural path.
\begin{prop}
\label{prop:Resc}
Let $\psi\in\mathcal{M}_{D}^{+}$ and $t\in[0,1]$. Then $\psi_{t}:=P_{\omega}[t\psi]\in\mathcal{M}_{D}^{+}$.
\end{prop}
\begin{proof}
As $\psi\in\mathcal{M}_{D}$, by Theorem \ref{thm:Reg} there exists a decreasing sequence $\{\psi_{k}=P_{\omega}[\varphi_{k}]\}_{k\in\mathbbm{N}}\subset \mathcal{M}_{D}^{+}$ of model type envelopes with algebraic singularity types converging to $\psi$. We denoted by $\varphi_{k}$ a choice of $\omega$-psh functions with algebraic singularities. Then, for any $t\in[0,1)$ the sequence $\{\psi_{k,t}:=P_{\omega}[t\varphi_{k}]\}$ is clearly a decreasing sequence with analytic singularity type, which implies that $\psi'_{t}:=\lim_{k\to +\infty}\psi_{k,t}\in\mathcal{M}_{D}$ by Proposition \ref{prop:MDPlus}. Moreover since $\varphi_{k}$ has $\psi_{k}$-relative minimal singularities we have $\psi_{k,t}=P_{\omega}[t\psi_{k}]$, and
$$
V_{\psi_{k,t}}=\int_{X}MA_{\omega}\big(t\psi_{k}\big)=\sum_{j=0}^{n}t^{n-j}(1-t)^{j}\int_{X}\langle \omega^{j}\wedge (\omega+dd^{c}\psi_{k})^{n-j}\rangle.
$$
Observe also that $V_{\psi_{k}}\searrow V_{\psi}$ by what said in section \S \ref{sec:Prelim} since $\psi_{k}$ are model type envelopes decreasing to $\psi$. Indeed, more generally we have
$$
\int_{X}\langle \omega^{j}\wedge (\omega+dd^{c}\psi_{k})^{n-j}\rangle\to \int_{X}\langle \omega^{j}\wedge (\omega+dd^{c}\psi)^{n-j}\rangle
$$
for any $j=0,\dots,n$ by Proposition $4.8$ in \cite{DDNL19} since we are assuming $V_{\psi}>0$. Hence $V_{\psi_{k,t}}\to V_{\psi_{t}'}=V_{\psi_{t}}>0$, which implies that $\psi_{t}=\psi_{t}'$ by Theorem $1.3.$ in \cite{DDNL17b}, as by construction $\psi_{t}$ is more singular than $\psi_{t}'$, i.e. $\psi_{t}\in\mathcal{M}_{D}^{+}$ for any $t\in[0,1]$.
\end{proof}
\subsection{Geodesic segments.}
\begin{defn}
Let $S:=\{t\in\mathbbm{C}\, : \, 0<\mbox{Re} \,t<1\}$ be the open strip and let $u_{0},u_{1}\in PSH(X,\omega)$. The elements $U'\in PSH(X\times S,\pi_{X}^{*}\omega)$ such that $\limsup_{t\to 0^{+}}U'(\cdot,t)\leq u_{0}$ and $\limsup_{t\to 1^{-}}U'(\cdot,t)\leq u_{1}$ are called \emph{weak subgeodesics of} $u_{0},u_{1}$, and if there exists at least one of these subgeodesics then the $\pi_{X}^{*}\omega$-psh function
$$
u_{t}(p):=U(p,t):=\Big(\sup\{U'\in PSH(X\times S, \pi_{X}^{*}\omega)\, : \, U'\, \mbox{subgeodesic of}\, u_{0},u_{1}\}\Big)^{*}
$$
is called \mbox{weak geodesic joining} $u_{0},u_{1}$.
\end{defn}
The next Proposition explores the properties of weak geodesics segments joining potentials in $\mathcal{E}^{1}(X,\omega,\psi)$ for $\psi\in\mathcal{M}$. We denote by $\mathcal{H}_{\omega}:=\{u\in PSH(X,\omega)\, : \, \omega+dd^{c}u\, \mbox{is Kähler}\}$ the set of Kähler potentials.
\begin{prop}
\label{prop:WGeod}
Let $u_{0},u_{1}\in \mathcal{E}^{1}(X,\omega,\psi)$ for $\psi\in\mathcal{M}^{+}$. Then the followings holds:
\begin{itemize}
\item[i)] there exists the weak geodesic $u_{t}(p)=U(t,p)\in PSH(X\times S, \pi_{X}^{*}\omega)$ joining $u_{0},u_{1}$, and it only depends on $\mbox{Re}\, t$ in the $t$-variable;
\item[ii)] $u_{t}\in\mathcal{E}^{1}(X,\omega,\psi)$ for any $t\in[0,1]$;
\item[iii)] letting $\{u_{0}^{k}\}_{k\in\mathbbm{N}}, \{u_{1}^{k}\}_{k\in\mathbbm{N}}\subset \mathcal{H}_{\omega}$ be decreasing sequences such that $u_{0}^{k}\searrow u_{0}, u_{1}^{k}\searrow u_{1}$ and letting $u_{t}^{k}$ the weak geodesic joining $u_{0}^{k}, u_{1}^{k}$, the convergence $u_{t}^{k}\searrow u_{t}$ holds.
\end{itemize}
Moreover if $u_{0},u_{1}$ have $\psi$-relative minimal singularities, then
\begin{itemize}
\item[iv)] $u_{t}\to u_{0}$, $u_{t}\to u_{1}$ in capacity;
\item[v)] $u_{t}$ has $\psi$-relative minimal singularities for any $t\in S$;
\item[vi)] $|u_{t}-u_{s}|\leq C|\mbox{Re}\, t-\mbox{Re}\, s|$ for any $t,s\in S$ where $C:=||u_{1}-u_{0}||_{\infty}$.
\end{itemize}
\end{prop}
The existence of the approximations of $(iii)$ is contained in \cite{BK07} while the existence of the weak geodesics joining elements in $\mathcal{H}_{\omega}$ is shown in \cite{Chen00}.  
\begin{proof}
By Proposition $2.13$ in \cite{Tru19} $P_{\omega}(u_{0},u_{1})\in \mathcal{E}^{1}(X,\omega,\psi)$, thus $(i)$ and $(iii)$ follows directly from Theorem $5.(i)$ in \cite{Dar15b}. Then by the $\mbox{Re}\, t$-convexity we obtain $P_{\omega}(u_{0},u_{1})\leq u_{t}\leq (\mbox{Re}\, t)u_{1}+(1-\mbox{Re}\,t)u_{0}$, hence $(ii)$ is given by the monotonicity of $E_{\psi}$ (Proposition $2.11$ in \cite{Tru19}).\newline
Next, assuming $u_{0},u_{1}$ with $\psi$-relative minimal singularities, $(iv)$ is a consequence of the second part of Theorem $5$ in \cite{Dar15b} as by definition it is immediate to check that $P_{\omega}[u_{0}](u_{1})=u_{1}$. Then, letting $C:=||u_{0}-u_{1}||_{L^{\infty}}$, from
$$
\max\big\{u_{0}-C\mbox{Re}\, t,u_{1}+C(\mbox{Re}\, t-1)\big\}\leq u_{t}\leq (\mbox{Re}\,t)u_{1}+(1-\mbox{Re}\,t)u_{0}
$$
we obtain that $||u_{0}-u_{t}||_{L^{\infty}}\leq C$, which implies that $u_{t}$ has uniformly bounded $\psi$-relative minimal singularities and in particular $(v)$ follows. Moreover $(i)$ and the $\mbox{Re}\,t$-convexity of $u_{t}$ yield that the $t$-derivative of $u_{t}$ is increasing. Thus, the inequality
$$
\max\big\{-C\mbox{Re}\, t,u_{1}-u_{0}+C(\mbox{Re}\, t-1)\big\}\leq u_{t}-u_{0}
$$ 
implies that the one-side derivative at $0$ of $u_{t}$ lie between $-C$ and $C$. Similarly for the one-side derivative at $1$ of $u_{t}$. Hence it follows that all the $t$-derivatives are bounded between $-C,C$, which gives $(vi)$ concluding the proof.
\end{proof}
As the weak geodesic joining two elements $u_{0},u_{1}\in \mathcal{E}^{1}(X,\omega,\psi)$ depends only on $\mbox{Re}\, t$ in the $t$-variable, the associated weak geodesic \emph{segment} will be the path $[0,1]\ni t\to u_{t}$.\newline
When $\psi\in\mathcal{M}^{+}$ has algebraic singularities type it is possible to relate the weak geodesics in $\mathcal{E}^{1}(X,\omega,\psi)$ in terms of the weak geodesics in $\mathcal{E}^{1}(Y,\eta)$.
\begin{prop}
\label{prop:WGAS}
With the same notation of subsection \ref{ssec:Anal}, let $u_{0},u_{1}\in \mathcal{E}^{1}(X,\omega,\psi)$ and let $\tilde{u}_{0}=(u_{0}-\varphi)\circ p,\tilde{u}_{1}=(u_{1}-\varphi)\circ p\in \mathcal{E}^{1}(Y,\eta)$. Then the weak geodesic $U$ joining $u_{0},u_{1}$ is given by
$$
U=(p\times Id)_{*}\tilde{U}+\Phi
$$
where $\tilde{U}$ is the weak geodesic joining $\tilde{u}_{0}, \tilde{u}_{1}$ and $\Phi\in PSH(X\times S,\pi_{X}^{*}\omega)$ is the constant weak geodesic at $\varphi$, i.e. $\Phi(\cdot,t)=\varphi(\cdot)$ for any $t\in S$
\end{prop}
\begin{proof}
By Proposition \ref{prop:WGeod} there exists the geodesic $U$ joining $u_{0},u_{1}$ and $u_{t}\in\mathcal{E}^{1}(X,\omega,\psi) $ for any $t$. Then by Proposition \ref{prop:Anal} the function
$$
\tilde{U}:=(U-\Phi)\circ (p\times Id)
$$
satisfies $\tilde{u}_{t}\in \mathcal{E}^{1}(Y,\omega)$ for any $t\in S$. Moreover it depends only on $\mbox{Re}\, t$ on the $t$-variable, and it is not difficult to check that it is upper semicontinuous and regular enough to consider the $(1,1)$-current $\pi_{Y}^{*}\eta+dd^{c}_{w,t}\tilde{U}$, which satisfies
\begin{equation}
\label{eqn:PSHGeod}
(p\times Id)^{*}\big(\pi_{X}^{*}\omega+dd^{c}_{z,t}U\big)=\pi_{Y}^{*}\eta+dd^{c}_{w,t}\tilde{U}+c\pi_{Y}^{*}[D].
\end{equation}
Therefore, as $\pi_{Y}^{*}\eta+dd_{w,t}\tilde{U}$ is positive on each fiber and $U\in PSH(X\times S, \pi_{X}^{*}\omega)$, from (\ref{eqn:PSHGeod}) we deduce that $\pi_{Y}^{*}\eta+dd_{w,t}\tilde{U}\geq 0$, i.e. that $\tilde{U}$ is a weak subgeodesic joining $\tilde{u}_{0}, \tilde{u}_{1}$.\newline
On the other hand, letting $\tilde{V}\in PSH(Y\times S, \pi_{Y}^{*}\eta)$ be the weak geodesic joining $\tilde{u}_{0}, \tilde{u}_{1}$, we obtain that
$$
V:=(p\times Id)_{*} \tilde{V}+\Phi
$$
is a weak subgeodesic joining $u_{0}, u_{1}$ from the equality (\ref{eqn:PSHGeod}) for $V,\tilde{V}$. Moreover $V\geq U$ by construction as $\tilde{V}$ is the weak geodesic, which clearly implies $V=U$. Hence $\tilde{V}=\tilde{U}$, i.e. $\tilde{U}$ is the weak geodesic joining $\tilde{u}_{0}, \tilde{u}_{1}$ and the proof is concluded.
\end{proof}
The reason of considering $\psi\in\mathcal{M}_{D}^{+}$ in Theorems \ref{thmB}, \ref{thmC} is because we can prove that the space $\big(\mathcal{E}^{1}(X,\omega,\psi),d\big)$ is \emph{geodesic}, that any weak geodesic is a metric geodesic, and that the $\psi$-relative energy becomes linear along these specific geodesics.
\begin{thm}
\label{thm:Linear}
Let $\psi\in \mathcal{M}_{D}^{+}$ and let $U$ be the weak geodesic joining $u_{0},u_{1}\in \mathcal{E}^{1}(X,\omega,\psi)$. Then $E_{\psi}$ is linear along $[0,1]\ni t\to u_{t}:=U(t,\cdot)\in\mathcal{E}^{1}(X,\omega,\psi)$, which is also a geodesic segment in $\big(\mathcal{E}^{1}(X,\omega,\psi),d\big)$, i.e.
$$
d(u_{t},u_{s})=|t-s|d(u_{0},u_{1}).
$$
\end{thm}
\begin{proof}
We set $u_{0,k}:=\max(u_{0},\psi-k),u_{1,k}:=\max(u_{1},\psi-k)$ observing that by construction the sequence of weak geodesic segments $U_{k}$ joining $u_{0,k},u_{1,k}$ decreases to $U$ (see Proposition \ref{prop:WGeod}). In particular since the $\psi$-relative energy $E_{\psi}$ and the distance $d$ are continuous along decreasing sequences in $\mathcal{E}^{1}(X,\omega,\psi)$ we may assume that $u_{0}, u_{1}$ have $\psi$-relative minimal singularities.\newline
Moreover if $\psi=P_{\omega}[\varphi]$ for $\varphi$ with analytical singularities, then the results required follow combining Proposition \ref{prop:Anal} with Proposition \ref{prop:WGAS}. Indeed, keeping the same notation of subsection \ref{ssec:Anal}, by Theorem $3.12$ in \cite{DDNL17a} the energy $E(\cdot)$ for $\mathcal{E}^{1}(Y,\eta)$ is linear along weak geodesic segments, which are metric geodesics in $\big(\mathcal{E}^{1}(X,\eta),d\big)$ by Proposition $3.13$ in \cite{DDNL18a}.\newline
For general $\psi\in\mathcal{M}_{D}^{+}$, by Theorem \ref{thm:Reg} there exist $\{\psi_{k}\}\in\mathcal{M}_{D}^{+}$, decreasing sequence converging to $\psi$ of model type envelopes with algebraic singularity types, and $u_{0}^{k},u_{1}^{k}\in\mathcal{E}^{1}(X,\omega,\psi_{k})$, decreasing sequences that converge respectively to $u_{0},u_{1}$. Observe that $||u_{0}^{k}-u_{1}^{k}||_{L^{\infty}}$ is uniformly bounded since we are assuming $u_{0},u_{1}$ with $\psi$-relative minimal singularities (Theorem \ref{thm:Reg}). Moreover by the first part of the proof the weak geodesic segment $[0,1]\ni t\to u_{t}^{k}\in\mathcal{E}^{1}(X,\omega,\psi_{k})$ joining $u_{0}^{k}, u_{1}^{k}$ is a metric geodesic in $\big(\mathcal{E}^{1}(X,\omega,\psi_{k}),d\big)$ and $E_{\psi_{k}}$ is linear along it. Furthermore for any $t,s\in [0,1]$, $||u_{t}^{k}-u_{s}^{k}||_{L^{\infty}}\leq C$ for a uniform constant $C$ and $u_{t}^{k}$ decreases to $u_{t}$ by Proposition \ref{prop:WGeod} as $k\to +\infty$. Hence the results required follow from the convergences
\begin{gather*}
d(u_{t}^{k},u_{s}^{k})\to d(u_{t},u_{s}),\\
E_{\psi_{k}}(u_{t}^{k})\to E_{\psi}(u_{t})
\end{gather*}
as $k\to +\infty$, given by Proposition $2.7$ in \cite{Tru19}.
\end{proof}
\section{$[\psi]$-KE metrics with prescribed singularities.}
\label{sec:KE}
From now on we will assume $\{\omega\}=c_{1}(X)$, i.e. $X$ is a Fano manifold and $\omega$ is a Kähler form in the anticanonical class.\newline
With $\mathcal{M}_{klt}$ we denote the set of all the model type envelopes $\psi$ such that $(X,\psi)$ is \emph{klt}, i.e. as said in the Introduction $\mathcal{I}(\psi)=\mathcal{O}_{X}$. Note that by the resolution of strong the openness conjecture (\cite{GZ14}) $\psi\in\mathcal{M}_{klt}$ if and only if there exists $p>1$ such that $e^{-\psi}\in L^{p}$. Moreover for a pair $(X,\psi)$ being klt is independent on the Kähler form chosen, indeed it holds for quasi-psh functions. We will also use the notation $\mathcal{M}_{klt}^{+}:=\mathcal{M}_{klt}\cap \mathcal{M}^{+}$ and similarly for $\mathcal{M}_{D}^{+},$  $\mathcal{M}_{D,klt}^{+}$.
\begin{prop}
\label{lem:Star}
For any $\psi\in\mathcal{M}^{+}$ the path $[0,1]\ni t\to \psi_{t}:=P_{\omega}[t\psi]\in\mathcal{M}^{+}$ is weakly continuous. In particular $\mathcal{M}^{+}$, $\mathcal{M}_{klt}^{+},$ $\mathcal{M}_{D}^{+}$ and $\mathcal{M}_{D,klt}^{+}$ are path-connected sets as subset of $PSH(X,\omega)$.
\end{prop}
\begin{proof}
Letting $t_{k}\nearrow \bar{t}\in[0,1]$ (resp. $t_{k}\searrow \bar{t}\in[0,1]$), we observe that the sequence $\psi_{t_{k}}$ converges weakly and monotonically to a model type envelope $\psi'_{\bar{t}}$ which is more singular (resp. less singular) than $\psi_{\bar{t}}$. But by construction it follows that
$$
V_{\psi_{t_{k}}}=\int_{X}t_{k}^{n-j}(1-t_{k})^{j}\langle \omega^{j}\wedge \omega_{\psi}^{j} \rangle \to V_{\psi_{\bar{t}}},
$$
which by Theorem $1.3$ in \cite{DDNL17b} implies that $\psi'_{\bar{t}}=\psi_{\bar{t}}$ since we are assuming $V_{\psi}>0$.\newline
Finally the fact that the subsets $\mathcal{M}_{D}^{+}$, $\mathcal{M}_{klt}^{+}$ and $\mathcal{M}_{D,klt}^{+}$ are path-connected follows from Proposition \ref{prop:Resc} and from the definition of being \emph{klt}.
\end{proof}

Recall that a positive measure $\mu$ on $X$ is said to have \emph{well-defined Ricci curvature} if it corresponds to a singular metric on $K_{X}$, i.e. locally
$$
\mu=e^{-f}i^{n^{2}}\Omega\wedge \bar{\Omega}
$$
where $f\in L^{1}_{loc}$ and $\Omega$ is a nowhere zero local holomorphic section of $K_{X}$, and in such case $Ric(\mu):=dd^{c}f$ (see \cite{Berm16}, \cite{BBJ15}). We also set $Ric(\omega+dd^{c}u):=Ric\big(MA_{\omega}(u)\big)$ for any $u\in PSH(X,\omega)$ so that it coincides with the usual definition of the Ricci curvature when $u\in C^{\infty}$.
\begin{defn}[\cite{Tru20b}]
Let $u\in PSH(X,\omega)$. Then $\omega_{u}:=\omega+dd^{c}u$ is said to be a \emph{Kähler-Einstein metric with prescribed singularities $[\psi]$ ($[\psi]$-KE metric)} if it has well-defined Ricci curvature,
\begin{equation}
\label{eqn:WeakSense}
Ric(\omega_{u})=\omega_{u}
\end{equation}
and $u\in\mathcal{E}^{1}(X,\omega,\psi)$.
\end{defn}
Similarly to the absolute case $\psi=0$, $\omega_{u}$ is a $[\psi]$-KE metric if and only if $u$ solves the complex Monge-Ampère equation
\begin{equation}
\label{eqn:MA!}
\begin{cases}
MA_{\omega}(u)=e^{-u+C}\mu\\
u\in\mathcal{E}^{1}(X,\omega,\psi)
\end{cases}
\end{equation}
for $C\in\mathbbm{R}$ where $\mu$ is a suitable volume form on $X$ such that $Ric(\mu)=\omega$, i.e. $\mu=e^{-f}\omega^{n}$ for $f$ Ricci potential (Lemma $4.3$ in \cite{Tru20b}). Without loss of generality we also assume that $\mu(X)=1$.\newline
Note that combining the resolution of the openness conjecture, Proposition \ref{prop:Lelong} and Theorem $A$ in \cite{DDNL18b}, any solution $u$ of (\ref{eqn:MA!}) has $\psi$-relative minimal singularities.
\begin{rem}
\label{rem:Assumptions}
\emph{As said in the Introduction, there are some necessary assumptions to make on $\psi\in PSH(X,\omega)$ for the search of $[\psi]$-KE metrics, i.e. we need $\psi\in\mathcal{M}_{klt}^{+}$.\newline
Indeed, although the sets $\mathcal{E}(X,\omega,\psi)$, $\mathcal{E}^{1}(X,\omega,\psi)$ can be defined for general $\psi$ and many of the properties recalled in section $\S$ \ref{sec:Prelim} hold, it is not always possible to solve complex Monge-Ampère equations in the class $\mathcal{E}(X,\omega,\psi)$ as shown in \cite{DDNL17b}. As an example, for any unbounded function $\psi\in \mathcal{E}(X,\omega)$ the equation $MA_{\omega}(u)=\omega^{n}$ cannot have solutions in $\mathcal{E}(X,\omega,\psi)$ since $\mathcal{E}(X,\omega,\psi)\subsetneq \mathcal{E}(X,\omega)$. Therefore, it is natural to assume $\psi=P_{\omega}[\psi]$, i.e. $\psi\in\mathcal{M}$, since $\psi-P_{\omega}[\psi]$ must be globally bounded.\newline
Then the assumption of non-zero total mass $V_{\psi}>0$ is clearly necessary by definition of the equation (\ref{eqn:WeakSense}), while the assumption $\psi\in\mathcal{M}_{klt}$ follows from the pluripotential characterization (\ref{eqn:MA!}) as the latter implies that the RHS must have finite total mass (which is equivalent to $\mathcal{I}(\psi)=\mathcal{O}_{X}$ by Proposition \ref{prop:Lelong}).}
\end{rem}
\begin{defn}
We define the \emph{Kähler-Einstein (KE) locus of $\mathcal{M}$} as
$$
\mathcal{M}_{KE}:=\{\psi\in\mathcal{M}^{+}\, : \, \mbox{there exists a }[\psi]\emph{-KE metric}\}.
$$
\end{defn}
Clearly $\mathcal{M}_{KE}\subset \mathcal{M}_{klt}^{+}$ by what said above (Remark \ref{rem:Assumptions})
\begin{rem}
\label{rem:CDS}
\emph{$\mathcal{M}_{KE}$ is not empty. Indeed letting $D$ be a smooth irreducible divisor in $|-rK_{X}|$ for $r\in\mathbbm{N}$, and letting $\varphi_{D}\in PSH(X,\omega)$ such that $\omega+dd^{c}\varphi_{D}=\frac{1}{r}[D]$, then finding a $[\psi_{t}]$-KE metric for $\psi_{t}:=P_{\omega}[t\varphi_{D}]$ and $t\in[0,1)$ is equivalent to solve
$$
Ric(\eta_{v})=\eta_{v}+\frac{t}{r}[D]
$$
where $\eta=(1-t)\omega$. Thus rescaling we get the renowned path
\begin{equation}
\label{eqn:CDSRen}
Ric(\omega_{w})=(1-t)\omega_{w}+\frac{t}{r}[D]
\end{equation}
used for instance in \cite{CDS15}. It is then well-known (\cite{Berm13}, \cite{JMR11}) that (\ref{eqn:CDSRen}) admits a solution for $0\ll t< 1$ close to $1$. Hence there exists a $[\psi_{t}]$-KE metric for $0\ll t<1$ close to $1$.}
\end{rem}
The set of $[\psi]$-KE metrics varying $\psi\in \mathcal{M}^{+}_{klt}$ includes all possible log KE metrics with respect to $(X,D)$ where $D$ varies among all effective klt $\mathbbm{R}$-divisors such that $-(K_{X}+D)$ is ample. But clearly the set of $\psi\in\mathcal{M}^{+}_{klt}$ with analytic singularities types is much bigger than the one associated to these particular pairs $(X,D)$. However, considering a resolution of the ideal $\mathcal{I}$ associated to a model type envelope $\psi$ with analytic singularities type, it is still possible to describe the set of $[\psi]$-KE metrics in a more classical way.
\begin{prop}[\cite{Tru20b}, Proposition $4.8$ and Theorem $4.9$]
\label{prop:AnalRec}
Let $\psi=P_{\omega}[\varphi]\in\mathcal{M}^{+}_{klt}$ with analytic singularities type formally encoded in $(\mathcal{I},c)$. Then any $[\psi]$-KE metric is smooth outside $V(\mathcal{I})$. Morever, letting $p:Y\to X$ be a resolution of the ideal $\mathcal{I}$ and letting $\eta$ be the big and semipositive $(1,1)$-form such that $p^{*}(\omega+dd^{c}\varphi)=\eta+c[D]$, the set of all $[\psi]$-KE metrics is in bijection with the set of all log-KE metrics in the class $\{\eta\}$ with respect to the weak log Fano pair $(Y,\Delta)$ for $\Delta:=cD-K_{Y/X}$.
\end{prop}
In Proposition \ref{prop:AnalRec} by $K_{Y/X}$ we denoted the \emph{relative canonical divisor} of $p:Y\to X$, i.e. $K_{Y/X}=K_{Y}-p^{*}K_{X}$. Note that the divisor $\Delta=cD-K_{Y/X}$ is neither necessarily effective nor necessarily antieffective. For instance, when $p=Id$ $\Delta$ is clearly effective, while, considering $\psi=P_{\omega}[\varphi]$ for $\varphi$ with analytic singularities at one point $x\in X$ such that $\delta:=\nu(\varphi,x)<n-1$ and letting $p:Y\to X$ the blow-up at $x$, it follows that $\Delta=-(n-1-\delta)E$ where $E$ is the exceptional divisor.\newline
Observe that when $\eta$ is Kähler and $\Delta$ is effective then any log KE metric in the class $\{\eta\}$ has \emph{conic singularities} along $D$ on its simple normal crossing locus by Theorem $6.2$ in \cite{GP13}.
\subsection{$\psi$-relative alpha invariant.}
\label{ssec:Alpha}
We introduce the key $\psi$-relative $\alpha$-invariant function and its modified version, which generalize to the relative setting the renowned Tian's $\alpha$-invariant (\cite{Tian87}).
\begin{defn}
Let $\psi\in\mathcal{M}$. We define the $\psi$-\emph{relative $\alpha$-invariant} $\alpha_{\omega}(\psi)$ as
$$
\alpha_{\omega}(\psi):=\sup\Big\{\alpha\geq 0\, : \, \sup_{\{u\preccurlyeq \psi\,:\, \sup_{X}u=0\}}\int_{X}e^{\alpha(\psi-u)}e^{-\psi}d\mu<+\infty\Big\}
$$
if $(X,\psi)$ is klt and $\alpha_{\omega}(\psi)=0$ otherwise. We also define the \emph{modified $\psi$-relative $\alpha$-invariant} $\tilde{\alpha}_{\omega}(\psi)$ as
$$
\tilde{\alpha}_{\omega}(\psi):=\sup\Big\{\alpha\geq 0\, : \, \sup_{\{u\preccurlyeq \psi\,:\, \sup_{X}u=0\}}\int_{X}e^{-\alpha u}d\mu<+\infty\Big\}.
$$
\end{defn}
The following equivalent forms of $\alpha_{\omega}(\psi)$ and of $\tilde{\alpha}_{\omega}(\psi)$ in terms of the \emph{complex singularity exponents} (see \cite{DK99}) are useful in the sequel.
\begin{lem}
\label{lem:Exp}
Let $\psi\in\mathcal{M}$. Let, for any $u\in PSH(X,\omega)$, $c_{\psi}(u):=\sup\{\alpha\geq 0\, : \, \int_{X}e^{\alpha(\psi-u)}e^{-\psi}d\mu<+\infty\}$ if $(X,\psi)$ is klt and $c_{\psi}(u)=0$ otherwise, and $c(u):=\sup\{\alpha\geq 0\, : \, \int_{X}e^{-\alpha u}d\mu<+\infty\}$. Then
\begin{gather*}
\min\big\{\alpha_{\omega}(\psi),1\big\}=\min\big\{\inf_{u\preccurlyeq \psi}c_{\psi}(u),1 \big\}\\
\tilde{\alpha}_{\omega}(\psi)=\inf_{u\preccurlyeq \psi} c(u).    
\end{gather*}
\end{lem}
Note that as a consequence of Proposition \ref{prop:Lelong} $c(u)=c(\psi)$ if $u\in\mathcal{E}(X,\omega,\psi)$ for $\psi\in\mathcal{M}^{+}$.\newline
In the absolute setting $\psi=0$ this characterization of the $\alpha$-invariant was proved by Demailly (see for instance Proposition $8.1$ in \cite{Tos10}). The proof for the $\psi$-relative setting is similar but we write the details as a courtesy to the reader.
\begin{proof}
We first prove the statement for $\tilde{\alpha}_{\omega}(\psi)$.\newline
Clearly $\tilde{\alpha}_{\omega}(\psi)\leq c(u)$ for any $u\preccurlyeq \psi$ with $\sup_{X}u=0$. So the first inequality immediately follows observing that $c(u)=c(u-\sup_{X}u)$ for any $u\preccurlyeq \psi$.\newline
Next, assume by contradiction that there exists $\alpha>0$ such that $\tilde{\alpha}_{\omega}(\psi)<\alpha<\inf_{u\preccurlyeq \psi}c(u)$. Then, there is a sequence $\{u_{j}\}_{j\in\mathbbm{N}}\subset PSH_{norm}(X,\omega,\psi):=\{u\preccurlyeq \psi\, : \, \sup_{X}u=0\}$ such that
\begin{equation}
\label{eqn:Alp}
\int_{X}e^{-\alpha u_{j}}d\mu\geq j
\end{equation}
for any $j\in\mathbbm{N}$. Moreover by weak compactness of $\{u\in PSH(X,\omega)\, : \, \sup_{X}u=0\}$ we may also assume that $u_{j}\to u\in PSH_{norm}(X,\omega,\psi)$ weakly. In particular $\int_{X}e^{-\alpha u}d\mu<+\infty$ since $\alpha<c(u)$. Hence by Theorem \ref{thm:DK99} quoted below $e^{-\alpha u_{j}}\to e^{-\alpha u}$ in $L^{1}$, which contradicts (\ref{eqn:Alp}) and concludes the first part of proof.\newline
Regarding $\alpha_{\omega}(\psi)$, we can clearly assume $(X,\psi)$ to be klt. The inequality $\alpha_{\omega}(\psi)\leq \inf_{u\preccurlyeq \psi}c_{\psi}(u)$ immediately follows from the definitions and it leads to $\min\big\{\alpha_{\omega}(\psi),1\big\}=\min\big\{\inf_{u\preccurlyeq \psi}c_{\psi}(u),1\big\}$. Then assuming by contradiction that $\alpha_{\omega}(\psi)<\alpha<\inf_{u\preccurlyeq \psi}c_{\psi}(u)$ for $\alpha<1$, we can proceed similarly to before getting a sequence $\{u_{j}\}_{j\in\mathbbm{N}}\subset PSH_{norm}(X,\omega,\psi)$ such that
$$
\int_{X}e^{-\alpha u_{j}-(1-\alpha)\psi}d\mu\geq j.
$$
Thus, again by Theorem \ref{thm:DK99} quoted below we get the contradiction by extracting a subsequence $u_{j_{k}}$ converging to $u\preccurlyeq \psi$ since 
$$
+\infty=\lim_{k\to \infty}\int_{X}e^{-\alpha u_{j}-(1-\alpha)\psi}d\mu=\int_{X}e^{\alpha(\psi-u)}e^{-\psi}d\mu<+\infty,
$$
where the last inequality follows from $c_{\psi}(u)>\alpha$. Note that the assumption $\alpha<1$ has been useful to apply the semicontinuity property of Theorem \ref{thm:DK99}.
\end{proof}
\begin{rem}
\label{rem:AnSinTerm}
\emph{If $\psi$ has analytic singularities type (see subsection \S \ref{ssec:Anal}) then the same proof of Lemma \ref{lem:Exp} leads to
\begin{equation}
    \label{eqn:Term}
    \alpha_{\omega}(\psi)=\inf_{u\preccurlyeq \psi}c_{\psi}(u).
\end{equation}
Indeed, letting $\varphi\in PSH(X,\omega)$ with analytic singularities encoded in $(\mathcal{I},c)$ such that $\psi=P_{\omega}[\varphi]$, it is not hard to check that there exists a uniform constant $A>0$ such that $u-\varphi\in PSH(X,A\omega)$ for any $u\in PSH(X,\omega)$, $u\preccurlyeq \psi$ (recall that taking a log resolution of $\mathcal{I}$ as in Proposition \ref{prop:Anal}, $(u-\varphi)\circ p\in PSH(Y,\eta)$ for any $u\in PSH(X,\omega), u\preccurlyeq \psi$ where $\eta$ is a smooth form). Hence we can apply Theorem \ref{thm:DK99} in the proof of Lemma \ref{lem:Exp} to deduce (\ref{eqn:Term}).}
\end{rem}
\begin{thm}[\cite{DK99}, Main Theorem]
\label{thm:DK99}
Let $(X,\omega)$ be a compact Kähler manifold. Then $ PSH(X,\omega)\ni u\to c(u) $ is lower semicontinuous with respect to the weak topology. Moreover, if $\{u_{k}\}_{k\in\mathbbm{N}}\subset PSH(X,\omega)$ converges weakly to $u\in PSH(X,\omega)$, then
$$
e^{-\alpha u_{k}}\to e^{-\alpha u}
$$
in $L^{1}$ for any $\alpha<c(u)$.
\end{thm}
We want now describe some properties of the functions $
\mathcal{M}\ni\psi\to \alpha_{\omega}(\psi), \tilde{\alpha}_{\omega}(\psi) $ starting with the latter. 
\begin{prop}
\label{prop:AlpTian}
The following statements hold:
\begin{itemize}
\item[i)] $\big(\mathcal{M},\preccurlyeq \big)\ni\psi\to \tilde{\alpha}_{\omega}(\psi)\in (0,+\infty)$ is decreasing and right-continuous;
\item[ii)] letting $\psi_{t}:=P_{\omega}[t\psi_{0}+(1-t)\psi_{1}]\in\mathcal{M}$ for $t\in[0,1]$ where $\psi_{0},\psi_{1}\in\mathcal{M}$ such that $\psi_{0}\succcurlyeq \psi_{1}$, then for any $t,s\in[0,1], t\geq s$
$$
t\tilde{\alpha}_{\omega}(\psi_{t})\geq s\tilde{\alpha}_{\omega}(\psi_{s}),
$$
i.e. $[0,1]\ni t\to t\tilde{\alpha}_{\omega}\big(\psi_{t}\big)$ is increasing.
\end{itemize}
\end{prop}
\begin{proof}
By definition $\tilde{\alpha}_{\omega}(\cdot)$ is clearly decreasing, i.e. it decreases when the singularities decreases, and $\tilde{\alpha}_{\omega}(0)>0$ by the uniform version of the Skoda's Integrability Theorem recalled below (Theorem \ref{thm:Skoda}). Then letting $\{\psi_{k}\}_{k\in\mathbbm{N}}\subset\mathcal{M}$ decreasing to $\psi\in\mathcal{M}$ we want to prove that $\tilde{\alpha}_{\omega}(\psi_{k})\to \tilde{\alpha}_{\omega}(\psi)$ as $k\to +\infty$. By monotonicity, we may assume by contradiction that there exists $\alpha>0$ such that $\tilde{\alpha}_{\omega}(\psi_{k})<\alpha <\tilde{\alpha}_{\omega}(\psi)$ for any $k\in\mathbbm{N}$. This implies that for any $k\in\mathbbm{N}$ there exists an element $u_{k}\in PSH_{norm}(X,\omega,\psi_{k}):=\{u\in PSH(X,\omega)\, : \, u\preccurlyeq \psi_{k}, \sup_{X}u=0\}$ such that
\begin{equation}
\label{eqn:AlpInf2}
\int_{X}e^{-\alpha u_{k}}d\mu\geq k.
\end{equation}
By weak compactness we may also suppose $u_{k}\to u\in PSH(X,\omega)$ weakly. Thus, as $u_{k}\leq \psi_{k}$ by construction and $\psi_{k}\searrow \psi$, we obtain $u\in PSH_{norm}(X,\omega,\psi)$. In particular
$$
\int_{X}e^{-\alpha u}d\mu<+\infty
$$
since by assumption $\alpha<\tilde{\alpha}_{\omega}(\psi)$. Finally Theorem \ref{thm:DK99} implies that $e^{-\alpha u_{k}}\to e^{-\alpha u}$ in $L^{1}$ which contradicts (\ref{eqn:AlpInf2}) and concludes the proof of $(i)$.\newline
Next, suppose $\{\psi_{t}\}_{t\in[0,1]}\subset\mathcal{M}$ as in $(ii)$ and let $s,t\in(0,1]$ such that $t\geq s$. Then for any $u\in PSH(X,\omega,\psi_{t})$ we claim that the $\omega$-psh function
\begin{equation}
\label{eqn:BiBi}
v:=\frac{s}{t}u+\frac{t-s}{t}\psi_{1}
\end{equation}
belongs to $PSH(X,\omega,\psi_{s})$. Indeed $v$ is clearly more singular than $\frac{s}{t}\psi_{t}+\frac{t-s}{t}\psi_{1}$, and for any $C>0$ the function $\psi_{t,C}:=P_{\omega}\big(t\psi_{0}+(1-t)\psi_{1}+C,0\big)$ is more singular than $t\psi_{0}+(1-t)\psi_{1}$, i.e.
$$
\frac{s}{t}\psi_{t,C}+\frac{t-s}{t}\psi_{1}\preccurlyeq \frac{s}{t}\big(t\psi_{0}+(1-t)\psi_{1}\big)+\frac{t-s}{t}\psi_{1}=s\psi_{0}+(1-s)\psi_{1}\preccurlyeq \psi_{s},
$$ 
which implies that $\frac{s}{t}\psi_{t,C}+\frac{t-s}{t}\psi_{1}\leq \psi_{s}$ since $\psi_{t,C},\psi_{1}\leq 0$. Thus, letting $C\to +\infty$ and taking the upper semicontinuity regularization, we get $\frac{s}{t}\psi_{t}+\frac{t-s}{t}\psi_{1}\leq \psi_{s}$ which concludes the claim. Then for any $0<\alpha<\tilde{\alpha}_{\omega}(\psi_{s})$, letting $u\in PSH(X,\omega,\psi_{t})$ and $v\in PSH(X,\omega,\psi_{s})$ as in (\ref{eqn:BiBi}) the inequality
$$
\int_{X}e^{-\frac{s}{t}\alpha u}d\mu=\int_{X}e^{-\alpha v}e^{\frac{\alpha(t-s)}{t}\psi_{1}}d\mu\leq \int_{X}e^{-\alpha v}d\mu
$$
implies $tc(u)\geq s\alpha$. Hence Lemma \ref{lem:Exp} concludes the proof by the arbitrariety of $u\in PSH(X,\omega,\psi_{t})$.
\end{proof}
\begin{thm}[\cite{Zer01}, Corollary $3.2$]
\label{thm:Skoda}
Let $K\subset PSH(X,\omega)$ be a weakly compact set such that $\sup_{x\in X}\sup_{u\in K}\nu(u,x)<1$. Then
$$
\sup_{u\in K}\int_{X}e^{-u}\omega^{n}<+\infty.
$$
\end{thm}
The monotonicity property of $\tilde{\alpha}_{\omega}$ described Proposition \ref{prop:AlpTian} does not hold for the $\psi$-relative $\alpha$-invariant $\alpha_{\omega}$ (restricted to $\mathcal{M}_{klt}$) as the next easy example shows.
\begin{esem}
\emph{
Let $X=\mathbbm{P}^{2}$, $\omega=3\omega_{FS}\in c_{1}(\mathbbm{P}^{2})$. Then $\alpha_{\omega}(0)\leq\frac{1}{3}$ since considering $h_{FS}$ the metric associated to $\omega_{FS}$ and $s_{H}$ a hyperplane section, the function $u=3\log |s|_{h_{FS}}^{2}$ belongs to $PSH(\mathbbm{P}^{2},\omega)$ and $ \int_{X}e^{-\alpha u}d\mu<+\infty $
if and only if $\alpha< 1/3$ by an immediate calculation (i.e. $c(u)=1/3$). Actually it can be shown that $\alpha_{\omega}(0)=1/3$ (see \cite{Chel08}).\newline
Consider now the model type envelopes $\psi_{t}:=P_{\omega}[tu]$ for $t\in [0,1)$. Then, since $c(\psi)=c(tu)$ by Proposition \ref{prop:Lelong}, it is easy to check that $\psi_{t}\in \mathcal{M}_{klt}$ if and only if $t<1/3$. However, for any $t<1/3$ fixed and any $\alpha>0$, by an easy calculation (replacing $\psi_{t}$ with $tu$ since $\lVert\psi_{t}-tu\rVert_{\infty}\leq C$) we have $\int_{X}e^{\alpha(\psi_{t}-u)}e^{-\psi_{t}}d\mu<+\infty$ if and only if
$$
\alpha<\frac{1-3t}{3(1-t)}.
$$
Hence, for any $t<1/3$ we get $\psi_{t}\in\mathcal{M}_{klt}^{+}$ but $\alpha_{\omega}(\psi_{t})\leq \frac{1-3t}{3(1-t)}<\frac{1}{3}=\alpha_{\omega}(0)$.
}
\end{esem}
However, although $\mathcal{M}_{klt}\ni \psi\to \alpha_{\omega}(\psi)$ is not a decreasing function, the next Proposition shows that $\alpha_{\omega}(\psi)\sim \tilde{\alpha}_{\omega}(\psi)$ when $c(\psi)\gg 1$.
\begin{prop}
\label{prop:ConnAlphas}
Let $\psi\in\mathcal{M}_{klt}$, and denote by $\mathrm{lct}(X,\psi)=c(\psi)$ the complex singularity exponent of $\psi$. If $\tilde{\alpha}_{\omega}(\psi)>1$ then
\begin{equation}
    \label{eqn:Gr2}
    \alpha_{\omega}(\psi)\geq \tilde{\alpha}_{\omega}(\psi)\\
    \end{equation}
if $\alpha_{\omega}(\psi)>1$ then
\begin{equation}
    \label{eqn:Gr2bis}
    \tilde{\alpha}_{\omega}(\psi)\geq \frac{\alpha_{\omega}(\psi)\mathrm{lct}(X,\psi)}{\alpha_{\omega}(\psi)+\mathrm{lct}(X,\psi)-1},
\end{equation}
if $\alpha_{\omega}(\psi)\leq 1$ then
\begin{gather}
    \label{eqn:Gr3}
    \tilde{\alpha}_{\omega}(\psi)\geq \alpha_{\omega}(\psi)
\end{gather}
and if $\tilde{\alpha}_{\omega}(\psi)\leq 1$ then
\begin{equation}
    \label{eqn:Gr3bis}
    \alpha_{\omega}(\psi)\geq\tilde{\alpha}_{\omega}(\psi) \frac{\mathrm{lct}(X,\psi)-1}{\mathrm{lct}(X,\psi)-\tilde{\alpha}_{\omega}(\psi)}.
\end{equation}
In particular
\begin{equation}
    \label{eqn:Gr1}
    \alpha_{\omega}(\psi)=1 \Longleftrightarrow \tilde{\alpha}_{\omega}(\psi)=1.
\end{equation}
\end{prop}
\begin{proof}
Fix $u\in PSH_{norm}(X,\omega,\psi):=\big\{ u\in PSH(X,\omega) \, :\, u\preccurlyeq \psi, \sup_{X}u=0\big\}$ and recall that $\psi\leq 0$. \newline
If $\alpha\leq 1$ then
\begin{equation}
    \label{eqn:Gren1}
    \int_{X}e^{-\alpha u}d\mu=\int_{X}e^{\alpha(\psi-u)}e^{-\alpha \psi}d\mu\leq \int_{X}e^{\alpha(\psi-u)}e^{-\psi}d\mu,
\end{equation}
which clearly leads to (\ref{eqn:Gr3}) taking $\alpha_{k}\nearrow \alpha_{\omega}(\psi)$ and considering the supremum over all $u\in PSH_{norm}(X,\omega,\psi)$. Similarly, if $\alpha\geq 1$ then
\begin{equation}
    \label{eqn:Gren2}
    \int_{X}e^{\alpha(\psi-u)}e^{-\psi}d\mu=\int_{X}e^{-\alpha u}e^{(\alpha-1)\psi}d\mu\leq \int_{X}e^{-\alpha u}d\mu,
\end{equation}
and (\ref{eqn:Gr2}) follows.\newline
We now want to deduce the right implication in (\ref{eqn:Gr1}). If $\alpha_{\omega}(\psi)=1$, then by (\ref{eqn:Gr3}) we have $\tilde{\alpha}_{\omega}(\psi)\geq\alpha_{\omega}(\psi)=1$. Moreover, $\tilde{\alpha}_{\omega}(\psi)$ cannot be strictly bigger than $1$, because otherwise $1=\alpha_{\omega}(\psi)\geq \tilde{\alpha}_{\omega}(\psi)>1$ by (\ref{eqn:Gr2}).\newline
Next, we want to prove (\ref{eqn:Gr2bis}). Let $1<\alpha<\alpha'<\alpha_{\omega}(\psi)$, and set $p=\alpha'/\alpha>1$. Then by Hölder's inequality ($q=\alpha'/(\alpha'-\alpha)$) we get
\begin{equation*}
    \int_{X}e^{-\alpha u}d\mu=\int_{X}e^{\alpha(\psi-u)}e^{-\frac{\alpha}{\alpha'}\psi}e^{\alpha\big(\frac{1}{\alpha'}-1\big)\psi}d\mu\leq \Big(\int_{X}e^{-\alpha'(\psi-u)}e^{-\psi}d\mu\Big)^{\frac{\alpha}{\alpha'}}\Big(\int_{X}e^{\frac{\alpha(1-\alpha')}{\alpha'-\alpha}\psi}d\mu\Big)^{\frac{\alpha'-\alpha}{\alpha'}},
\end{equation*}
which implies $\tilde{\alpha}_{\omega}(\psi)\geq\alpha$ for any $\alpha<\alpha'$ such that
$$
\frac{\alpha(\alpha'-1)}{\alpha'-\alpha}<\mathrm{lct}(X,\psi).
$$
Therefore by an easy calculation we get
$$
\tilde{\alpha}_{\omega}(\psi)\geq \frac{\alpha'\mathrm{lct}(X,\psi)}{\alpha'+\mathrm{lct}(X,\psi)-1},
$$
which clearly leads to (\ref{eqn:Gr2bis}) letting $\alpha'\nearrow \alpha_{\omega}(\psi)$.\newline
Then, we want to prove (\ref{eqn:Gr3bis}). Let $\alpha<\alpha'<\tilde{\alpha}_{\omega}(\psi)\leq 1$ and set $p=\alpha'/\alpha$. Similarly as before, by Hölder's inequality ($q=\alpha'/(\alpha'-\alpha)$) we obtain
\begin{equation}
    \int_{X}e^{\alpha(\psi-u)}e^{-\psi}d\mu=\int_{X}e^{-\alpha u}e^{(\alpha-1)\psi}d\mu\leq \Big(e^{-\alpha' u}d\mu\Big)^{\frac{\alpha}{\alpha'}}\Big(\int_{X}e^{\frac{\alpha'(\alpha-1)}{\alpha'-\alpha}\psi}d\mu\Big)^{\frac{\alpha'-\alpha}{\alpha'}},
\end{equation}
which implies $\alpha_{\omega}(\psi)\geq\alpha$ for any $\alpha<\alpha'$ such that
$$
\frac{\alpha'(1-\alpha)}{\alpha'-\alpha}<\mathrm{lct}(X,\psi).
$$
Hence $\alpha_{\omega}(\psi)\geq \alpha'\frac{\mathrm{lct}(X,\psi)-1}{\mathrm{lct}(X,\psi)-\alpha'} $, and letting $\alpha'\nearrow {\alpha}_{\omega}(\psi)$, we get (\ref{eqn:Gr3bis}).\newline
It remains to prove the left implication in (\ref{eqn:Gr1}). Assuming $\tilde{\alpha}_{\omega}(\psi)=1$, from (\ref{eqn:Gr3bis}) we deduce
$$
\alpha_{\omega}(\psi)\geq \tilde{\alpha}_{\omega}(\psi)\frac{\mathrm{lct}(X,\psi)-1}{\mathrm{lct}(X,\psi)-\tilde{\alpha}_{\omega}(\psi)}=1.
$$
On the other hand if $\alpha_{\omega}(\psi)>1$ then by (\ref{eqn:Gr2bis}) we obtain the contradiction
$$
1=\tilde{\alpha}_{\omega}(\psi)\geq \frac{\alpha_{\omega}(\psi)\mathrm{lct}(X,\psi)}{\alpha_{\omega}(\psi)+\mathrm{lct}(X,\psi)-1}>1
$$
where the right inequality follows from an easy calculation since $\mathrm{lct}(X,\psi)>1$.
\end{proof}
When $\psi$ has algebraic singularities type, as finding a $[\psi]$-KE metric is equivalent to find a log-KE metric on a resolution (Proposition \ref{prop:AnalRec}), it is natural to wonder if it is possible to express $\alpha_{\omega}(\psi)$ algebraically and what is the connection between this new invariant and the usual \emph{log} $\alpha$\emph{-invariant}. Letting $(Y,\Delta)$ be a weak log Fano pair, i.e. $Y$ be a projective variety (which for our purpose we can assume to be smooth) and $\Delta$ be a $\mathbbm{Q}$-divisor such that $(Y,\Delta)$ is klt and $-(K_{Y}+\Delta)=:L$ is big and semipositive, the log $\alpha$-invariant of the pair $(Y,\Delta)$ is defined as
\begin{multline}
\alpha(Y,\Delta):=\sup\{\alpha\in\mathbbm{Q}_{\geq 0}\, : \, (Y, \Delta+\alpha F)\, \, \mbox{is klt for any}\, \, F\geq 0\, \, \mathbbm{Q}\mbox{-divisor such that} \, \, F\sim_{\mathbbm{Q}}L\}=\\
=\inf_{F\sim_{\mathbbm{Q}}L, F\geq 0}lct(Y,\Delta,F)
\end{multline}
where $\sim_{\mathbbm{Q}}$ is the linear equivalence extended to $\mathbbm{Q}$-divisors through rescaling, i.e. there exists $r\in\mathbbm{N}$ such that $rF\in |rL|$, and where $lct(Y,\Delta,F):=\sup\{\alpha\in\mathbbm{Q}_{\geq 0} \, : \, (Y,\Delta+\alpha F)\, \mbox{is klt}\}$ is the \emph{log canonical threshold} of $F$ with respect to $(Y,\Delta)$. We refer to \cite{Kol13} and \cite{KolPairs} for the theory of singularities of pairs $(X,D)$, here we just need to recall the following analytical description (see Proposition $3.20$ in \cite{KolPairs}): a pair $(Y,\Delta+\alpha F)$ is klt over a projective manifold $Y$ if and only if $e^{-\alpha v_{F}}\nu_{(Y,\Delta)}\in L^{1}$ where if $F=\sum_{j=1}^{m}a_{j}F_{j}$ for prime divisors $F_{j}$ then $v_{F}=\sum_{j=1}^{m}a_{j}\log|s_{j}|^{2}_{h_{j}}$ for $s_{j}$ holomophic sections cutting $F_{j}$ and $h_{j}$ smooth metrics on $\mathcal{O}_{Y}(F_{j})$, and where $\nu_{(Y,\Delta)}$ is an \emph{adapted measure} associated to the pair $(Y,\Delta)$. In particular letting $D_{j}$ prime divisors such that $\Delta=\sum_{k=1}^{l}b_{k}D_{k}$, $t_{k}$ holomorphic sections cutting the divisors $D_{k}$ and $\tilde{h}_{k}$ smooth metric on $\mathcal{O}_{Y}(D_{k})$, then $\nu_{(Y,\Delta)}=e^{-\sum_{k=1}^{l}b_{k}\log|s_{D_{k}}|^{2}_{\tilde{h}_{k}}}dV$ for a suitable volume form $dV$ on $Y$ (see \cite{BBJ15}).

The next Proposition shows that $\alpha_{\omega}(\psi)$ for $\psi$ with algebraic singularities coincides with the log-$\alpha$-invariant of the weak log Fano pair obtained through Proposition \ref{prop:AnalRec}.
\begin{lem}
\label{lem:Exp2}
Let $(Y,\Delta)$ be a weak Fano pair and $\eta$ be a smooth $(1,1)$-form representative of $c_{1}(L)$ where $L:=-(K_{Y}+\Delta)$. Let also $D$ be an effective $\mathbbm{Q}$-divisor on $Y$, and $\theta$ smooth $(1,1)$-form on $\{[D]\}$. Then
\begin{equation}
\label{eqn:UffyPuffy}
\inf_{F\sim_{\mathbbm{Q}}L, F\geq 0}lct(Y,\Delta,F+D)=\inf_{v\in PSH(Y,\eta)}lct(Y,\Delta,v+D)
\end{equation}
where we set $lct(Y,\Delta,v+D):=\sup\{\alpha\in \mathbbm{Q}_{\geq 0}\, : \, \int_{Y}e^{-\alpha(v+v_{D})}d\nu_{(Y,\Delta)}<+\infty\}$ for $v_{D}$ quasi-psh function such that $\theta+dd^{c}v_{D}=[D]$.
\end{lem}
\begin{proof}
One inequality in (\ref{eqn:UffyPuffy}) follows immediately from the fact that any effective $\mathbbm{Q}$-divisor $F\sim_{Q}L$ is associated to a function $v_{F}\in PSH(Y,\eta)$ such that $\eta+dd^{c}v_{F}=[F]$ (obviously $v_{F}$ is defined up to an additive constant), and $lct(Y,\Delta,F+D)=lct(Y,\Delta,v_{F}+D)$ by the analytic characterization of the log canonical threshold recalled above.\newline
For the reverse inequality, letting $\omega'$ be a fixed Kähler form, we first assume $v\in PSH(X,\eta)$ such that $\eta+dd^{c}v\geq \epsilon\omega'$ (i.e. a Kähler current). Then fix $r\in\mathbbm{N}$ such that $rL$ is a line bundle and denote by $h$ the singular hermitian metric on $rL$ associated to $r\eta+dd^{c}rv$. It is then well-known that for any $k\in\mathbbm{N}$ the set $H^{0}\big(Y,L^{kr}\otimes\mathcal{I}(kr v)\big)$ of holomorphic sections $\sigma\in H^{0}(Y,L^{k r})$ such that $\int_{Y}|\sigma|^{2}_{h^{k}}\omega'^{n}<+\infty$ is a not-empty finite-dimensional Hilbert space. Choose for any $k\in\mathbbm{N}$ an element $\sigma_{k}\in H^{0}\big(Y,L^{kr}\otimes \mathcal{I}(kr v)\big)$ of norm $1$ and define
$$
w_{k}:=v+\frac{1}{kr}\log |\sigma_{k}|^{2}_{h^{k}}\in PSH(Y,\eta).
$$
Then since $(Y,\Delta)$ is klt there exists $f\in L^{p'}$ for $p'>1$ such that $\nu_{(Y,\Delta)}=f\omega'^{n}$. We denote by $q'$ the Sobolev conjugate exponent of $p'$. For $k\in\mathbbm{N}$ and $\alpha<lct(Y,\Delta,w_{k}+D)$ fixed, we also set $c:=\frac{\alpha q'}{r}$, $p:=1+\frac{k}{c}$ and $q:=1+\frac{c}{k}$. Clearly $p,q$ are Sobolev conjugate exponents, i.e. $\frac{1}{p}+\frac{1}{q}=1$. Then by construction $\int_{Y}e^{rk(w_{k}-v)}\omega'^{n}=1$, and using Hölder's inequality twice, we obtain
\begin{multline}
\int_{Y}e^{-\frac{rk}{pq'}(v+v_{D})}d\nu_{(Y,\Delta)}=\int_{Y}\Big(e^{\frac{rk}{pq'}	 (w_{k}-v)}e^{-\frac{rk}{pq'}(w_{k}+v_{D})}\Big)d\nu_{(Y,\Delta)}\leq\\
\leq \Big( \int_{Y}e^{\frac{rk}{q'}(w_{k}-v)}f\omega'^{n}\Big)^{\frac{1}{p}}\Big(\int_{Y}e^{-\frac{rkq}{pq'}(w_{k}+v_{D})}d\nu_{(Y,\Delta)}\Big)^{\frac{1}{q}}\leq \\
\leq ||f||^{1/p}_{L^{p'}(\omega'^{n})}\Big(\int_{Y}e^{rk(w_{k}-v)}\omega'^{n}\Big)^{\frac{1}{pq'}}\Big(\int_{Y}e^{-\frac{rkq}{pq'}(w_{k}+v_{D})}d\nu_{(Y,\Delta)}\Big)^{\frac{1}{q}}\leq\\
\leq ||f||^{1/p}_{L^{p'}(\omega'^{n})}\Big(\int_{Y}e^{-\alpha (w_{k}+v_{D})}d\nu_{(Y,\Delta)}\Big)^{1/q}<+\infty
\end{multline}
Thus $lct(Y,\Delta,v+D)\geq \frac{rk}{q'(1+\frac{k}{c})}$, i.e.
$$
lct(Y,\Delta,v+D)\geq\frac{rk}{rk+q' lct(Y,\Delta,w_{k}+D)}lct(Y,\Delta,w_{k}+D)
$$
by the arbitrariness of $\alpha<lct(Y,\Delta,w_{k}+D)$. Therefore as $\eta+dd^{c}w_{k}=[F_{k}]$ for a $\mathbbm{Q}$-effective divisor $F_{k}$ by construction, it follows that
\begin{multline}
\label{eqn:01}
lct(Y,\Delta,v+D)\geq\liminf_{k\to +\infty} \Big( \frac{rk}{rk+q' lct(Y,\Delta,F_{k}+D)}lct(Y,\Delta,F_{k}+D)\Big)\geq \\
\geq\inf_{F\sim_{\mathbbm{Q}}L,F\geq 0}lct(Y,\Delta,F+D) \liminf_{k\to +\infty} \frac{rk}{rk+q' lct(Y,\Delta,F_{k}+D)}=\inf_{F\sim_{\mathbbm{Q}}L,F\geq 0}lct(Y,\Delta,F+D)
\end{multline}
using also that $lct(Y,\Delta,F_{k}+D)\leq lct(Y,\Delta,D)<+\infty$ is uniformly bounded in $k$.\newline
Next, we note that for any $w\in PSH(Y,\eta)$ and for any $t\in[0,1)$, $w_{t}:=tw+(1-t)v\in PSH(Y,\eta)$ is a Kähler current since $\eta_{w_{t}}\geq (1-t)\epsilon \omega'$. Moreover
\begin{multline}
\label{eqn:02}
lct(Y,\Delta,w_{t}+D)=\sup\Big\{\alpha\in \mathbbm{Q}_{\geq 0}\,:\, \int_{Y}e^{-\alpha(w_{t}+v_{D})}d\nu_{(Y,\Delta)}<+\infty \Big\}\leq\\
\leq \sup\Big\{\alpha\in \mathbbm{Q}_{\geq 0}\,:\, \int_{Y}e^{-\alpha t(w+v_{D})}d\nu_{(Y,\Delta)}<+\infty \Big\}=\frac{1}{t} lct(Y,\Delta,w+D)
\end{multline}
as $w_{t}+v_{D}$ is more singular than $t(w+v_{D})$ for $t\in[0,1)$. Hence combining (\ref{eqn:01}) and (\ref{eqn:02}) the conclusion follows letting $t\to 1$.
\end{proof}
\begin{prop}
\label{prop:AlphaAlg}
Assume $\psi=P_{\omega}[\varphi]\in\mathcal{M}_{klt}$ with algebraic singularities type formally encoded in $(\mathcal{I},c)$. Let $p:Y\to X$, $\eta$, $D$ and $\Delta$ as in Proposition \ref{prop:AnalRec}, and let $L=-(K_{Y}+\Delta)$ be the $\mathbbm{Q}$-line bundle on $Y$ such that $c_{1}(L)=\{\eta\}$. Then
\begin{equation*}
    \alpha_{\omega}(\psi)=\alpha(Y,D).
\end{equation*}
\end{prop}
\begin{proof}
As $p^{*}\omega=\eta+\theta$ for $\theta$ smooth $(1,1)$-form and $p^{*}(\omega+dd^{c}\varphi)=\eta+c[D]$, letting $v_{D}$ such that $\theta+dd^{c}v_{D}=c[D]$, it follows from pluriharmonicity that $\varphi\circ p=v_{D}+ a$ for a constant $a$, which we may assume to be equal to $0$ up to replace $v_{D}$ with $v_{D}+a$. Moreover as proved in Proposition $4.8$ in \cite{Tru20b} it is not difficult to check that over the open Zariski set $\Omega$ where $p$ is an isomorphism we have
\begin{equation}
\label{eqn:First!}
p^{-1}_{*}\big(e^{-\varphi}\mu\big)=\nu_{(Y,\Delta)}
\end{equation}
where $\nu_{(Y,\Delta)}$ is an adapted measure of the pair $(Y,\Delta)$. Thus since $p$ is an isomorphism outside a pluripolar set, we can extend to $0$ the measure $\nu_{(Y,\Delta)}$ and (\ref{eqn:First!}) means that the \emph{lift} of $e^{-\varphi}\mu$ is equal to $\nu_{(Y,\Delta)}$.\newline
Next we observe that, since $\psi-\varphi$ is globally bounded, we can replace $\psi$ by $\varphi$ in the definition of $\alpha_{\omega}(\psi)$, i.e.
$$
\alpha_{\omega}(\psi)=\sup\Big\{\alpha\geq 0 \, :\, \sup_{\{u\preccurlyeq \varphi\, : \sup_{X}(u-\varphi)=0\}}\int_{X}e^{\alpha(\varphi-u)}e^{-\varphi}d\mu<+\infty\Big\}=\inf_{u\preccurlyeq \varphi} c_{\varphi}(u)
$$
where the last equality follows from Lemma \ref{lem:Exp} and Remark \ref{rem:AnSinTerm} and where with obvious notation,
$$
c_{\varphi}(u):=\sup\Big\{\alpha\geq 0\, :\, \int_{X}e^{\alpha(\varphi-u)}e^{-\varphi}d\mu<+\infty \Big\}.
$$
Therefore for any $\alpha\geq 0$ and for any $u\preccurlyeq \varphi$ we obtain that $e^{\alpha(\varphi-u)}e^{-\varphi}\mu$ lifts to $e^{-\alpha \tilde{u}}\nu_{(Y,\Delta)}$ where $\tilde{u}:=(u-\varphi)\circ p\in PSH(Y,\eta)$. In particular, since by Proposition \ref{prop:Anal} any $\tilde{u}\in PSH(Y,\eta)$ is given as $(u-\varphi)\circ p$ for $u\preccurlyeq \varphi$, it follows that
$$
\alpha_{\omega}(\psi)=\inf_{\tilde{u}\in PSH(Y,\eta)}\sup\Big\{\alpha\geq 0\, : \, \int_{X}e^{-\alpha \tilde{u}}d\nu_{(Y,\Delta)}<+\infty\Big\}=\inf_{\tilde{u}\in PSH(Y,\eta)}\mathrm{lct}(Y,\Delta,\tilde{u}).
$$
Hence Lemma \ref{lem:Exp2} concludes the proof as
$$
\alpha(Y,\Delta)=\inf_{F\sim_{\mathbbm{Q}}L, F\geq 0}\mathrm{lct}(Y,\Delta,F)
$$
by definition.
\end{proof}
\subsection{Ding functional and uniqueness.}
Similarly to the companion paper \cite{Tru20b}, we define for $\psi\in\mathcal{M}_{klt}^{+}$ the functional $D_{\psi}:\mathcal{E}^{1}(X,\omega,\psi)\to \mathbbm{R}$ as
$$
D_{\psi}:=V_{\psi}L_{\mu}-E_{\psi}
$$
where $L_{\mu}(u):=-\log\int_{X}e^{-u}d\mu$. It is translation invariant and it assumes finite values by Proposition \ref{prop:Lelong} as we are assuming $\psi\in\mathcal{M}_{klt}^{+}$. We call it $\psi$-\emph{relative Ding functional} since it coincides with the renewed Ding functional in the case $\psi=0$.
\begin{rem}
\emph{When $\psi=P_{\omega}[\varphi]$ has analytic singularities type, then with the same notations of Propositions \ref{prop:Anal}, \ref{prop:AnalRec},
$$
D_{\psi}(u)-E_{\psi}(\varphi)=-V_{\eta}\log \int_{X}e^{-\tilde{u}}d\nu-E(\tilde{u})=:D_{\eta}(\tilde{u})
$$
where $\nu$ is an adapted measure associated to the log-setting. Note that $D_{\eta}(\tilde{u})$ is the usual log-Ding functional associated to the pair $(Y,\Delta)$.}
\end{rem}
\begin{prop}
\label{prop:ContDing}
Let $\psi\in\mathcal{M}_{klt}^{+}$. Then $D_{\psi}$ is continuous on $\big(\mathcal{E}^{1}(X,\omega,\psi),d\big)$ and it is lower semicontinuous with respect to the weak topology.
\end{prop}
\begin{proof}
The continuity of $\mathcal{E}^{1}(X,\omega,\psi)\ni u\to L_{\mu}(u)$ with respect to the weak topology is given by Theorem \ref{thm:DK99}. Therefore the result follows observing that $E_{\psi}$ is upper semicontinuous in $\mathcal{E}^{1}(X,\omega,\psi)$ with respect to the weak topology (Proposition \ref{prop:Usc}) while it is strongly continuous by definition.
\end{proof}
In the absolute setting, a key property of the Ding functional is its convexity along weak geodesic segments, which is the starting point to study the uniqueness of KE metrics. The analogue holds in the relative setting if $\psi$ belongs to $\mathcal{M}_{D,klt}^{+}$, and the following result is crucial to prove it.
\begin{thm}
\label{thm:Subh}
Assume $\psi\in\mathcal{M}_{klt}$. Let $u_{0},u_{1}\in \mathcal{E}^{1}(X,\omega,\psi)$ with $\psi$-relative minimal singularities and let $u_{t}$ be the weak geodesic joining them. Then $\mathcal{F}(t):=L_{\mu}(u_{t})$ is subharmonic on $S$. Moreover if $\mathcal{F}$ is affine over the real segment, then there is a holomorphic vector field $V$ with flow $F_{s}$ such that $F_{s}^{*}\big(\omega+dd^{c}u_{s}\big)=\omega+dd^{c}u_{0}$ for any $s\in[0,1]$.
\end{thm}
Note that $\mathcal{F}$ can be thought as a function on $[0,1]$ since $u_{t}$ does not depend on $\mbox{Im}\, t$.
\begin{proof}
Letting $u_{0,k}, u_{1,k}\in \mathcal{H}_{\omega}$ be decreasing bounded sequences converging to $u_{0},u_{1}$ (\cite{BK07}) and letting $u_{t,k}$ be the weak geodesic segment joining $u_{0,k}, u_{1,k}$, Proposition \ref{prop:WGeod}$.(iii)$ gives $u_{t,k}\searrow u_{t}$. We observe that $\mathcal{F}_{k}(t):=L_{\mu}(u_{t,k})$ is subharmonic by Theorem $1.1$ in \cite{Bern15} since, with the same notations and terminology of \cite{Bern15}, replacing the potentials $v\in PSH(X,\omega)$ with the corresponding \emph{metrics}, the functional $L_{\mu}(v)$ becomes $-\log \int_{X}e^{-v}$. In particular, $\mathcal{F}(t)$ is subharmonic since it is the decreasing limit of $\mathcal{F}_{k}(t):=L_{\mu}(u_{t,k})$.\newline
Next assume that $\mathcal{F}(t)$ is affine.\newline
If $\psi$ has small unbounded locus (locally bounded on the complement of a closed complete pluripolar set), writing $u_{t}:=\psi+(u_{t}-\psi)$, we are in the situation described in section \S $6.1$ of \cite{Bern15} as $u_{t}-\psi$ is globally uniformly bounded by Proposition \ref{prop:WGeod}. Thus Theorem $6.1$ in \cite{Bern15} concludes the proof in this case.\newline
In the general case, we need to slightly improve the proof of Theorem $6.1$ in \cite{Bern15}. First, letting $S:=\{t\in\mathbbm{C}\, : \, 0< \mathrm{Re}\,t<1\}$, we proceed as in section \S $2.3$ in \cite{Bern15} to produce a sequence $\{u_{t}^{k}\}_{k\in\mathbbm{N}}$ of positive smooth metrics on $X\times S$ (depending only on $\mathrm{Re}\, t$ on the $t$-variable) which decreases to $u_{t}$. Moreover, letting $A:=||u_{0}-u_{1}||_{\infty}$ and replacing $u_{t}^{k}$ with
$$
\max_{reg}\Big(u_{t}^{k}, \max_{reg}\big(u_{0}^{k}-A\mathrm{Re}\,t, u_{1}^{k}+A(\mathrm{Re}\,t-1)\big)\Big)
$$
for $\max_{reg}$ a regularized maximum, we can also assume that $u_{t}^{k}-u_{s}^{k}$ is bounded uniformly in $s,t, k$. In particular the time-derivative of $u_{t}^{k}$ is uniformly bounded in $t,k$. Then, we fix a non vanishing trivializing section $w=1$ of the trivial line bundle over $X$, and, denoting with $\pi_{\perp}^{t,k}$ the orthogonal projection on the orthogonal complement of the space of holomorphic forms (with respect to the $L^{2}$-norm induced by $u_{t}^{k}$), we can solve the equation
\begin{equation}
\label{eqn:Bibibi}
    \partial^{u_{t}^{k}}v_{t}^{k}=\pi_{\perp}^{t,k}\big(\Dot{u}_{t}^{k}w\big)
\end{equation}
over $X$ for $t,k$ fixed as in \cite{Bern15}, where $\partial^{u_{t}^{k}}:=e^{u_{t}^{k}}\partial e^{-u_{t}^{k}}=\partial-\partial u_{t}^{k}\wedge$. Furthermore, we deduce that the $L^{2}$-estimate of $v_{t}^{k}$ is uniformly bounded in $t,k$ thanks to Remark $2$ and Lemma $6.2$ in \cite{Bern15}. Thus, up to consider a subsequence, it follows that $v_{t}^{k}$ weakly converges in $L^{2}$ to a $(n-1,0)$-form $v_{t}$ as $k\to +\infty$. Moreover, defining $\mathcal{F}_{k}(t):=-\log\int_{X}e^{-u_{t}^{k}}$, as $\partial \bar{\partial}\mathcal{F}_{k}\to 0$ weakly, by the "curvature formula" of Theorem $3.1$ in \cite{Bern15} we obtain that $\bar{\partial} v_{t}=0$, i.e. that $v_{t}$ is a holomorphic $(n-1,0)$-form for any $t$ fixed. Then, the proof of the holomorphy of $v_{t}$ with respect to the $t$-variable is more involved, and we directly refer to section \S $4$ in \cite{Bern15}. Indeed, although in section \S $4$ of \cite{Bern15} Berndtsson proved its results for a curve of \emph{bounded} metrics, the proof of the holomorphy in $t$ of $v_{t}$ works for more general curves of metrics (as he also claimed showing Theorem $6.1$ of the same paper \cite{Bern15}, see also Step $4$ in Theorem $11.1$ in \cite{BBEGZ16}). Note that, being holomorphic on $X\times S$ and being independent of $\mathrm{Im}\, t$ by construction, $v=v_{t}$ is independent of $t$.\newline
Instead, the assumption on the small unbounded locus in Theorem $6.1$ of \cite{Bern15} was useful to prove the formula
\begin{equation}
\label{eqn:KeyFormula}
    \partial \bar{\partial}u_{t}\wedge v=\bar{\partial} \Dot{u}_{t}\wedge w 
\end{equation}
on $X$ in the distributional sense. Once formula (\ref{eqn:KeyFormula}) is proved, one can define the holomorphic vector field $V$ by $V\rfloor w=-v$ so that
\begin{equation*}
\bar{\partial}\Dot{u}_{t}\wedge w=\partial\bar{\partial}u_{t}\wedge v=-\partial \bar{\partial}u_{t}\wedge \big(V\rfloor w\big)=\big(V\rfloor \partial \bar{\partial}u_{t}\big)\wedge w,
\end{equation*}
where in the last equality we used that $\partial\bar{\partial}u_{t}\wedge w=0$ for bidegree reason. Hence $V\rfloor \partial \bar{\partial}u_{t}=\bar{\partial}\Dot{u}_{t}$ as $w$ is nonvanishing. Finally, the Cartan identity $\mathcal{L}_{V}=d V\rfloor+ V\rfloor d$, for the Lie derivative $\mathcal{L}_{V}$ clearly implies
$$
\mathcal{L}_{V}\partial \bar{\partial}u_{t}=\partial V\rfloor \partial \bar{\partial}u_{t} =\partial\bar{\partial}\Dot{u}_{t}=\frac{\partial}{\partial t} \partial \bar{\partial}u_{t}, 
$$
which is what we wanted to prove. Therefore, it remains to show (\ref{eqn:KeyFormula}).\newline
As in \cite{Bern15} (see also Lemma $11.11$ in \cite{BBEGZ16}), the desired distributional equation (\ref{eqn:KeyFormula}) will easily follow applying $\bar{\partial}_{X}$ to the equation
\begin{equation}
    \label{eqn:PAMA}
    \partial_{X}v-\partial_{X}u_{t}\wedge v=\pi_{\perp}^{t}(\Dot{u}_{t}w).
\end{equation}
Moreover, as any plurisubharmonic function belongs to the Sobolev space $W^{1,1}_{loc}$, it will be sufficient to prove (\ref{eqn:PAMA}) over the set $U:=(X\setminus A)\times S$ where $A:=\{z\in X\, :\, \nu(\psi,z)\geq 1/2\}$ is an analytic set (by \cite{Siu74}) so that $e^{-\psi}\in L^{2}_{loc}(U)$ by the Skoda's Integrability Theorem (\cite{Sko72}).\newline
We start rewriting (\ref{eqn:Bibibi}) as
\begin{equation}
    \label{eqn:Boooo}
\partial_{X}\big(v_{t}^{k}e^{-u_{t}^{k}}\big)=\pi_{\perp}^{t,k}\big(\Dot{u}_{t}^{k}w\big)e^{-u_{t}^{k}}.
\end{equation}
By construction $v_{t}^{k}$ converges weakly in $L^{2}(X\times S)$ to $v$, while $e^{-u_{t}^{k}}$ converges strongly in $L^{2}_{loc}(U)$ to $e^{-u_{t}}$ as $u_{t}-\psi$ is uniformly bounded and $u_{t}^{k}$ decreases to $u_{t}$. This is enough to conclude that the left-hand side of (\ref{eqn:Boooo}) converges in the distributional sense to $\partial_{X}\big(ve^{-u_{t}}\big)$. About the right-hand side of (\ref{eqn:Boooo}), we write $\pi_{\perp}^{t,k}(\Dot{u}_{t}^{k}w)=\Dot{u}_{t}^{k}w+h_{t}^{k}$ where $\bar{\partial}_{X}h_{t}^{k}=0$, observing that both terms are uniformly bounded in $L^{2}(X\times S)$ thanks to the uniform Lipschitz continuity of $u_{t}^{k}$. Thus, up to extract a subsequence, we may assume that $h_{t}^{k}\to h_{t}$ converges weakly in $L^{2}(X\times S)$ where $h_{t}$ satisfies $\bar{\partial}_{X}h_{t}=0$. Furthermore, by Lemma $4.1$ in \cite{Bern15} $\Dot{u}_{t}^{k}w$ necessarily converges weakly in $L^{2}$ to $\Dot{u}_{t}w$. Hence, the right-hand side of (\ref{eqn:Boooo}) converges weakly in $L^{1}_{loc}(U)$ to $\big(\Dot{u}_{t}w+h_{t}\big)e^{-u_{t}}$, and the equation
\begin{equation}
    \label{eqn:Bababa}
\partial_{X}(ve^{-u_{t}})=\pi_{\perp}^{t}(\Dot{u}_{t}w)e^{-u_{t}}
\end{equation}
over $U$ follows observing that $\Dot{u}_{t}w+h_{t}$ is orthogonal to holomorphic forms by a similar approximation argument.
Next, an easy analysis shows that
\begin{equation}
    \label{eqn:Cicici}
    \partial_{X}(ve^{-u_{t}})=\lim_{k\to +\infty}\partial_{X}(ve^{-u_{t}^{k}})=\lim_{k\to +\infty}\big(\partial_{X}v-\partial_{X}u_{t}^{k}\wedge v\big)e^{-u_{t}^{k}}
\end{equation}
in the sense of distributions over $U$. Moreover, by the resolution to the openness conjecture (\cite{GZ14}), for any relatively compact set $V\Subset U$ there exists $p>2$ such that $e^{-u_{t}^{k}}$ converges to $e^{-u_{t}}$ in $L^{p}_{loc}(V)$., while $\partial_{X}u_{t}^{k}$ converges in $L^{r}(X\times S)$ to $\partial_{X}u_{t}$ for any $r<2$ by the properties of plurisubharmonic functions (see for instance Theorem $1.48$ in \cite{GZ17}). Thus, the right-hand side of (\ref{eqn:Cicici}) converges in the distributional sense to $(\partial_{X}v-\partial_{X}u_{t}\wedge v)e^{-u_{t}}$ and, combined with (\ref{eqn:Bababa}), this last convergence implies the required equality (\ref{eqn:PAMA}).
\end{proof}
Since by Theorem \ref{thm:Linear} the $\psi$-relative energy $E_{\psi}$ is linear along weak geodesic segments if $\psi\in\mathcal{M}_{D}^{+}$, Theorem \ref{thm:Subh} gives the convexity of $D_{\psi}$ requested.
\begin{cor}
\label{cor:GeodConv}
Let $\psi\in\mathcal{M}_{D,klt}^{+}$. Then the $\psi$-relative Ding functional $D_{\psi}$ is convex along any weak geodesic segment $[0,1]\ni t \to u_{t}\in\mathcal{E}^{1}(X,\omega,\psi)$ joining two potentials $u_{0},u_{1}\in\mathcal{E}^{1}(X,\omega,\psi)$ with $\psi$-relative minimal singularities.
\end{cor}
Next, to prove the first part of Theorem \ref{thmB} we need to introduce the set
$$
\mbox{Aut}(X,[\psi]):=\{F\in\mbox{Aut}(X)\, : \, [F^{*}\psi]=[\psi]\}
$$
of all automorphisms that preserve the singularity type $[\psi]$, where we recall that $[u]=[v]$ is equivalent to $u-v\in L^{\infty}$. Observe that, being a subgroup of $\mbox{Aut}(X)$, $\mbox{Aut}(X,[\psi])$ is a linear algebraic group. We also define $\mbox{Aut}(X,[\psi])^{\circ}:=\mbox{Aut}(X,[\psi])\cap \mbox{Aut}(X)^{\circ}$ where $\mbox{Aut}(X)^{\circ}$ is the connected component of the identity map.
\begin{thm}
\label{thm:FirstA}
Assume $\psi\in\mathcal{M}_{D,klt}^{+}$ and let $u\in\mathcal{E}^{1}(X,\omega,\psi)$. Then the following statements are equivalent:
\begin{itemize}
\item[i)] $\omega_{u}:=\omega+dd^{c}u$ is a $[\psi]$-KE metric;
\item[ii)] $D_{\psi}(u)=\inf_{\mathcal{E}^{1}(X,\omega,\psi)}D_{\psi}$.
\end{itemize}
Furthermore if $\omega_{u},\omega_{v}$ are $[\psi]$-KE metrics, then there exists $F\in \mbox{Aut}(X,[\psi])^{\circ}$ such that $F^{*}(\omega_{u})=\omega_{v}$.
\end{thm}
\begin{proof}
The implication $(ii)\Rightarrow (i)$ is the content of Theorem $4.22$ in \cite{DDNL17b} (whose proof immediately extends to the general case, i.e. without the assumption of small unbounded locus, thanks to Proposition $2.4.(vii)$ in \cite{Tru20a}).\newline
Conversely, the proof of $(i)\Rightarrow (ii)$ is the $\psi$-relative version of that in Theorem $6.6$ in \cite{BBGZ09}.\newline
We want to prove that $D_{\psi}(u)\leq D_{\psi}(v)$ for any $v\in\mathcal{E}^{1}(X,\omega,\psi)$ and by Proposition \ref{prop:ContDing} we may suppose $v$ to have $\psi$-relative minimal singularities as decreasing sequences are strongly continuous. Moreover without loss of generality we can assume $\int_{X}e^{-u}d\mu=V_{\psi}$, i.e. $C=0$ in the Monge-Ampère equation (\ref{eqn:MA!}), recalling also that $u$ has $\psi$-relative minimal singularities as mentioned in the beginning of this section. Then, letting $u_{t}$ be the weak geodesic joining $u_{0}:=u$ and $u_{1}:=v$, Corollary \ref{cor:GeodConv} yields the convexity of $t\to D_{\psi}(u_{t})$. Therefore it is enough to prove that
\begin{equation}
\label{eqn:1}
\frac{d}{dt}\big(D_{\psi}(u_{t})\big)_{|t=0^{+}}\geq 0.
\end{equation}
By Proposition \ref{prop:WGeod} the function $w_{t}:=(u_{t}-u)/t$ is uniformly bounded and converges almost everywhere to a bounded function $w$. Moreover by the concavity of the $\psi$-relative energy (Proposition $2.11$ in \cite{Tru19})
$$
\frac{E_{\psi}(u_{t})-E_{\psi}(u)}{t}\leq \int_{X}w_{t}MA_{\omega}(u)=\int_{X}w_{t}e^{-u}d\mu,
$$
which implies
\begin{equation}
\label{eqn:2}
\frac{d}{dt}\big(E_{\psi}(u_{t})\big)_{|t=0^{+}}\leq \int_{X}w e^{-u}d\mu.
\end{equation}
On the other hand,
$$
\frac{\int_{X}(e^{-u_{t}}-e^{-u})d\mu}{t}=-\int_{X}w_{t}f(u_{t}-u)e^{-u}d\mu
$$
where $f(x):=(1-e^{-x})/x$ is a continuous function. As $||u_{t}-u||_{\infty}\leq C$ for any $t\in[0,1]$, $f(u_{t}-u)$ is uniformly bounded, and by Dominated Convergence Theorem
\begin{equation}
\label{eqn:3}
\frac{d}{dt}\Big(\int_{X}e^{-u_{t}}d\mu\Big)_{|t=0^{+}}=-\int_{X}we^{-u}d\mu
\end{equation}
since $f(u_{t}-u)\to 1$ as $t\to 0^{+}$ again by Proposition \ref{prop:WGeod}. Therefore the required inequality (\ref{eqn:1}) follows combining (\ref{eqn:2}), (\ref{eqn:3}) and using the chain rule of derivation.\newline
Finally, if $\omega_{u},\omega_{v}$ are $[\psi]$-KE metrics, then $u, v$ minimize $D_{\psi}$ and $D_{\psi}$ is constant along the weak geodesic segment $[0,1]\ni t\to u_{t}$ joining $u,v$ by convexity (Corollary \ref{cor:GeodConv}). Hence, as $t\to E_{\psi}(u_{t})$ is affine (Theorem \ref{thm:Linear}), $t\to L_{\mu}(u_{t})$ is affine as well and Theorem \ref{thm:Subh} gives the existence of $F\in \mbox{Aut}(X,[\psi])^{\circ}$ such that $F^{*}(\omega_{u})=\omega_{v}$, concluding the proof.
\end{proof}
\begin{rem}
\emph{The implication $(ii)\Rightarrow (i)$ in Theorem \ref{thm:FirstA} holds as soon as $\psi\in\mathcal{M}^{+}_{klt}$. When $\psi=0$, Theorem \ref{thm:FirstA} has been proved in \cite{BBGZ09}.}
\end{rem}
\subsection{Mabuchi functional and Theorem \ref{thmB}.}
\label{ssec:Funct}
In this subsection we keep assuming $\psi\in\mathcal{M}^{+}_{klt}$. Moreover, without loss of generality we will also assume from now on that $\mu(X)=1$, i.e. that $\mu$ is a probability measure, where we recall that $\mu$ is the suitable volume form defined in the preamble of this section.

Before defining the $\psi$-relative Mabuchi functional we need to recall the $\psi$-relative $I,J$-functionals:
\begin{gather*}
J_{\psi}(u):=\int_{X}(u-\psi)MA_{\omega}(\psi)-E_{\psi}(u),\\
I_{\psi}(u):=\int_{X}(u-\psi)\big(MA_{\omega}(\psi)-MA_{\omega}(u)\big)
\end{gather*}
for any $u\in\mathcal{E}^{1}(X,\omega,\psi)$.
These functionals are translation invariant and strongly continuous, i.e. in $\big(\mathcal{E}^{1}(X,\omega,\psi),d\big)$, as a consequence of Proposition $3.4$ and Corollary $3.5$ in \cite{Tru20a}. Moreover they satisfy the following important properties.
\begin{prop}[\cite{Tru20a}, Lemma $3.1$; \cite{Tru20b}, Lemma $3.7$]
\label{prop:PropertiesIJ}
Let $u\in\mathcal{E}^{1}(X,\omega,\psi)$. Then
\begin{itemize}
\item[i)] $\frac{1}{n+1}I_{\psi}(u)\leq J_{\psi}(u)\leq \frac{n}{n+1}I_{\psi}(u)$;
\item[ii)] there exists a constant $C>0$ depending uniquely on $(X,\omega)$ such that
\begin{equation}
\label{eqn:DandJ}
d(\psi,u)-C\leq J_{\psi}(u)\leq d(\psi,u)
\end{equation}
for any $u\in\mathcal{E}^{1}_{norm}(X,\omega,\psi)$.
\end{itemize}
\end{prop}
We recall that $\mathcal{E}^{1}_{norm}(X,\omega,\psi)$ denotes all elements $u\in\mathcal{E}^{1}(X,\omega,\psi)$ such that $\sup_{X}u=0$.\newline
It is also necessary to retrieve the definition of the \emph{entropy}.
\begin{defn}
Let $\nu_{1},\nu_{2}\in \mathcal{P}(X)$, i.e. two probability measures on X. The \emph{relative entropy} $H_{\nu_{1}}(\nu_{2})\in[0,+\infty]$ of $\nu_{2}$ with respect to $\nu_{1}$ is defined as follows. If $\nu_{2}$ is absolutely continuous with respect to $\nu_{1}$ with density $f:=\frac{d\nu_{2}}{d\nu_{1}}$ satisfying $f\log f\in L^{1}(\nu_{1})$ then
$$
H_{\nu_{1}}(\nu_{2}):=\int_{X}f\log f d\nu_{1}=\int_{X}\log f d\nu_{2}.
$$
Otherwise $H_{\nu_{1}}(\nu_{2}):=+\infty$.
\end{defn}
The following easy and interesting Lemma says that $L^{1}(\nu)$ contains $PSH(X,\omega)$ if $\nu$ has finite entropy (with respect to $\mu$, or to any smooth volume form).
\begin{lem}[\cite{Tru20b}, Lemma $3.16$]
\label{lem:AV}
Let $\nu$ be a probability measure. If $H_{\mu}(\nu)<+\infty$ then $PSH(X,\omega)\subset L^{1}(\nu)$.
\end{lem}
We can now define the $\psi$-relative Mabuchi functional.
\begin{defn}
The $\psi$\emph{-relative Mabuchi functional} $M_{\psi}:\mathcal{E}^{1}(X,\omega,\psi)\to \mathbbm{R}\cup \{+\infty\}$ is defined as
$$
M_{\psi}(u):=V_{\psi}H_{\mu}\big(MA_{\omega}(u)/V_{\psi}\big)+\int_{X}\psi MA_{\omega}(u)+J_{\psi}(u)-I_{\psi}(u)
$$
if $H_{\mu}\big(MA_{\omega}(u)/V_{\psi}\big)<+\infty$, and $M_{\psi}(u)=+\infty$ otherwise.
\end{defn}
Observe that it is clearly a translation invariant functional and that in the absolute setting $\psi=0$ it coincides with the usual Mabuchi functional (see \cite{Mab86} and the Tian's formula in \cite{Chen00}, \cite{Tian}). Moreover by definition and Theorem \ref{thm:Riass} we have
$$
M_{\psi}(u)=\big(V_{\psi}H_{\mu}-V_{\psi}F_{\psi}-E_{\psi}^{*}\big)\big(MA_{\omega}(u)/V_{\psi}\big)
$$
where $F_{\psi}\big(\nu\big):=||\psi||_{L^{1}(\nu)}$ for any $\nu$ probability measure. See subsection \S \ref{ssec:StrongTop} for the definition of the energy $E^{*}$.\newline
The Mabuchi functional dominates the Ding functional as the next result shows.
\begin{prop}
\label{prop:TwoFunct}
$D_{\psi}(u)\leq M_{\psi}(u)$ for any $u\in\mathcal{E}^{1}(X,\omega,\psi)$ with the equality if and only if $\omega_{u}$ is a $[\psi]$-KE metric.
\end{prop}
\begin{proof}
The result is a particular case of Proposition $3.18$ in \cite{Tru20b}, but we present the details here as a courtesy to the reader.\newline
We can assume $H_{\mu}\big(MA_{\omega}(u)/V_{\psi}\big)<+\infty$. By Lemma \ref{lem:AV} we have $PSH(X,\omega)\subset L^{1}\big(MA_{\omega}(u)/V_{\psi}\big)$ and that, setting $\mu_{u}:=\frac{e^{-u}\mu}{\int_{X}e^{-u}\mu}=e^{-u+L_{\mu}(u)}\mu$,
\begin{equation}
\label{eqn:MabuMist}
V_{\psi}H_{\mu}\big(MA_{\omega}(u)/V_{\psi}\big)=V_{\psi}H_{\mu_{u}}\big(MA_{\omega}(u)/V_{\psi}\big)+V_{\psi}L_{\mu}(u)-\int_{X}uMA_{\omega}(u).
\end{equation}
Thus, using the definitions and (\ref{eqn:MabuMist}),
\begin{multline*}
\big(M_{\psi}-D_{\psi}\big)(u)=V_{\psi}H_{\mu}\big(MA_{\omega}(u)/V_{\psi}\big)+\int_{X}\psi MA_{\omega}(u)+\int_{X}(u-\psi)MA_{\omega}(u)-V_{\psi}L_{\mu}(u)=V_{\psi}H_{\mu_{u}}\big(MA_{\omega}(u)/V_{\psi}\big).
\end{multline*}
Therefore $M_{\psi}\geq D_{\psi}$ and the proof concludes observing that $H_{\mu}(\nu)=0$ if and only if $\nu=\mu$ (Proposition $2.10.(ii)$ in \cite{BBEGZ16}).
\end{proof}
We can now finish to prove Theorem \ref{thmB} using the following two Lemmas.
\begin{lem}[\cite{BBEGZ16}, Lemma $2.11.$]
\label{lem:2.11}
For any lower semicontinuous function $g$ on $X$ and any $\nu_{1}\in \mathcal{P}(X)$ probability measure,
$$
\log\int_{X}e^{g}d\nu_{1}=\sup_{\nu_{2}\in\mathcal{P}(X)}\Big(\int_{X}gd\nu_{2}-H_{\nu_{1}}(\nu_{2})\Big).
$$
\end{lem}
\begin{lem}
\label{lem:Inf}
For any $u\in\mathcal{E}^{1}(X,\omega,\psi)$,
\begin{gather*}
V_{\psi}L_{\mu}(u)=\inf_{v\in\mathcal{E}^{1}(X,\omega,\psi)\, : H_{\mu}\big(MA_{\omega}(v)/V_{\psi}\big)<+\infty}\Big(V_{\psi}H_{\mu}\big(MA_{\omega}(v)/V_{\psi}\big)+\int_{X}u\,MA_{\omega}(v)\Big),\\
E_{\psi}(u)=\inf_{v\in\mathcal{E}^{1}(X,\omega,\psi)}\Big(E_{\psi}^{*}\big(MA_{\omega}(v)/V_{\psi}\big)+\int_{X}(u-\psi)MA_{\omega}(v)\Big).
\end{gather*}
\end{lem}
\begin{proof}
The second equality easily follows from the concavity of $E_{\psi}$ (Proposition $2.11$ in \cite{Tru19}) since
$$
E_{\psi}(u)\leq E_{\psi}(v)+\int_{X}(u-v)MA_{\omega}(v)=E_{\psi}^{*}\big(MA_{\omega}(v)/V_{\psi}\big)+\int_{X}(u-\psi)MA_{\omega}(v)
$$
with the equality when $v=u$.\newline
For the first equality, we fix $v\in\mathcal{E}^{1}(X,\omega,\psi)$ such that $H_{\mu}\big(MA_{\omega}(v)/V_{\psi}\big)<+\infty$, recalling also that $PSH(X,\omega)\subset L^{1}\big(MA_{\omega}(v)/V_{\psi}\big)$ by Lemma \ref{lem:AV}. Then, Lemma \ref{lem:2.11} gives
\begin{equation*}
-\infty<-\int_{X}u\,MA_{\omega}(v)-V_{\psi}H_{\mu}\big(MA_{\omega}(v)/V_{\psi}\big)\leq V_{\psi}\log\int_{X}e^{-u}d\mu=-V_{\psi}L_{\mu}(u),
\end{equation*}
which clearly implies
$$
V_{\psi}L_{\mu}(u)\leq \inf_{v\in\mathcal{E}^{1}(X,\omega,\psi)}\Big(V_{\psi}H_{\mu}\big(MA_{\omega}(v)/V_{\psi}\big)+\int_{X}u\,MA_{\omega}(v)\Big).
$$
To prove the reverse equality we set $\mu_{u}:=e^{-u}\mu/\int_{X}e^{-u}\mu=e^{-u+L_{\mu}(u)}\mu$, observing that $H_{\mu}\big(\mu_{u}\big)<+\infty$. Moreover, by the resolution of the openness conjecture there exists $p>1$ such that $e^{-u}\in L^{p}$, thus by Theorem $1.4.(i)$ in \cite{DDNL17b} there exists $v\in\mathcal{E}^{1}(X,\omega,\psi)$ with $\psi$-relative minimal singularities such that $MA_{\omega}(v)=V_{\psi}\mu_{u}$.
Next, since $MA_{\omega}(v)/V_{\psi}=\mu_{u}=e^{-u+L_{\mu}(u)}\mu$, by an easy calculation we obtain
\begin{multline*}
V_{\psi}L_{\mu}(u)=\int_{X}L_{\mu}(u)MA_{\omega}(v)=\int_{X}\Big(\log\Big(\frac{MA_{\omega}(v)/V_{\psi}}{d\mu}\Big)+u\Big)MA_{\omega}(v)=V_{\psi}H_{\mu}\big(MA_{\omega}(v)/V_{\psi}\big)+\int_{X}u\,MA_{\omega}(v),
\end{multline*}
which concludes the proof.
\end{proof}
\begin{reptheorem}{thmB}
Assume $\psi\in\mathcal{M}^{+}_{D,klt}$ and let $u\in\mathcal{E}^{1}(X,\omega,\psi)$. Then the following statements are equivalent:
\begin{itemize}
\item[i)] $\omega_{u}=\omega+dd^{c}u$ is a $[\psi]$-KE metric;
\item[ii)] $D_{\psi}(u)=\inf_{\mathcal{E}^{1}(X,\omega,\psi)}D_{\psi}$;
\item[iii)] $M_{\psi}(u)=\inf_{\mathcal{E}^{1}(X,\omega,\psi)}M_{\psi}$.
\end{itemize}
Moreover if $\omega_{u}$ is a $[\psi]$-KE metric then $u$ has $\psi$-relative minimal singularities and if $\omega_{v}$ is another $[\psi]$-KE metric then there exists $F\in\mbox{Aut}(X,[\psi])^{\circ}$ such that $F^{*}\omega_{v}=\omega_{u}$. 
\end{reptheorem}
\begin{proof}
As said in the beginning of this section if $\omega+dd^{c}u$ is a $[\psi]$-KE metric then by the complex Monge-Ampère equation (\ref{eqn:MA!}) it follows that $u$ has $\psi$-relative minimal singularities. Moreover the uniqueness modulo $\mbox{Aut}(X,[\psi])^{\circ}$ was already stated in Theorem \ref{thm:FirstA} where we also proved the equivalence between $(i)$ and $(ii)$. Furthermore, if $(i)$ and $(ii)$ hold, then $(iii)$ follows from Proposition \ref{prop:TwoFunct}. Thus it is sufficient to show that $(iii)$ implies $(ii)$.\newline
Set $m:=M_{\psi}(u)$. Then for any $v\in\mathcal{E}^{1}(X,\omega,\psi)$ such that $H_{\mu}\big(MA_{\omega}(v)/V_{\psi}\big)<+\infty$ by Lemmas \ref{lem:AV}, \ref{lem:Inf} we obtain
\begin{multline*}
V_{\psi}H_{\mu}\big(MA_{\omega}(v)/V_{\psi}\big)+\int_{X}u\,MA_{\omega}(v)-E_{\psi}(u)\geq\\
\geq V_{\psi}H_{\mu}\big(MA_{\omega}(v)/V_{\psi}\big)+\int_{X}\psi\,MA_{\omega}(v)-E_{\psi}^{*}\big(MA_{\omega}(v)/V_{\psi}\big)=M_{\psi}(v)\geq m
\end{multline*}
Hence taking the infimum among all $v\in\mathcal{E}^{1}(X,\omega,\psi)$ such that $H_{\mu}\big(MA_{\omega}(v)/V_{\psi}\big)<+\infty$ again by Lemma \ref{lem:Inf} we get
$$
\inf_{\mathcal{E}^{1}(X,\omega,\psi)}D_{\psi}\geq m.
$$
So to conclude the proof it is enough to observe that necessarily $D_{\psi}(u)=m$ by Proposition \ref{prop:TwoFunct}.
\end{proof}
\subsection{Proof of Theorem \ref{thmC}.}
In this subsection we assume $\psi\in\mathcal{M}^{+}_{klt}$.\newline

We first prove the following key result that, through the $\psi$-relative $\alpha$-invariant, relates the level sets of the entropy to the $\psi$-relative energy $E^{*}_{\psi}$.
\begin{prop}
\label{prop:Alpha}
Let $0<\alpha<\alpha_{\omega}(\psi)$ and set $C:=\sup_{\{u\preccurlyeq \psi, \sup_{X}u=0\}}\int_{X}e^{\alpha(\psi-u)}e^{-\psi}d\mu$. Then
\begin{equation}
\label{eqn:Bah1}
H_{\mu}(\nu)\geq \frac{\alpha}{V_{\psi}}E_{\psi}^{*}(\nu)-C
\end{equation}
for any $\nu$ probability measure such that $H_{\mu}(\nu)<+\infty$. In particular, if $H_{\mu}(\nu)<+\infty$ then there exists $u\in\mathcal{E}^{1}_{norm}(X,\omega,\psi)$ such that $V_{\psi}\nu=MA_{\omega}(u)$, and, furthermore,
$$
H_{\mu}(\nu)+\int_{X}\psi\, d\nu\geq \frac{\alpha}{V_{\psi}}I_{\psi}(u)-C.
$$
\end{prop}
\begin{proof}
By hypothesis
\begin{equation}
    \label{eqn:Hope2}
    \log\int_{X}e^{\alpha(\psi- u)}e^{-\psi}d\mu\leq -\alpha \sup_{X}u+C
\end{equation}
for any $u\in PSH(X,\omega)$ such that $u\preccurlyeq \psi$. Then since $\sup_{X}u=\sup_{X}(u-\psi)$ (Lemma $3.7.$ in \cite{Tru19} quoted in Lemma \ref{lem:PropOld} below) and clearly $E_{\psi}(u)\leq V_{\psi}\sup_{X}(u-\psi)$ we obtain
\begin{equation}
\label{eqn:Hope}
-\log\int_{X}e^{\alpha(\psi- u)}e^{-\psi}d\mu\geq \frac{\alpha}{V_{\psi}} E_{\psi}(u)-C.
\end{equation}
Next we fix a probability measure $\nu$ such that $H_{\mu}(\nu)<+\infty$. Then by Lemma \ref{lem:AV} $PSH(X,\omega)\subset L^{1}(\nu)$, and
$$
H_{e^{(\alpha-1)\psi+b}\mu}(\nu)=H_{\mu}(\nu)-(\alpha-1)\int_{X}\psi\,d\nu-b
$$
where $b\in \mathbbm{R}$ such that $e^{(\alpha-1)\psi+b}\mu$ is a probability measure. Thus, combining Lemma \ref{lem:2.11} with (\ref{eqn:Hope}) it follows that
\begin{multline*}
H_{\mu}(\nu)+\int_{X}\psi\, d\nu-b=H_{e^{(\alpha-1)\psi+b}\mu}(\nu)+\alpha\int_{X}\psi\, d\nu\geq -\log\int_{X}e^{-\alpha u}e^{(\alpha-1)\psi+b}d\mu+\alpha \int_{X}(\psi-u) \,d\nu\geq\\
\geq \frac{\alpha}{V_{\psi}}\Big(E_{\psi}(u)-\int_{X}(u-\psi) V_{\psi}d\nu\Big)-b-C.
\end{multline*}
Taking the supremum over all $u\in\mathcal{E}^{1}(X,\omega,\psi)$, we obtain (\ref{eqn:Bah1}) by definition of the energy $E^{*}_{\psi}$. We also deduce that $V_{\psi}\nu\in\mathcal{P}^{1}(X,\omega,\psi)$ for any $\nu\in\mathcal{P}(X)$ such that $H_{\mu}(\nu)<+\infty$. Hence by Theorem \ref{thm:Riass} there exists a unique $u\in\mathcal{E}^{1}_{norm}(X,\omega,\psi)$ such that $MA_{\omega}(u)=V_{\psi}\nu$ and, using (\ref{eqn:Hope2}), we easily get
\begin{multline*}
H_{\mu}(\nu)+\int_{X}\psi \, d\nu\geq \alpha \sup_{X}(u-\psi)+\alpha\int_{X}(\psi-u)d\nu-C
\geq \frac{\alpha}{V_{\psi}}\Big(\int_{X}(u-\psi)MA_{\omega}(\psi)-\int_{X}(u-\psi)V_{\psi}d\nu\Big)-C=\\
=\frac{\alpha}{V_{\psi}}I_{\psi}(u)-C,
\end{multline*}
which concludes the proof.
\end{proof}

We recall the definition of $d$-coercivity.
\begin{defn}
Let $F:\mathcal{E}^{1}(X,\omega,\psi)\to \overline{\mathbbm{R}}$ be a translation invariant functional. Then $F$ is said to be $d-$\emph{coercive over} $\mathcal{E}^{1}_{norm}(X,\omega,\psi)$ if there exist $A>0, B\geq 0$ such that
$$
F(u)\geq Ad(u,\psi)-B
$$
for any $u\in\mathcal{E}^{1}_{norm}(X,\omega,\psi)$.
\end{defn}
As an easy consequence of Proposition \ref{prop:PropertiesIJ}, for any translation invariant functional $F$ the $d$-coercivity over $\mathcal{E}^{1}_{norm}(X,\omega,\psi)$ is equivalent to the $J_{\psi}$-coercivity over $\mathcal{E}^{1}(X,\omega,\psi)$, i.e.
$$
F(u)\geq AJ_{\psi}(u)-B
$$
for any $u\in\mathcal{E}^{1}(X,\omega,\psi)$ where $A>0,B\geq 0$.\newline
The $d$-coercivity of the $\psi$-relative Ding functional and of the $\psi$-relative Mabuchi functional are both equivalent to a $\psi$-relative version of a Mose-Trudinger type inequality as our next result shows.
\begin{prop}
\label{prop:Coercivity}
The followings are equivalent:
\begin{itemize}
\item[i)] the $\psi$-relative Ding functional is $d$-coercive over $\mathcal{E}^{1}_{norm}(X,\omega,\psi)$;
\item[ii)] the $\psi$-relative Mabuchi functional is $d$-coercive over $\mathcal{E}^{1}_{norm}(X,\omega,\psi)$;
\item[iii)] there exist $p>1, C>0$ such that
\begin{equation}
\label{eqn:MoserOld}
    ||e^{\psi-u}||_{L^{p}(e^{-\psi}\mu)}\leq C e^{-E_{\psi}(u)/V_{\psi}}
\end{equation}
for any $u\in\mathcal{E}^{1}(X,\omega,\psi)$;
\item[iv)] there exist $p>1, C>0$ such that 
\begin{equation}
\label{eqn:Moser}
||e^{-u}||_{L^{p}(\mu)}\leq Ce^{-E_{\psi}(u)/V_{\psi}}
\end{equation}
for any $u\in\mathcal{E}^{1}(X,\omega,\psi)$.
\end{itemize}
\end{prop}
\begin{proof}
The equivalences $(i)\Leftrightarrow(ii)\Leftrightarrow (iii)$ follows from Proposition $3.19$ in \cite{Tru20b}, but it is useful to rewrite all the details in this setting.\newline
The implication $(i)\Rightarrow (ii)$ follows from Proposition \ref{prop:TwoFunct}. Then let assume $(ii)$ to hold, i.e. there exists $A>0,B\geq 0$ such that
$$
M_{\psi}(u)\geq Ad(u,\psi)-B
$$
for any $u\in\mathcal{E}^{1}_{norm}(X,\omega,\psi)$. As $d(u-\sup_{X}u,\psi)\geq J_{\psi}(u-\sup_{X}u)=J_{\psi}(u) $ for any $u\in\mathcal{E}^{1}(X,\omega,\psi)$ (Proposition \ref{prop:PropertiesIJ}) and $M_{\psi}$ is translation invariant, we get
\begin{equation}
\label{eqn:Mabu}
M_{\psi}(u)\geq AJ_{\psi}(u)-B=A\Big(\frac{n+1}{n}J_{\psi}(u)-\frac{1}{n}J_{\psi}(u)\Big)-B\geq \frac{A}{n}E^{*}_{\psi}\big(MA_{\omega}(u)/V_{\psi}\big)-B
\end{equation}
for any $u\in\mathcal{E}^{1}(X,\omega,\psi)$, where we used again Proposition \ref{prop:PropertiesIJ} in the last inequality. The inequality (\ref{eqn:Mabu}) is equivalent to
$$
V_{\psi}H_{\mu}\big(MA_{\omega}(u)/V_{\psi}\big)+\int_{X}\psi MA_{\omega}(u)\geq p E_{\psi}^{*}\big(MA_{\omega}(u)/V_{\psi}\big)-B
$$
for $p:=1+A/n>1$, which implies $V_{\psi}H_{\mu}(\nu)+\int_{X}\psi\, V_{\psi}d\nu\geq pE_{\psi}^{*}(\nu)-B$ for any $\nu\in \mathcal{P}(X)$ such that $H_{\mu}(\nu)<+\infty$ as, by Proposition \ref{prop:Alpha}, for these probability measures the complex Monge-Ampère equation $MA_{\omega}(u)=V_{\psi}\nu$ admits solution in $\mathcal{E}^{1}(X,\omega,\psi)$. Next, letting $a\in \mathbbm{R}$ such that $e^{(p-1)\psi+a}\mu\in \mathcal{P}(X)$, by Lemma \ref{lem:2.11} it follows that
$$
\log\int_{X}e^{p(\psi-u)-\psi+a}d\mu=\sup_{\nu\in \mathcal{P}(X)}\Big\{-p\int_{X}u\, d\nu-H_{e^{(p-1)\psi+a}\mu}(\nu)\Big\}.
$$
So for any $\epsilon>0$ fixed there exists $\nu_{\epsilon}\in \mathcal{P}(X)$ such that
$H_{e^{(p-1)\psi+a}\mu}(\nu_{\epsilon})<+\infty$ and
$$
\log\int_{X}e^{p(\psi-u)}e^{-\psi+a}d\mu\leq \epsilon-p\int_{X}u\,d\nu_{\epsilon}-H_{e^{(p-1)\psi+a}\mu}(\nu_{\epsilon}).
$$
Moreover, by the definition of entropy it is immediate to check that $H_{\mu}(\nu_{\epsilon})= H_{e^{(p-1)\psi+a}\mu}(\nu_{\epsilon})+a+(p-1)\int_{X}\psi\, d\nu_{\epsilon}$, which in particular yields $H_{\mu}(\nu_{\epsilon})<+\infty$. Therefore by an easy calculation, we obtain
\begin{multline*}
V_{\psi}\log\int_{X}e^{p(\psi-u)}e^{-\psi+a}d\mu\leq V_{\psi}\epsilon +\int_{X}p(\psi-u)V_{\psi}d\nu_{\epsilon}+aV_{\psi}-V_{\psi}H_{\mu}(\nu_{\epsilon})-\int_{X}\psi V_{\psi}d\nu_{\epsilon}\leq\\
\leq V_{\psi}(\epsilon+a)+B+p\Big(\int_{X}(\psi-u)V_{\psi}d\nu_{\epsilon}-E_{\psi}^{*}(\nu_{\epsilon})\Big)\leq V_{\psi}(\epsilon+a)+B-pE_{\psi}(u)
\end{multline*}
where in the inequality we used Lemma \ref{lem:Inf}. Hence, by the arbitrariness of $\epsilon$,
\begin{equation}
    \label{eqn:Tol1}
    V_{\psi}\log\int_{X}e^{p(\psi-u)}e^{-\psi}d\mu\leq-pE_{\psi}(u)+B
\end{equation}
for any $u\in\mathcal{E}^{1}(X,\omega,\psi)$, which is the required inequality (\ref{eqn:MoserOld}).\newline
Next, assuming $(iii)$, i.e. that (\ref{eqn:Tol1}) holds for $p>1, B\geq 0$ uniformly in $u\in\mathcal{E}^{1}(X,\omega,\psi)$, we observe that for $q\in (1,p)$, by Hölder's inequality,
\begin{equation}
    \label{eqn:Tol2}
    \int_{X}e^{-qu}d\mu=\int_{X}e^{q(\psi-u)}e^{(1-q)\psi}e^{-\psi}d\mu\leq\Big(\int_{X}e^{p'q(\psi-u)}e^{-\psi}d\mu\Big)^{\frac{1}{p'}}\Big(\int_{X}e^{q'(1-q)\psi}e^{-\psi}d\mu\Big)^{\frac{1}{q'}}.
\end{equation}
Thus, by the resolution to the Strong Openness Conjecture (\cite{GZ14}) and by the klt assumption on $\psi$, if $q\in (1,p)$ is small enough there exists $p',q'>1$ Sobolev conjugates ($\frac{1}{p'}+\frac{1}{q'}=1$) such that $p'q = p$ while $q'(q-1)+1<\mathrm{lct}(X,\psi)$. Therefore combining (\ref{eqn:Tol1}) and (\ref{eqn:Tol2}), there exists an uniform constant $C>0$ such that
$$
\int_{X}e^{-qu}d\mu\leq Ce^{-qE_{\psi}(u)/V_{\psi}}
$$
for any $u\in\mathcal{E}^{1}(X,\omega,\psi)$, and the inequality (\ref{eqn:Moser}) clearly follows.\newline
Finally supposing $(iv)$ to hold it remains to prove the $d$-coercivity of $D_{\psi}$. Fix $\epsilon\in (0,1)$ small enough such that $p:=1+\epsilon$ satisfies (\ref{eqn:Moser}). Then for any $u\in\mathcal{E}^{1}_{norm}(X,\omega,\psi)$, combining the equality $u=(1+\epsilon)(1-\epsilon)u+\epsilon^{2}u$ with the convexity of $f\to \log \int_{X}e^{-f}d\nu$ for any $\nu\in \mathcal{P}(X)$, we get
$$
\log \int_{X}e^{-u}d\mu\leq (1-\epsilon)\log\int_{X}e^{-(1+\epsilon)u}d\mu+\epsilon\log\int_{X}e^{-\epsilon u}d\mu,
$$
and the first term in the right-hand side is dominated by $(1-\epsilon)\big(-\frac{(1+\epsilon)}{V_{\psi}}E_{\psi}(u)+D\big)$ for a constant $D$ while the second term is uniformly bounded if $\epsilon< \tilde{\alpha}_{\omega}(\psi)$ since $\sup_{X}u=0$.
Therefore it follows that
$$
V_{\psi}\log \int_{X}e^{-u}d\mu\leq -(1-\epsilon^{2})E_{\psi}(u)+B
$$
for any $u\in\mathcal{E}^{1}_{norm}(X,\omega,\psi)$ where $B$ is a uniform constant. Hence
$$
D_{\psi}(u)\geq (1-\epsilon^{2})E_{\psi}(u)-B-E_{\psi}(u)=\epsilon^{2}d(\psi,u)-B,
$$
for any $u\in\mathcal{E}^{1}_{norm}(X,\omega,\psi)$, which concludes the proof.
\end{proof}
We can now prove our second main result which partly generalizes to the relative setting Theorem $2.4$ in \cite{DR15}.
\begin{reptheorem}{thmC}
Let $\psi\in\mathcal{M}_{D,klt}^{+}$ and assume that $\mbox{Aut}(X,[\psi])^{\circ}=\{\mbox{Id}\}$. Then the following conditions are equivalent:
\begin{itemize}
\item[i)] the $\psi$-relative Ding functional is $d$-coercive over $\mathcal{E}^{1}_{norm}(X,\omega,\psi)$;
\item[ii)] the $\psi$-relative Mabuchi functional is $d$-coercive over $\mathcal{E}^{1}_{norm}(X,\omega,\psi)$;
\item[iii)] there exists a unique $[\psi]$-KE metric.
\end{itemize}
\end{reptheorem}
\begin{proof}
The equivalence between $(i)$ and $(ii)$ is part of the content in Proposition \ref{prop:Coercivity}, and the implication $(i)\Rightarrow (iii)$ is included in the proof of from Theorem $C$ in \cite{Tru20b} but we recall it briefly here as a courtesy to the reader. Let $A>0,B\geq 0$ such that $D_{\psi}(u)\geq Ad(\psi,u)-B $ for any $u\in\mathcal{E}^{1}_{norm}(X,\omega,\psi)$ and let $\{u_{k}\}_{k\in\mathbbm{N}}\subset \mathcal{E}^{1}_{norm}(X,\omega,\psi)$ such that $D_{\psi}(u_{k})\searrow \inf_{\mathcal{E}^{1}(X,\omega,\psi)}D_{\psi}\geq -B$. Then there exists $C>0$ such that
$$
\{u_{k}\}_{k\in\mathbbm{N}}\subset \mathcal{E}_{C}^{1}(X,\omega,\psi):=\{u\in\mathcal{E}^{1}(X,\omega,\psi)\, : \, E_{\psi}(u)\geq-C, \sup_{X}u=0\},
$$
which is weakly compact by Proposition \ref{prop:Usc}. Therefore up to considering a subsequence we may suppose $u_{k}\to u\in\mathcal{E}^{1}_{C}(X,\omega,\psi)$ weakly. Moreover by Hartogs' Lemma $u\in\mathcal{E}^{1}_{norm}(X,\omega,\psi)$. Finally by the lower semicontinuity of $D_{\psi}$ (Proposition \ref{prop:ContDing}) it follows that
$$
D_{\psi}(u)\leq \liminf_{k\to +\infty }D_{\psi}(u_{k})=\inf_{\mathcal{E}^{1}(X,\omega,\psi)}D_{\psi},
$$
i.e. $\omega+dd^{c}u$ is the unique KE metric with prescribed singularities $[\psi]$ by Theorem \ref{thm:FirstA}.\newline
Finally we want to prove that $(iii) \Rightarrow (i)$. Letting $u\in\mathcal{E}^{1}_{norm}(X,\omega,\psi)$ such that $\omega_{u}$ is the unique KE metric with prescribed singularities $[\psi]$, we define
$$
A:=\inf\Big\{\frac{D_{\psi}(v)-D_{\psi}(u)}{d(u,v)}\, : \, v\in \mathcal{E}^{1}_{norm}(X,\omega,\psi)\, \mbox{with} \, \psi\mbox{-relative minimal singularities,} \, d(u,v)\geq 1 \Big\},
$$
and we claim that it is enough to prove $A>0$. Indeed setting $B':=A\sup\{d(v,\psi)\, : \, d(v,u)\leq 1\}-D_{\psi}(u)\leq A+Ad(u,\psi)-D_{\psi}(u)$ we clearly have
$$
D_{\psi}(v)\geq Ad(v,\psi)-B'
$$
for any $v\in\mathcal{E}^{1}_{norm}(X,\omega,\psi)$ such that $d(u,v)\leq 1$ since $D_{\psi}(v)\geq D_{\psi}(u)$ (Theorem \ref{thm:FirstA}). Thus we would get 
\begin{equation}
\label{eqn:QCoerc}
D_{\psi}(v)\geq Ad(v,\psi)-\max\big\{B',-D_{\psi}(u)\big\}
\end{equation}
for any $v\in\mathcal{E}^{1}_{norm}(X,\omega,\psi)$ with $\psi$-relative minimal singularities. Moreover, by the strong continuity of $D_{\psi}$ (Proposition \ref{prop:ContDing}) the inequality (\ref{eqn:QCoerc}) would extend to any $v\in\mathcal{E}^{1}_{norm}(X,\omega,\psi)$ considering the canonical approximants $v_{k}:=\max(v,\psi-k)$.\newline
Therefore it remains to prove that $A>0$. Assume by contradiction $A=0$. Then there exists a sequence $\{v^{k}\}_{k\in\mathbbm{N}}\subset \mathcal{E}^{1}_{norm}(X,\omega,\psi)$ of potentials with $\psi$-relative minimal singularities, $d(v^{k},u)\geq 1$ such that
$$
\frac{D_{\psi}(v^{k})-D_{\psi}(u)}{d(v^{k},u)}\longrightarrow 0
$$
as $k\to +\infty$. Thus letting $[0,d(v^{k},u)]\ni t\to v_{t}^{k}$ be the unit speed weak geodesic segment joining $u$ and $v^{k}$, the function $t\to D_{\psi}(v_{t}^{k})$ is convex by Corollary \ref{cor:GeodConv}. Hence defining $w_{k}:=v_{1}^{k}$ we have $d(w_{k},u)=1$ and
\begin{equation}
\label{eqn:Approx}
0\leq D_{\psi}(w_{k})-D_{\psi}(u)\leq \frac{D_{\psi}(v^{k})-D_{\psi}(u)}{d(u,v^{k})}\longrightarrow 0
\end{equation}
as $k\to +\infty$. Moreover as by the triangle inequality
$$
\{w_{k}\}_{k\in\mathbbm{N}}\subset \mathcal{E}^{1}_{1+d(u,\psi)}(X,\omega,\psi):=\{w\in\mathcal{E}^{1}(X,\omega,\psi)\, : \, E_{\psi}(w)\geq -1-d(u,\psi), \sup_{X}w\leq 1+d(u,\psi)\}
$$
which is weakly compact (Proposition \ref{prop:Usc}), up to considering a subsequence we may assume that $w_{k}\to w$ weakly for $w\in\mathcal{E}^{1}_{norm}(X,\omega,\psi)$. But from (\ref{eqn:Approx}) and the lower-semicontinuity of $D_{\psi}$ with respect to the weak topology (Proposition \ref{prop:ContDing}) we get $D_{\psi}(w)\leq D_{\psi}(u)$ which by Theorem \ref{thm:FirstA} implies $w=u$. In particular $\liminf_{k\to +\infty}D_{\psi}(w_{k})=D_{\psi}(w)$ which yields $E_{\psi}(w_{k})\to E_{\psi}(w)$ as $L_{\mu}$ is continuous with respect to the weak topology (Theorem \ref{thm:DK99}). Hence $w_{k}\to u$ strongly, and $1=d(u,w_{k})\to 0$ gives a contradiction.
\end{proof}
\begin{rem}
\label{rem:David}
\emph{As seen during the proof of Theorem \ref{thmC} the $d$-coercivity of the $\psi$-relative Ding functional implies the existence of a $[\psi]$-KE metric as soon as $\psi\in\mathcal{M}^{+}_{klt}$.}
\end{rem}
\section{Why the prescribed singularities setting?}
\label{sec:PSS}
As stated in the Introduction, there are two main reasons to study these KE metrics with prescribed singularities: the naturality of looking for canonical metrics that have prescribed singularities, and the following question.
\begin{repquestion}{Question}
Let $(X,\omega)$ be a Fano manifold. Is it possible to characterize the KE locus $\mathcal{M}_{KE}$? When $\mathcal{M}_{KE}=\mathcal{M}_{klt}^{+}$? What geometric properties does the set $\mathcal{M}_{KE}$ possess?
\end{repquestion}
Recall that $\mathcal{M}_{KE}:=\{\psi\in \mathcal{M}_{klt}^{+}\, : \, \mbox{there exists a}\, [\psi]\mbox{-KE metric}\}$.

To prove Theorem \ref{thmA2}, which gives a first partial answer to Question \ref{Question}, the following two results will be crucial.
\begin{lem}[\cite{Tru19}, Lemmas $3.7$, $4.2$ and Proposition $4.3$]
\label{lem:PropOld}
Let $\psi_{1},\psi_{2}\in\mathcal{M}^{+}$ such that $\psi_{2}\preccurlyeq \psi_{1}$. Then
\begin{itemize}
\item[i)] for any $u,v\in\mathcal{E}^{1}(X,\omega,\psi_{1})$ such that $u\leq v+ C$ for a constant $C$, then $P_{\omega}[\psi_{2}](u)\leq P_{\omega}[\psi_{2}](v) + C$;
\item[ii)] for any $u,v\in\mathcal{E}^{1}(X,\omega,\psi_{1})$, $d\big(P_{\omega}[\psi_{2}](u),P_{\omega}[\psi_{2}](v)\big)\leq d(u,v)$;
\item[iii)] there are two constants $A>1,B>0$ depending uniquely on $X, \omega, n$ such that
$$
-d(u,\psi_{1})\leq V_{\psi_{1}}\sup_{X}u=V_{\psi_{1}}\sup_{X}(u-\psi_{1})\leq Ad(u,\psi_{1})+B
$$
for any $u\in\mathcal{E}^{1}(X,\omega,\psi)$.
\end{itemize}
\end{lem}
\begin{prop}
\label{prop:AlphaResult}
Let $\psi\in\mathcal{M}_{klt}^{+}$. Then for any $0<\alpha<\alpha_{\omega}(\psi)$ there exists a constant $C>0$ such that
\begin{equation}
    \label{eqn:MT1}
    \int_{X} e^{\frac{n+1}{n}\alpha(\psi-u)} e^{-\psi}d\mu\leq Ce^{-\frac{n+1}{n}\alpha \frac{ E_{\psi}(u)}{V_{\psi}}}
\end{equation}
for any $u\in\mathcal{E}^{1}(X,\omega,\psi)$. Moreover for any $0<\alpha< \tilde{\alpha}_{\omega}(\psi)$ there exists a constant $C>0$ such that
\begin{equation}
    \label{eqn:MT2}
    \int_{X}e^{-\alpha u}d\mu\leq Ce^{-\alpha \frac{E_{\psi}(u)}{V_{\psi}}}
\end{equation}
for any $u\in\mathcal{E}^{1}(X,\omega,\psi)$.
\end{prop}
\begin{proof}
Proceeding as in Proposition \ref{prop:Alpha}, there exists a uniform constant $C>0$ such that
\begin{equation}
    \label{eqn:Start0}
    H_{\mu}(\nu)+\int_{X}\psi\,d\nu\geq -\log \int_{X}e^{\alpha(\psi-u)}e^{-\psi}d\mu +\alpha\int_{X}(\psi-u)d\nu\geq \alpha\sup_{X}u+\alpha\int_{X}(\psi-u)\,d\nu-C
\end{equation}
for any $\nu\in \mathcal{P}(X)$ and any $u\in PSH(X,\omega,\psi):=\big\{u\in PSH(X,\omega) \,:\, u\preccurlyeq\psi\big\}$. Thus, letting $b\in\mathbbm{R}$ such that $\mu_{\psi,\alpha}:=e^{\big(\frac{n+1}{n}\alpha-1\big)\psi+b}d\mu$ is a probability measure, by Lemma \ref{lem:2.11} for any $\epsilon>0$ there exists a probability measure $\nu_{\epsilon}\in \mathcal{P}(X)$ with $H_{\mu_{\psi,\alpha}}(\nu_{\epsilon})<+\infty$ such that
\begin{multline*}
    \log\int_{X}e^{\frac{n+1}{n}\alpha(\psi- u)}e^{-\psi+b}d\mu=\log \int_{X}e^{-\frac{n+1}{n}\alpha u}e^{\big(\frac{n+1}{n}\alpha-1\big)\psi+b}d\mu\leq \epsilon - H_{\mu_{\psi,\alpha}}(\nu_{\epsilon})- \frac{n+1}{n}\alpha\int_{X}u\,d\nu_{\epsilon}=\\
    = \epsilon -H_{\mu}(\nu_{\epsilon})+\frac{n+1}{n}\alpha\int_{X}(\psi-u)\, d\nu_{\epsilon}-\int_{X}\psi \,d\nu_{\epsilon}+b
\end{multline*}
where the last equality follows from $H_{\mu_{\psi,\alpha}}(\nu_{\epsilon})=\int_{X}\psi \, d\nu_{\epsilon}-\frac{n+1}{n}\alpha\int_{X}\psi\, d\nu_{\epsilon}-b$. Therefore, letting $v_{\epsilon}\in \mathcal{E}^{1}_{norm}(X,\omega,\psi)$ such that $MA_{\omega}(v_{\epsilon})=V_{\psi}\nu_{\epsilon}$ (the existence of $v_{\epsilon}$ is given by Theorem \ref{thm:Riass} since combining Proposition \ref{prop:Alpha} and Lemma \ref{lem:AV} it follows that $\nu_{\epsilon}\in \mathcal{P}^{1}(X,\omega,\psi)$), thanks to (\ref{eqn:Start0}) we get
\begin{multline}
    \label{eqn:Cont1}
    \log \int_{X}e^{\frac{n+1}{n}\alpha(\psi-u)}e^{-\psi}\,d\mu\leq \epsilon +\frac{n+1}{n}\alpha\int_{X}(\psi-u)\, d\nu_{\epsilon}+\alpha\int_{X}(v_{\epsilon}-\psi)d\nu_{\epsilon}+C=\\
    =\epsilon+ \frac{n+1}{n}\alpha \Big(\int_{X}(\psi-u)\,d\nu_{\epsilon}-\frac{E_{\psi}^{*}(\nu_{\epsilon})}{V_{\psi}}\Big)+\frac{n+1}{n}\alpha \frac{E_{\psi}^{*}(\nu_{\epsilon})}{V_{\psi}}+\alpha\int_{X}(v_{\epsilon}-\psi)d\nu_{\epsilon}+C\leq\\
    \leq\epsilon -\frac{n+1}{n}\alpha\frac{E_{\psi}(u)}{V_{\psi}}+\frac{n+1}{n}\alpha \frac{E_{\psi}^{*}(\nu_{\epsilon})}{V_{\psi}}+\alpha\int_{X}(v_{\epsilon}-\psi)d\nu_{\epsilon}+C
\end{multline}
where the last inequality follows from Lemma \ref{lem:Inf}. Next, combining Theorem \ref{thm:Riass}$.(iii)$ with Proposition $2.12.(ii)$ in \cite{Tru19}, we also get
$$
E_{\psi}^{*}(\nu_{\epsilon})=E_{\psi}(v_{\epsilon})-\int_{X}(v_{\epsilon}-\psi)V_{\psi}\,d\nu_{\epsilon}\leq \Big(\frac{1}{n+1}-1\Big)\int_{X}(v_{\epsilon}-\psi)V_{\psi}\,d\nu_{\epsilon}=-\frac{n}{n+1}\int_{X}(v_{\epsilon}-\psi)V_{\psi}\,d\nu_{\epsilon},
$$
which together with (\ref{eqn:Cont1}) leads to
$$
\log\int_{X}e^{\frac{n+1}{n}\alpha(\psi-u)}e^{-\psi}d\mu\leq -\frac{n+1}{n}\alpha \frac{E_{\psi}(u)}{V_{\psi}}+C,
$$
i.e. (\ref{eqn:MT1}).\newline
Next, assuming $0<\alpha<\tilde{\alpha}_{\omega}(\psi)$ and proceeding similarly to the proof of Proposition \ref{prop:Alpha}, by definition of $\tilde{\alpha}_{\omega}(\psi)$ and by Lemma \ref{lem:2.11} for any $\nu\in\mathcal{P}(X)$ such that $H_{\mu}(\nu)<+\infty$ we have
\begin{multline*}
H_{\mu}(\nu)\geq -\log \int_{X}e^{-\alpha u}d\mu -\alpha\int_{X}u\, d\nu\geq-C+ \alpha \sup_{X}(u-\psi) +\alpha\int_{X}(\psi-u)\,d\nu- \alpha\int_{X}\psi \,d\nu\geq \\
\geq-C+ \frac{\alpha}{V_{\psi}}\Big(E_{\psi}(u)+\int_{X}(\psi-u)\, d\nu\Big)-\alpha\int_{X}\psi\,d\nu
\end{multline*}
and taking the supremum over all $u\in\mathcal{E}^{1}(X,\omega,\psi)$ we get
\begin{equation}
    \label{eqn:Cont2}
    H_{\mu}(\nu)\geq\frac{\alpha}{V_{\psi}}E_{\psi}^{*}(\nu)-\alpha\int_{X}\psi \, d\nu-C
\end{equation}
thanks to Lemma \ref{lem:Inf}. Thus, again by Lemma \ref{lem:2.11} for any $\epsilon>0$ there exists a probability measure $\nu_{\epsilon}$ with finite entropy with respect to $\mu$ such that
\begin{equation}
    \label{eqn:Start2}
    \log \int_{X}e^{-\alpha u}d\mu\leq \epsilon -\alpha\int_{X}u\, d\nu_{\epsilon}-H_{\mu}(\nu_{\epsilon})\leq \epsilon +\alpha \int_{X}(\psi-u)d\nu_{\epsilon}-\frac{\alpha}{V_{\psi}}E_{\psi}^{*}(\nu_{\epsilon})+C,
\end{equation}
where clearly the last inequality follows from (\ref{eqn:Cont2}). Hence, since $\int_{X}(\psi-u)V_{\psi}\,d\nu_{\epsilon}-E_{\psi}^{*}(\nu_{\epsilon})\leq -E_{\psi}(u)$ (Lemma \ref{lem:Inf}), from (\ref{eqn:Start2}) we deduce
$$
\log \int_{X}e^{-\alpha u}d\mu\leq C-\alpha\frac{E_{\psi}(u)}{V_{\psi}},
$$
which is clearly equivalent to (\ref{eqn:MT2}) and concludes the proof.
\end{proof}
\begin{rem}
\emph{To the author's knowledge, in the absolute setting $\psi=0$, Proposition \ref{prop:AlphaResult} combined with the Main Theorem in \cite{Zhang21} gives the first analytic proof to the inequality $\delta(-K_{X})\geq \frac{n}{n+1}\alpha(-K_{X})$, which was proved algebraically in Proposition $2.1$ in \cite{Fuj19}.}
\end{rem}
\begin{reptheorem}{thmA2}
\label{thm:2}
Let $(X,\omega)$ be a Fano manifold. Then
\begin{equation}
\label{eqn:SubsetMKE}
\Big\{\psi\in\mathcal{M}_{klt}^{+}\, : \, \alpha_{\omega}(\psi)>\frac{n}{n+1}\Big\}\subset \mathcal{M}_{KE}.
\end{equation}
Moreover
\begin{itemize}
    \item[i)] $\alpha_{\omega}(0)>\frac{n}{n+1}$ implies
    \begin{equation}
        \label{eqn:F1}
        \Big\{\psi\in\mathcal{M}_{klt}^{+}\, :\, \mathrm{lct}(\psi)>1+\frac{1-\alpha_{\omega}(0)}{\frac{n+1}{n}\alpha_{\omega}(0)-1}\Big\}\subset \mathcal{M}_{KE}.
    \end{equation}
    In particular $\alpha_{\omega}(0)\geq 1$ implies $\mathcal{M}_{KE}=\mathcal{M}_{klt}^{+}$.
\end{itemize}
On the other hand, each of the followings implies $0\in\mathcal{M}_{KE}$:
\begin{itemize}
    \item[ii)] or there exists $\psi\in\mathcal{M}$, $t\in (0,1]$ such that $\tilde{\alpha}_{\omega}\big(P_{\omega}[(1-t)\psi]\big)>\frac{n}{(n+1)t}$;
    \item[iii)] given $\alpha< \alpha_{\omega}(0)$ there exists $\psi\in\big\{\psi'\in \mathcal{M}_{klt}^{+}\, : \, V_{\psi'}>V_{0}(1-\alpha^{2})\big\}$ such that
    \begin{equation}
        \label{eqn:M1}
        \tilde{\alpha}_{\omega}(\psi)> 1+\alpha;
    \end{equation}
    \item[iv)] given $\alpha<\alpha_{\omega}(0)$ there exists $\psi\in\Big\{\psi'\in \mathcal{M}_{klt}^{+}\, : \, V_{\psi'}>V_{0}\big(1-\alpha^{2}\frac{(\mathrm{lct}(\psi')-1)^{2}}{(\mathrm{lct}(\psi')-\alpha)^{2}}\big)\Big\}$ such that
    \begin{equation}
        \label{eqn:M2}
        \alpha_{\omega}(\psi)>\frac{n}{n+1}\Big(1+\alpha\frac{\mathrm{lct}(\psi)-1}{\mathrm{lct}(\psi)-\alpha}\Big);
    \end{equation}
    \item[v)] given $\alpha<\alpha_{\omega}(0)$ there exists $\psi\in\Big\{\psi'\in \mathcal{M}_{klt}^{+}\, : \, V_{\psi'}>V_{0}\big(1-\frac{\alpha}{2}\big)^{2}\Big\}$ such that
    \begin{equation}
        \label{eqn:M3}
        \tilde{\alpha}_{\omega}(\psi)>\frac{\alpha}{1-\big(1-\frac{\alpha}{2}\big)^{2}\frac{V_{0}}{V_{\psi}}};
    \end{equation}
    \item[vi)] given $\alpha<\alpha_{\omega}(0)$ there exists $\psi\in\Big\{\psi'\in \mathcal{M}_{klt}^{+}\, : \, V_{\psi'}>V_{0}\big(1-\frac{\alpha(\mathrm{lct}(\psi')-1)}{2(\mathrm{lct}(\psi')-\alpha)}\big)^{2}\Big\}$ such that
    \begin{equation}
        \label{eqn:M4}
        \alpha_{\omega}(\psi)>\frac{n}{n+1}\Bigg(\frac{\alpha\frac{\mathrm{lct}(\psi)-1}{\mathrm{lct}(\psi)-\alpha}}{1-\big(1-\frac{\alpha(\mathrm{lct}(\psi)-1)}{2(\mathrm{lct}(\psi)-\alpha)}\big)^{2}\frac{V_{0}}{V_{\psi}}}\Bigg);
    \end{equation}
\end{itemize}
\end{reptheorem}
Before proceeding with the proof, let me underline some connections among the conditions $(iii), (iv), (v)$ and $(vi)$. Indeed, when $V_{\psi}>V_{0}(1-\alpha^{2}\frac{(\mathrm{lct}(\psi)-1)^{2}}{(\mathrm{lct}(\psi)-\alpha)^{2}})$ (and $\alpha<1$), the condition (\ref{eqn:M2}) is \emph{better} than (\ref{eqn:M1}), i.e. if (\ref{eqn:M1}) is satisfied then (\ref{eqn:M2}) holds (see also Proposition \ref{prop:ConnAlphas}). A similar link holds between the conditions $(v)$ and $(vi)$. Namely, the $\alpha$-invariant function in general seems to provide a more useful tool than its modified version to understand the existence of genuine KE metrics. However, the subsets of $\mathcal{M}_{klt}^{+}$ taken under consideration in $(iv)$ and $(vi)$ do not include many model type envelopes which instead belong to the subsets considered for $(iii)$ and $(v)$, i.e. elements $\psi\in\mathcal{M}_{klt}^{+}$ with $\mathrm{lct}(\psi)$ not large. Thus, for these kind of model type envelopes the conditions $(iii)$ and $(v)$ on the modified $\alpha$-invariant might play an important role.\newline
Next, when $\alpha$ is small enough and $V_{\psi}>V_{0}(1-\alpha^{2})$, (\ref{eqn:M1}) is a better condition than (\ref{eqn:M3}), i.e. (\ref{eqn:M1}) holds if (\ref{eqn:M3}) is satisfied, but the set of model type envelopes taken under consideration in $(v)$ is in general bigger than that considered in $(iii)$. A similar analysis also applies comparing $(iv)$ and $(vi)$. 
\begin{proof}
\textbf{Step 1: proof of (\ref{eqn:SubsetMKE}).}\newline
Suppose $\psi\in\mathcal{M}_{klt}^{+}$ such that $\alpha_{\omega}(\psi)>\frac{n}{n+1}$. Then the $\psi$-relative Ding functional is $d$-coercive over $\mathcal{E}^{1}_{norm}(X,\omega,\psi)$ as immediate consequence of Proposition \ref{prop:AlphaResult} and Proposition \ref{prop:Coercivity}. Therefore by Theorem \ref{thmC} and Remark \ref{rem:David} there exists a $[\psi]$-KE metrics and (\ref{eqn:SubsetMKE}) follows.\newline
\textbf{Step 2: proof of (i).}\newline
If $\alpha_{\omega}(0)\geq 1$ then by the monotonicity of the modified $\alpha$-invariant function $\tilde{\alpha}_{\omega}(0)$ (Proposition \ref{prop:AlpTian}) we get $\tilde{\alpha}_{\omega}(\psi)\geq 1$ for any $\psi\in \mathcal{M}$. Thus by Proposition \ref{prop:ConnAlphas} we deduce $\alpha_{\omega}(\psi)\geq 1$ for any $\psi\in\mathcal{M}_{klt}$ and (\ref{eqn:SubsetMKE}) gives $\mathcal{M}_{KE}=\mathcal{M}_{klt}^{+}$. Similarly, if $\alpha_{\omega}(0)>\frac{n}{n+1}$ then $\tilde{\alpha}_{\omega}(\psi)>\frac{n}{n+1}$ for any $\psi\in\mathcal{M}$ (Proposition \ref{prop:AlpTian}). For fixed $\psi\in\mathcal{M}_{klt}^{+}$ we can then clearly assume that $\tilde{\alpha}_{\omega}(\psi)\in \big(\frac{n}{n+1},1\big)$ because if $\tilde{\alpha}_{\omega}(\psi)\geq 1$ then $\alpha_{\omega}(\psi)\geq 1$ (Proposition \ref{prop:ConnAlphas}) and $\psi\in\mathcal{M}_{KE}$ by (\ref{eqn:SubsetMKE}). Thus again by Proposition \ref{prop:ConnAlphas} we have
\begin{equation}
    \label{eqn:F3}
    \alpha_{\omega}(\psi)\geq \tilde{\alpha}_{\omega}(\psi)\frac{\mathrm{lct}(\psi)-1}{\mathrm{lct}(\psi)-\tilde{\alpha}_{\omega}(\psi)}\geq \tilde{\alpha}_{\omega}(0)\frac{\mathrm{lct}(\psi)-1}{\mathrm{lct}(\psi)-\tilde{\alpha}_{\omega}(0)},
\end{equation}
where the last inequality follows by monotonicity of $\tilde{\alpha}_{\omega}(\psi)$ (Proposition \ref{prop:AlpTian}). Therefore, by an easy calculation it is immediate to check that (\ref{eqn:F3}) leads to $\alpha_{\omega}(\psi)>\frac{n}{n+1}$ as soon as
$$
\mathrm{lct}(\psi)>1+\frac{n\big(1-\alpha_{\omega}(0)\big)}{(n+1)\big(\alpha_{\omega}(0)-\frac{n}{n+1}\big)},
$$
and hence (\ref{eqn:F1}) follows.\newline
\textbf{Step 3: (ii) implies $\mathbf{0\in\mathcal{M}_{KE}}$.}\newline
If $(ii)$ holds then $\alpha_{\omega}(0)>\frac{n}{n+1}$ thanks to Proposition \ref{prop:AlpTian}$.(ii)$ and clearly $0\in\mathcal{M}_{KE}$.\newline

From now on we can clearly assume $\alpha<1$. \newline
\textbf{Step 4: (iii) implies $\mathbf{0\in\mathcal{M}_{KE}}$.}\newline
Letting $\psi\in\big\{\psi'\in\mathcal{M}_{klt}^{+}\, : \, V_{\psi'}>V_{0}(1-\alpha^{2}) \big\}$ such that (\ref{eqn:M1}) holds, by the convexity of $f\to \log\int_{X}e^{-f}d\mu$ for any $\varphi\in\mathcal{E}^{1}_{norm}(X,\omega)$ we have
$$
\log\int_{X}e^{-\varphi}d\mu\leq (1-\alpha)\log\int_{X}e^{-(1+\alpha)\varphi}d\mu+\alpha \log\int_{X}e^{-\alpha \varphi}d\mu\leq (1-\alpha)\log\int_{X}e^{-(1+\alpha)\varphi}d\mu +C_{1},
$$
where the last inequality clearly follows from $\alpha<\alpha_{\omega}(0)$. Then we set $u:=P_{\omega}[\psi](\varphi)$ to get
$$
\log \int_{X}e^{-(1+\alpha)\varphi}d\mu\leq \log\int_{X}e^{-(1+\alpha)u}d\mu \leq \frac{1+\alpha}{V_{\psi}}(-E_{\psi}(u))
$$
thanks to Proposition \ref{prop:AlphaResult}. Therefore
$$
\log \int_{X}e^{-\varphi}d\mu\leq \frac{1-\alpha^{2}}{V_{\psi}}\big(-E_{\psi}(u)\big)+C_{1}=\frac{1-\alpha^{2}}{V_{\psi}}d(\psi,u) +C_{1}\leq \frac{1-\alpha^{2}}{V_{\psi}}d(0,\varphi) +C_{1}
$$
where the last inequality is a consequence of Lemma \ref{lem:PropOld}. In particular we obtain
$$
D(\varphi)=-E_{0}(\varphi)-V_{0}\log\int_{X}e^{-\varphi}d\mu\geq \Big(1-(1-\alpha^{2})\frac{V_{0}}{V_{\psi}}\Big)d(0,\varphi) +C_{2}
$$
for any $\varphi\in\mathcal{E}^{1}_{norm}(X,\omega)$, i.e. the Ding functional is coercive since we are assuming $V_{\psi}>(1-\alpha^{2})V_{0}$. Hence $0\in\mathcal{M}_{KE}$ by Theorem \ref{thmC}.\newline
\textbf{Step 5: (iv) implies $\mathbf{0\in\mathcal{M}_{KE}}$.}\newline
Assume that $\psi\in \Big\{\psi'\in\mathcal{M}_{klt}^{+}\, : \, V_{\psi'}>V_{0}\big(1-\alpha^{2}\frac{(\mathrm{lct}(\psi')-1)^{2}}{(\mathrm{lct}(\psi')-\alpha)^{2}}\big)\Big\}$ satisfies (\ref{eqn:M2}). We first set $\hat{\alpha}:=\alpha \frac{\mathrm{lct}(\psi)-1}{\mathrm{lct}(\psi)-\alpha}$ to have that
\begin{equation}
    \label{eqn:esp}
    \sup_{\varphi\in\mathcal{E}^{1}_{norm}(X,\omega)}\int_{X}e^{\hat{\alpha}(\psi-\varphi)}e^{-\psi}d\mu\leq C_{3}
\end{equation}
for a uniform constant $C_{3}\in\mathbbm{R}$. Indeed, this follows by Hölder's inequality since $\alpha< \alpha_{\omega}(0)$. More precisely, considering $\delta>1$ small enough such that $\delta \alpha <\alpha_{\omega}(0)$, to get (\ref{eqn:esp}) it is sufficient to apply Hölder's inequality to the functions $f:=e^{-\hat{\alpha}\varphi}$, $g:=e^{-(1-\hat{\alpha})\psi}$ with respect to the weights $p:=\alpha \delta/\hat{\alpha}$ and $q:=p/(p-1)$. Then, proceeding similarly to before, setting $u=P_{\omega}[\psi](\varphi)$, we obtain
\begin{multline}
    \log\int_{X}e^{-\varphi}d\mu=\log\int_{X}e^{(\psi-\varphi)}e^{-\psi}d\mu\leq (1-\hat{\alpha})\log\int_{X}e^{(1+\hat{\alpha})(\psi-\varphi)}e^{-\psi}d\mu+\hat{\alpha}\log\int_{X}e^{\hat{\alpha}(\psi-\varphi)}e^{-\psi}d\mu\leq\\
    \leq (1-\hat{\alpha})\log\int_{X}e^{(1+\hat{\alpha})(\psi-u)}e^{-\psi}d\mu+C_{3}\leq \frac{(1-\hat{\alpha}^{2})}{V_{\psi}}\big(-E_{\psi}(u)\big)+C_{4}\leq\\
    \leq \frac{(1-\hat{\alpha}^{2})}{V_{\psi}}d(0,\varphi)+C_{4},
\end{multline}
which leads to the coercivity of the Ding functional over $\mathcal{E}^{1}_{norm}(X,\omega)$ since $V_{\psi}>V_{0}(1-\hat{\alpha}^{2})$.\newline
\textbf{Step 6: (v) implies $\mathbf{0\in\mathcal{M}_{KE}}$.}\newline
Assume that $\psi\in \big\{\psi'\in\mathcal{M}_{klt}^{+}\, : \, V_{\psi'}>V_{0}(1-\frac{\alpha}{2})^{2}\big\}$ satisfies (\ref{eqn:M3}). It is easy to check that the quantity on the right hand side in (\ref{eqn:M3}) is bigger than $1$. Therefore, by Proposition \ref{prop:AlphaResult} for any $\beta \in \big(1,\tilde{\alpha}_{\omega}(\psi)\big)$ there exists a constant $C_{\beta}>0$ such that
\begin{equation}
    \label{eqn:F5}
    \lVert e^{-u}\rVert_{L^{\beta}(\mu)}\leq C_{\beta}e^{-E_{\psi}(u)/V_{\psi}}
\end{equation}
for any $u\in\mathcal{E}^{1}(X,\omega,\psi)$. We then produce an iteration process if $\beta\geq 1+\alpha$. Namely we define $\beta_{0}=\beta$, $\epsilon_{0}=\alpha/\beta_{0}$ and for any $j\in \{1,\dots,k_{\beta}\}$ for $k_{\beta}:=\lfloor\frac{\beta-1}{\alpha}\rfloor$ we set $\beta_{j}:=\beta_{j-1}-\alpha$ and $\epsilon_{j}:=\alpha/\beta_{j}$. Thus, for any $j\in \{1,\dots,k_{\beta}\}$ and for $\varphi\in\mathcal{E}^{1}_{norm}(X,\omega)$ fixed, again by convexity of $\log \int_{X}e^{-f}d\mu$ we get
\begin{multline*}
\log \int_{X}e^{-\beta_{j}\varphi}d\mu\leq (1-\epsilon_{j})\log \int_{X}e^{-(1+\epsilon_{j})\beta_{j} \varphi}d\mu + \epsilon_{j} \log \int_{X}e^{-\epsilon_{j}\beta_{j}\varphi}d\mu
\leq \frac{\beta_{j+1}}{\beta_{j}}\log \int_{X}e^{-\beta_{j-1}\varphi}d\mu+ C_{5}
\end{multline*}
where the last inequality follows by definition of $\beta_{j},\epsilon_{j}$ and from the fact that $\epsilon_{j}\beta_{j}=\alpha<\alpha_{\omega}(0)$. Therefore we deduce that
$$
\log\int_{X}e^{-\beta_{k_{\beta}}\varphi}d\mu\leq\Big(\prod_{j=1}^{k_{\beta}}\frac{\beta_{j+1}}{\beta_{j}}\Big)\log\int_{X}e^{-\beta \varphi}d\mu+C_{6}= \frac{\beta_{k_{\beta}}-\alpha}{\beta-\alpha}\log\int_{X}e^{-\beta \varphi}d\mu+C_{6},
$$
which leads to 
\begin{multline}
\label{eqn:F6}
    \log\int_{X}e^{-\varphi}d\mu\leq (2-\beta_{k_{\beta}})\log\int_{X}e^{-\beta_{k_{\beta}}\varphi}d\mu+(\beta_{k_{\beta}}-1)\log\int_{X}e^{-(\beta_{k_{\beta}}-1)\varphi}d\mu\leq\\
    \leq \frac{(2-\beta_{k_{\beta}})(\beta_{k_{\beta}}-\alpha)}{\beta-\alpha}\log\int_{X}e^{-\beta \varphi}d\mu+C_{7}
\end{multline}
for any $\varphi\in \mathcal{E}^{1}_{norm}(X,\omega)$ since $\beta_{k_{\beta}}\in [1,1+\alpha)$. Thus, letting $u:=P_{\omega}[\psi](\varphi)$, from (\ref{eqn:F6}) we get
\begin{multline*}
    \log\int_{X}e^{-\varphi}d\mu\leq \frac{(2-\beta_{k_{\beta}})(\beta_{k_{\beta}}-\alpha)}{\beta-\alpha}\log\int_{X}e^{-\beta u}d\mu+C_{7}\leq\\
    \leq(2-\beta_{k_{\beta}})(\beta_{k_{\beta}}-\alpha)\frac{\beta}{\beta-\alpha}\frac{(-E_{\psi}(u))}{V_{\psi}}+C_{8}\leq  (2-\beta_{k_{\beta}})(\beta_{k_{\beta}}-\alpha)\frac{\beta}{\beta-\alpha}\frac{(-E_{0}(\varphi))}{V_{\psi}}+C_{8}
\end{multline*}
where in the last inequalities we clearly used that $\varphi\geq u$, (\ref{eqn:F5}) and that $-E_{\psi}(u)=d(u,\psi)\leq d(\varphi,0)=-E_{0}(\varphi)$ by Lemma \ref{lem:PropOld}. In particular for any $\varphi\in\mathcal{E}^{1}_{norm}(X,\omega)$, we have
\begin{multline*}
    D(\varphi)=-E_{0}(\varphi)-V_{0}\log\int_{X}e^{-\varphi}d\mu\geq d(0,\varphi)-(2-\beta_{k_{\beta}})(\beta_{k_{\beta}}-\alpha)\frac{\beta V_{0}}{(\beta-\alpha)V_{\psi}}d(0,\varphi)-C_{9}\geq\\
    \geq \Big(1-(2-\beta_{k_{\beta}})(\beta_{k_{\beta}}-\alpha)\frac{\beta V_{0}}{(\beta-\alpha)V_{\psi}}\Big)d(0,\varphi)-C_{9},
\end{multline*}
i.e. the Ding functional is coercive over $\mathcal{E}^{1}_{norm}(X,\omega)$ if
\begin{equation}
    \label{eqn:F7}
    (2-\beta_{k_{\beta}})(\beta_{k_{\beta}}-\alpha) \frac{\beta V_{0}}{(\beta-\alpha)V_{\psi}}< 1.
\end{equation}
Since $\beta_{k_{\beta}}\in [1,1+\alpha)$ and the function $f(x)=(2-x)(x-\alpha)$ assumes its maximum at $x=1+\frac{\alpha}{2}$, the condition (\ref{eqn:F7}) holds if
$$
\Big(1-\frac{\alpha}{2}\Big)^{2}\frac{\beta V_{0}}{(\beta-\alpha)V_{\psi}}< 1,
$$
i.e. when
\begin{equation}
    \label{eqn:Maf}
    \beta>\frac{\alpha}{1-\frac{V_{0}}{V_{\psi}}\big(1-\frac{\alpha}{2}\big)^{2}}
\end{equation}
since we are assuming $V_{\psi}>V_{0}\big(1-\frac{\alpha}{2}\big)^{2}$. Hence, from (\ref{eqn:M3}) it is immediate to see that (\ref{eqn:Maf}) holds as soon as $\beta\in \big(1,\tilde{\alpha}_{\omega}(\psi)\big)$ is chosen to be close enough to $\tilde{\alpha}_{\omega}(\psi)$.\newline
\textbf{Step 7: (vi) implies $\mathbf{0\in\mathcal{M}_{KE}}$.}\newline
Finally, suppose $\psi\in\Big\{\psi'\in\mathcal{M}_{klt}^{+}\, : \, V_{\psi'}>V_{0}\big(1-\frac{\alpha(\mathrm{lct}(\psi')-1)}{2(\mathrm{lct}(\psi')-\alpha)}\big)^{2}\Big\}$ such that (\ref{eqn:M4}) holds. Similarly, to the previous step, it is easy to check that the quantity on the RHS of (\ref{eqn:M4}) is bigger than $\frac{n}{n+1}$. Then by Proposition \ref{prop:AlphaResult} for any $\beta\in \big(1, \frac{n+1}{n}\alpha_{\omega}(\psi)\big)$ there exists a constant $C_{\beta}>0$ such that
$$
\lVert e^{\psi-u} \rVert_{L^{\beta}(e^{-\psi}\mu)}\leq C_{\beta}e^{-E_{\psi}(u)/V_{\psi}}
$$
for any $u\in\mathcal{E}^{1}(X,\omega,\psi)$. Moreover, setting $\hat{\alpha}:=\alpha \frac{\mathrm{lct}(\psi)-1}{\mathrm{lct}(\psi)-\alpha}$, we can run the same iteration process of before replacing $\alpha$ with $\hat{\alpha}$ to get
\begin{multline*}
\log \int_{X}e^{\beta_{j}(\psi-\varphi)}e^{-\psi}d\mu\leq (1-\epsilon_{j})\log\int_{X}e^{(1+\epsilon_{j})\beta_{j}(\psi-\varphi)}e^{-\psi}d\mu+\epsilon_{j}\log\int_{X}e^{\epsilon_{j}\beta_{j}(\psi-\varphi)}e^{-\psi}d\mu\leq\\
\leq\frac{\beta_{j+1}}{\beta_{j}}\log\int_{X}e^{\beta_{j-1}(\psi-\varphi)}e^{-\psi}d\mu +C_{10}
\end{multline*}
where the last inequality follows from (\ref{eqn:esp}). Thus, similarly to Step 6, we obtain
$$
\log\int_{X}e^{-\varphi}d\mu=\log\int_{X}e^{(\psi-\varphi)}e^{-\psi}d\mu\leq (2-\beta_{k_{\beta}})(\beta_{k_{\beta}}-\hat{\alpha})\frac{\beta}{\beta-\hat{\alpha}}\frac{\big(-E_{0}(\varphi)\big)}{V_{\psi}}+C_{11}
$$
for any $\varphi\in\mathcal{E}^{1}_{norm}(X,\omega)$. As seen above, this leads to the coercivity of the Ding functional over $\mathcal{E}^{1}_{norm}(X,\omega)$ if
$$
\beta>\frac{\hat{\alpha}}{1-(1-\frac{\hat{\alpha}}{2})^{2}\frac{V_{0}}{V_{\psi}}}
$$
since we are assuming $V_{\psi}>V_{0}(1-\frac{\hat{\alpha}}{2})^{2}$. Choosing $\beta\in \big(1,\frac{n+1}{n}\alpha_{\omega}(\psi)\big)$ close enough to $\frac{n+1}{n}\alpha_{\omega}(\psi)$, we get the coercivity of the Ding functional thanks to (\ref{eqn:M4}). Hence Theorem \ref{thmC} implies $0\in\mathcal{M}_{KE}$ concluding the proof.
\end{proof}
The monotonicity of $\tilde{\alpha}_{\omega}(\cdot)$ and Theorem \ref{thmA2} may suggest that
\begin{equation}
\label{eqn:ForConj}
0\in\mathcal{M}_{KE} \Longrightarrow \mathcal{M}_{KE}=\mathcal{M}_{klt}^{+}
\end{equation}
but, when $\mbox{Aut}(X)$ is not finite, this is false as the following easy example shows.
\begin{esem}
\label{esem:1}
\emph{Let $X=\mathbbm{P}^{2}$, $\omega=\omega_{FS}$ and let $Z:=\{p\}\in X$ be a point on $X$. Then it is well-known that there exists a function $\varphi\in PSH(X,\omega)$ with analytic singularities formally encoded in $(\mathcal{I}_{Z},1)$. So, letting $\psi:=P_{\omega}[\varphi]\in\mathcal{M}_{klt}^{+}$, by Proposition \ref{prop:AnalRec} and with the same notations, the set of $[\psi]$-KE metrics is in bijection with the set of log KE metrics in the class $\{\eta\}$ for the weak log Fano pair $(Y,\Delta)$ where $Y=\mathrm{Bl}_{Z}X$. But $\{\eta\}=-K_{Y}$ and $\Delta=0$, so we necessarily have $\psi\notin \mathcal{M}_{KE}$ as $Y$ does not have any KE metric.}
\end{esem}
However, we think that the existence of no trivial holomorphic vector fields is the unique obstruction to (\ref{eqn:ForConj}), i.e. we pose the following conjecture.
\begin{repconjecture}{conjA}
Let $(X,\omega)$ be a Fano manifold such that $\mbox{Aut}(X)^{\circ}=\{\mbox{Id}\}$. Then
$$
0\in\mathcal{M}_{KE} \Longrightarrow \mathcal{M}_{KE}=\mathcal{M}_{klt}^{+}.
$$
\end{repconjecture}
\begin{rem}
\emph{As immediate consequence of Theorem \ref{thmA2}, Conjecture \ref{conjA} holds for any $X$ Fano manifolds such that $\alpha_{\omega}(0)\geq 1$. For instance, Conjecture \ref{conjA} holds for $S$ Del Pezzo surface of degree $1$ such that $\lvert -K_{S} \rvert$ contains no cuspidal curves (see Theorem $1.7$ in \cite{Chel08}) and for any $X$ general hypersurface of degree $n$ in $\mathbbm{P}^{n}$ for $n\geq 6$ (see Theorem $2.(i)$ in \cite{Puk05}).}
\end{rem}
\subsection{Strong continuity in $\mathcal{M}_{KE}$.}
Here we prove our Theorem \ref{thmA}.
\begin{reptheorem}{thmA}
Let $X$ be a Fano manifold and let $\{\psi_{t}\}_{t\in[0,1]}\subset\mathcal{M}_{klt}^{+}$ be a weakly continuous segment such that
\begin{itemize}
\item[i)] $\psi_{0}\in\mathcal{M}_{KE}$;
\item[ii)] $\mbox{Aut}(X,[\psi_{t}])^{\circ}=\{\mbox{Id}\}$ for any $t\in [0,1]$;
\item[iii)] $\psi_{t}\preccurlyeq \psi_{s}$ if $t\leq s$;
\item[iv)] $\{\psi_{t}\}_{t\in[0,1]}\subset\mathcal{M}_{D}$.
\end{itemize}
Then the set
$$
S:=\{t\in[0,1]\, : \, \psi_{t}\in \mathcal{M}_{KE}\}
$$
is open, the unique family of $[\psi_{t}]$-KE metrics $\{\omega_{u_{t}}\}_{t\in S}$ is weakly continuous and the family of potentials $\{u_{t}\}_{t\in S}$ can be chosen so that the curve $S\ni t\to u_{t}\in \mathcal{E}^{1}(X,\omega,\psi_{t})$ is strongly continuous.
\end{reptheorem}
Observe that by Propositions \ref{prop:Resc}, \ref{lem:Star} the assumptions $(iii), (iv)$ are automatically satisfied for the segment $\psi_{t}:=P_{\omega}[(1-t)\psi]$ if $\psi\in\mathcal{M}_{D,klt}^{+}$. We also recall that the strong convergence means that $u_{t}\to u$ weakly and $E_{\psi_{t}}(u_{t})\to E_{\psi}(u)$ (section \S \ref{sec:Prelim}).\newline

The strategy to prove Theorem \ref{thmA} is to use the new continuity method introduced in the companion paper \cite{Tru20b} where one \emph{moves the prescribed singularities} instead of the density of the Monge-Ampère equation (see \cite{Tru20b} for a mixed continuity method). The main tools to prove Theorem \ref{thmA} are the openness and the closedness results given by the following theorems.
\begin{thm}[\cite{Tru20b}, Theorem $C$]\footnote{The result in the companion paper \cite{Tru20b} is stated without the quantitative description of the constants $A>1,B>0$, whose form can be easily deduced from the proof.}
\label{thm:Open}
Let $\psi\in\mathcal{M}^{+}_{klt}$. If $D_{\psi}$ is $d$-coercive over $\mathcal{E}^{1}_{norm}(X,\omega,\psi)$, then there exist constants $A=A(\psi,X,\omega)>1, B=B(\psi,X,\omega)>0$ such that
$$
D_{\psi'}(u')\geq \Big(1-\frac{V_{\psi'}}{AV_{\psi}}\Big)d(u',\psi')-B
$$
for any $\psi'\in\mathcal{M}_{klt}^{+},$ $\psi'\succcurlyeq \psi$ and for any $u'\in\mathcal{E}^{1}_{norm}(X,\omega,\psi')$. In particular $D_{\psi'}$ is coercive when $V_{\psi'}< AV_{\psi}$.
\end{thm}
\begin{thm}[\cite{Tru20b}, Theorem $D$]
\label{thm:Closed}
Let $\{\psi_{k}\}_{k\in\mathbbm{N}}\subset \mathcal{M}^{+}_{klt}$ be a increasing sequence of model type envelopes converging weakly to $\psi\in\mathcal{M}^{+}_{klt}$. Assume that
\begin{itemize}
\item[i)] $\omega_{u_{k}}$ is a sequence of $[\psi_{k}]$-KE metrics where $u_{k}\in\mathcal{E}^{1}(X,\omega,\psi_{k})$ minimizes $D_{\psi_{k}}$ and it is normalized so that satisfies $MA_{\omega}(u_{k})=e^{-u_{k}}\mu$ for any $k\in\mathbbm{N}$;
\item[ii)] the sequence $\{u_{k}\}_{k\in\mathbbm{N}}$ is uniformly bounded from above, i.e. there exists $C\in\mathbbm{R}$ such that $\sup_{X}u_{k}\leq C$ for any $k\in\mathbbm{N}$.
\end{itemize}
Then there exists a subsequence $u_{k_{h}}$ that converges strongly to $u\in\mathcal{E}^{1}(X,\omega,\psi)$ solution of $MA_{\omega}(u)=e^{-u}\mu$.
\end{thm}
If $\psi_{k}\in\mathcal{M}_{D,klt}^{+}$ then by Theorem \ref{thmB} the potential of any $[\psi_{k}]$-KE metric minimizes $D_{\psi_{k}}$, so the condition $(i)$ of Theorem \ref{thm:Closed} is automatically satisfied for any sequence of relative KE metrics that will be considered in the proof of Theorem \ref{thmA}. Instead, the assumption $(ii)$ is the real obstruction to the closedness and it is also necessary as Remark \ref{rem:CDS} shows.\newline

\begin{proof}{\textit{of Theorem \ref{thmA}}}\newline
\textbf{Step 1: Openness with respect to $\mathcal{T}:=\{[a,b)\}_{a<b}$.}\newline
We first note that since $[0,1]\ni t\to \psi_{t}\in \mathcal{M}_{klt}^{+}$ is weakly continuous and $\psi_{t}\preccurlyeq\psi_{s}$ if $t\leq s$ then
$$
[0,1]\ni t\to V_{\psi_{t}}=\int_{X}MA_{\omega}(\psi_{t})
$$
is continuous by what said in Section \S \ref{sec:Prelim} (i.e. Lemma $3.12$ in \cite{Tru20a}). Thus combining Theorem \ref{thmC} and Theorem \ref{thm:Open} it immediately follows that $S$ is open with respect to the topology generated by open sets $[a,b)$.\newline
\textbf{Step 2: Openness.}\newline
As a consequence of Step 1 to prove that $S$ is open it is sufficient to show that given $t_{0}\in S$, $t_{0}>0$ there exists $0<\epsilon\ll 1$ small enough such that $(t_{0}-\epsilon,t_{0}]\subset S$.\newline
By Theorem \ref{thmC} the $\psi_{t_{0}}$-Ding functional is coercive and there exists a unique $[\psi_{t_{0}}]$-KE metric. We denote by $u_{t_{0}}\in\mathcal{E}^{1}(X,\omega,\psi_{t_{0}})$ its potential given as solution of the Monge-Ampère equation
\begin{equation*}
\begin{cases}
MA_{\omega}(u_{t_{0}})=e^{-u_{t_{0}}}\mu\\
u_{t_{0}}\in\mathcal{E}^{1}(X,\omega,\psi_{t_{0}}).
\end{cases}
\end{equation*}
Then we assume by contradiction that there exists a sequence $t_{k}\nearrow t_{0}$ such that $t_{k}\notin S$ for any $k\in\mathbbm{N}$. By Theorem \ref{thmB} this means that $D_{\psi_{t_{k}}}$ does not admit a minimizer for any $k\in\mathbbm{N}$. Recalling that $D_{\psi_{t_{k}}}$ is translation invariant and lower-semicontinuous with respect to the weak topology (Proposition \ref{prop:ContDing}), we get that any minimizing sequence $\{u_{k,h}\}_{h\in\mathbbm{N}}\subset \mathcal{E}^{1}_{norm}(X,\omega,\psi_{t_{k}})$, for $k\in\mathbbm{N}$ fixed, necessarily satisfies $d(u_{k,h},\psi_{t_{k}})\to +\infty$ as $h\to +\infty$. Indeed if $d(\psi_{t_{k}},u_{k,h})\leq C$ then
$$
\{u_{k,h}\}_{h\in\mathbbm{N}}\subset\mathcal{E}^{1}_{C}(X,\omega,\psi_{t_{k}}):=\{u\in\mathcal{E}^{1}(X,\omega,\psi_{t_{k}})\, : \, \sup_{X}u\leq C, E_{\psi_{t_{k}}}(u)>-k\}
$$
which is weakly compact (Proposition \ref{prop:Usc}). Hence, up to considering a subsequence, $u_{k,h}\to u_{k}\in\mathcal{E}_{C}^{1}(X,\omega,\psi_{t_{k}})$ and
$$
D_{\psi_{t_{k}}}(u_{k})\leq \liminf_{h\to +\infty}D_{\psi_{t_{k}}}(u_{k,h})=\inf_{\mathcal{E}^{1}(X,\omega,\psi_{t_{k}})}D_{\psi_{t_{k}}},
$$
which would lead to $t_{k}\in S$ by Theorem \ref{thmB}.\newline
Therefore we can fix a sequence $\{u_{k}\}_{k\in\mathbbm{N}}$ such that $u_{k}\in\mathcal{E}^{1}(X,\omega,\psi_{t_{k}})$ for any $k\in\mathbbm{N}$, $d(u_{k},\psi_{t_{k}})\to +\infty$, $E_{\psi_{t_{k}}}(u_{k})=0$ and
$$
D_{\psi_{t_{k}}}(u_{k})< D_{\psi_{t_{k}}}(v_{k})
$$
for $v_{k}:=P_{\omega}[\psi_{t_{k}}](u_{t_{0}})-c_{k}$ where $c_{k}=V_{\psi_{t_{k}}}E_{\psi_{t_{k}}}\big(P_{\omega}[\psi_{t_{k}}](u_{t_{0}})\big)$. By continuity of the $\psi$-relative Ding functional with respect to decreasing sequences we may also assume without loss of generality that $u_{k}$ has $\psi_{t_{k}}$-relative minimal singularities.\newline
We now claim that $v_{k}\to u_{t_{0}}^{N}:=u_{t_{0}}-V_{\psi_{t_{0}}}E_{\psi_{t_{0}}}(u_{t_{0}})$ strongly, noting that by definition it is enough to prove that $\tilde{v}_{k}:=P_{\omega}[\psi_{t_{k}}](u_{t_{0}})$ converges strongly to $u_{t_{0}}$. But, letting $c:=\sup_{X}u_{t_{0}}$, $c-E_{\psi_{t_{k}}}(\tilde{v}_{k})=d(\psi_{t_{k}}+c,\tilde{v}_{k})\to d(\psi+c,u_{t_{0}})=c-E_{\psi}$ by Proposition $4.5$ in \cite{Tru19}, and the claim follows.\newline
In particular, as $\int_{X}e^{-v_{k}}d\mu\to \int_{X}e^{-u_{t_{0}}^{N}}d\mu$ by Theorem \ref{thm:DK99}, we also have
$$
D_{\psi_{t_{k}}}(v_{k})\to D_{\psi_{t_{0}}}(u_{t_{0}})
$$
as $k\to +\infty$. Next for $C=d(\psi_{t_{0}},u_{t_{0}}^{N})+1$ fixed and $k\gg 1$ big enough we denote by $w_{k}\in\mathcal{E}^{1}(X,\omega,\psi_{t_{k}})$ the element on the weak geodesic segment joining $v_{k}$ and $u_{k}$ such that $d(\psi_{t_{k}},w_{k})=C$. Note that such sequence $w_{k}$ exists as $d(\psi_{t_{k}},v_{k})\leq d(\psi_{t_{0}},u_{t_{0}}^{N})$ by Lemma \ref{lem:PropOld}. Moreover $E_{\psi_{t_{k}}}(w_{k})=0$ by linearity of the Monge-Ampère energy along weak geodesic segments (Theorem \ref{thm:Linear}). Then by convexity of the $\psi_{t_{k}}$-Ding functional it follows that 
$$
D_{\psi_{t_{k}}}(w_{k})<D_{\psi_{t_{k}}}(v_{k})
$$
for any $k\in\mathbbm{N}$. Furthermore by Lemma \ref{lem:PropOld} $|\sup_{X}w_{k}|\leq A$ uniformly as $d(\psi_{t_{k}},w_{k})=C$ and $V_{\psi_{t_{k}}}\geq V_{\psi_{0}}>0$. Hence up to considering a subsequence, $w_{k}\to w$ weakly where $w\in \mathcal{E}^{1}(X,\omega,\psi_{t_{0}})$ by the compactness of Proposition \ref{prop:Usc} which also yields $E_{\psi_{t_{0}}}(w)\geq 0$. Thus, as by Theorem \ref{thm:DK99} $\int_{X}e^{-w_{k}}d\mu\to \int_{X}e^{-w}d\mu$, we obtain
$$
D_{\psi_{t_{0}}}(w)\leq \liminf_{k\to +\infty}D_{\psi_{t_{k}}}(w_{k})\leq \lim_{k\to +\infty}D_{\psi_{t_{k}}}(v_{k})=D_{\psi_{t_{0}}}(u_{t_{0}})=\inf_{\mathcal{E}^{1}(X,\omega,\psi_{t_{0}})}D_{\psi_{t_{0}}}\leq D_{\psi_{t_{0}}}(w).
$$
Therefore $D_{\psi_{t_{k}}}(w_{k})\to D_{\psi_{t_{0}}}(w)$ which reads as $w_{k}\to w$ strongly. Moreover as $E_{\psi_{t_{0}}}(w)=E_{\psi_{t_{0}}}(u_{t_{0}}^{N})=0$ the uniqueness of solutions (Theorem \ref{thmB}) implies $w=u_{t_{0}}$. Finally the contradiction is given by
$$
C-d(\psi_{t_{k}},v_{k})=d(\psi_{t_{k}},w_{k})-d(\psi_{t_{k}},v_{k})\leq d(v_{k},w_{k}).
$$
since the left-hand side converges to $1$ ($v_{k}\to u_{t_{0}}^{N}$ strongly) while the right-hand side goes to $0$.\newline
\textbf{Step 3: Strong Continuity.}\newline
Let $\{t_{k}\}_{k\in\mathbbm{N}}\subset S$ be a converging sequence to $t_{0}\in S$ and denote by $u_{k}\in \mathcal{E}^{1}(X,\omega,\psi_{t_{k}})$ the unique potential of the corresponding KE metric such that
\begin{equation*}
\begin{cases}
MA_{\omega}(u_{k})=e^{-u_{k}}d\mu\\
u_{k}\in\mathcal{E}^{1}(X,\omega,\psi_{t_{k}}),
\end{cases}
\end{equation*}
and similarly for $u\in\mathcal{E}^{1}(X,\omega,\psi_{t_{0}})$ potential for the $[\psi_{t_{0}}]$-KE metric. As $\{\psi_{t}\}_{t\in[0,1]}$ is totally ordered, to prove that $u_{k}\to u$ strongly it is enough to consider the two monotonically cases $t_{k}\nearrow t_{0}, t_{k}\searrow t_{0}$ and show the strong convergence for subsequence.\newline
In the case $t_{k}\searrow t_{0}$, Theorem \ref{thm:Open} implies that there exist uniform coefficients for the coercivity if $k\gg 1$ big enough, i.e. there exists $A>0,B\geq 0$ such that
$$
D_{\psi_{t_{k}}}(v)\geq Ad(\psi_{t_{k}},v)-B
$$
for any $v\in\mathcal{E}^{1}_{norm}(X,\omega,\psi_{t_{k}})$. Thus, as clearly $D_{\psi_{t_{k}}}(u_{k})\leq D_{\psi_{t_{k}}}(\psi_{t_{k}})\leq C_{1}$ uniformly, we obtain $d(u_{k},\psi_{t_{k}})\leq C_{2}$ uniformly. Hence, by Lemma \ref{lem:PropOld} $|\sup_{X}u_{k}|\leq C_{3}$ uniformly and Theorem \ref{thm:Closed} concludes this case.\newline
If instead $t_{k}\nearrow t_{0}$ we first replace $u_{k},u$ respectively with $u_{k}^{N}:=u_{k}-V_{\psi_{t_{k}}}E_{\psi_{t_{k}}}(u_{k}), u^{N}:=u-V_{\psi_{t_{0}}}E_{\psi_{t_{0}}}(u)$ so that they have null relative energies. Then, proceeding by contradiction as in Step 2, we necessarily have $ d(\psi_{t_{k}},u_{k}^{N})\leq C_{4} $ uniformly, which again by Lemma \ref{lem:PropOld} implies 
\begin{equation}
\label{eqn:Cici}
\sup_{X}u_{k}^{N}\leq C_{5}.
\end{equation}
Therefore by the weak compactness of Proposition \ref{prop:Usc}, up to considering a subsequence, it follows that $u_{k}^{N}\to \tilde{u}\in \mathcal{E}^{1}(X,\omega,\psi_{t_{0}})$. On the other hand the Monge-Ampère equations yields
$$
\log \int_{X}e^{-u_{k}^{N}}d\mu=\log V_{\psi_{t_{k}}}+V_{\psi_{t_{k}}}E_{\psi_{t_{k}}}(u_{k}),
$$
and by Theorem \ref{thm:DK99} we deduce $E_{\psi_{t_{k}}}(u_{k})\leq C_{6}$ uniformly. Hence (\ref{eqn:Cici}) implies $\sup_{X}u_{k}\leq C_{7}$ and Theorem \ref{thm:Closed} concludes the proof.
\end{proof}
\begin{rem}
\emph{Observe that when $\psi\in \mathcal{M}_{KE}$ does not belong to $\mathcal{M}_{D}$ but its $\psi$-relative Ding functional is coercive, then a natural way to connect $\psi$ with model type envelopes $\tilde{\psi}\in\mathcal{M}_{D}, \tilde{\psi}\succcurlyeq \psi$ as in Theorem \ref{thmA} is to pass through the element $\psi'\in\mathcal{M}_{D}$ having the same singularity data of $\psi$. Indeed, if
$$
V_{\psi'}<AV_{\psi}
$$
where $A>1$ is given by Theorem \ref{thm:Open}, then the slopes of the $d$-coercivity of $D_{\psi_{s}}$ for $\psi_{s}:=s\psi'+(1-s)\psi$ are uniformly bounded. Thus, proceeding as in the proof of Theorem \ref{thmA}, appropriate potentials for the KE metrics are uniformly bounded from above and the strong continuity of Theorem \ref{thmA} holds for this path as a consequence of Theorem \ref{thm:Closed}.}
\end{rem}
\subsection{$0$-dimensional equisingularities}
\label{ssec:Estimate}
An estimate on the $\psi$-relative $\alpha$-invariant for $\psi\in\mathcal{M}_{klt}^{+}$ may give information about the existence of KE metrics by Theorem \ref{thmA2} (see also Question \ref{Question} and Conjecture \ref{conjA}).\newline
Moreover, as said in the Introduction, for singular model type envelopes the definition of the $\psi$-relative $\alpha$-invariant involves less functions, thus it might be easier to compute. A very natural case to consider is when $\psi$ has isolated singularities at $N$ points with the same \emph{weight} (i.e. with the same Lelong numbers at these points). Indeed, in this case the singularities can have weight arbitrarily small (which clearly implies $lct(X,\psi)$ arbitrarily big and $\tilde{\alpha}_{\omega}(\psi)\sim \alpha_{\omega}(\psi)$ thanks to Proposition \ref{prop:ConnAlphas}), the locus of the singularities is always $0$-dimensional and the total mass $V_{\psi}$ may basically be chosen arbitrarily and independent of the weight of the singularities. Moreover, we expect that, assuming $\mbox{Aut}(X)^{\circ}=\{\mbox{Id}\}$, any KE metric should be recovered by these KE metrics with $0$-dimensional equisingularities when $N$ moves to $+\infty$ and, conversely, that any of these sequences of singular KE metrics diverges if $X$ does not admit a KE metric. This \emph{point process} will be subject of studies in future works.\newline

Given an ample line bundle $L$ and $R:=\{p_{1},\dots,p_{N}\}$ a set of $N$ distinct points on $Y$ compact Kähler manifold, the \emph{multipoint Seshadri constant} at $p_{1},\dots,p_{N}$ is defined as
$$
\epsilon(L;p_{1},\dots,p_{N}):=\sup\{a>0\, : \, f_{N}^{*}L-a\mathbbm{E}_{N} \, \mbox{is nef}\,\}
$$
where $f_{N}:Z\to Y$ is the blow-up along $R$ and $\mathbbm{E}_{N}:=\sum_{j=1}^{N}E_{j}$ the sum of the exceptional divisors (see \cite{Dem90}, \cite{Laz04}, \cite{BDRH}, \cite{Tru18}). The definition extends to $\mathbbm{Q}$-line bundles by rescaling and to $\mathbbm{R}$-line bundles by continuity. Moreover $\epsilon(L;\cdot)$ is lower-semicontinuous and its supremum is reached outside a countable union of proper subvarieties, i.e. when the points are in \emph{very general position}. In this case we will use the notation $\epsilon(L;N)$ for simplicity.\newline
The characterization of the multipoint Seshadri constants in terms of jets implies that given $N\in\mathbbm{N}, \delta>0$ there exists a $\omega$-psh function $\varphi_{N,\delta}$ with analytic singularities formally encoded in $(\mathcal{I}_{R},\delta)$ if $\delta<\epsilon(L;N)$ (while $\delta\leq \epsilon(L;N)$ is a necessary condition). In particular, letting $\psi_{N,\delta}:=P_{\omega}[\varphi_{N,\delta}]$ and $\eta_{N,\delta}$ the big and semipositive form given by $f^{*}\omega_{\varphi_{N,\delta}}=\eta_{N,\delta}+\delta[\mathbbm{E}_{N}]$,
$$
V_{\psi_{N,\delta}}=\mathrm{Vol}_{Z}\big(\{\eta_{N,\delta}\}\big)=\mathrm{Vol}_{Y}(L)-N\delta^{n}
$$
where $\mathrm{Vol}_{Y}(L)=\int_{X}\omega^{n}$, $\mathrm{Vol}_{Z}(\{\eta_{N,\delta}\})=\int_{Z}\eta_{N,\delta}^{n}$. Observe also that $\delta=\nu(\psi_{N,\delta}):=\sup_{y\in Y}\nu(\psi_{N,\delta},y)$ by definition and that $\psi_{N,\delta}\in \mathcal{M}^{+}_{klt}$ if and only if $\delta<n$ (see also \cite{Dem90}).\newline

Then letting $\Theta\in H^{1,1}(Y,\mathbbm{R})$ be a pseudoeffective cohomology class on a compact manifold $Y$, we define the \emph{pseudoeffective threshold of $\Theta$ at $y$} as
$$
\sigma(\Theta;y):=\sup\{a\geq 0\, :\, f^{*}\Theta-a E\, \mbox{is pseudoeffective}\,\},
$$
where we denoted by $f:Z\to Y$ the blow-up at $y$ and with $E$ the exceptional divisor.
\begin{lem}
\label{lem:SecondEstimate}
Let $\Theta\in H^{1,1}(Y,\mathbbm{R})$ be a pseudoeffective cohomology class on a compact manifold $Y$ and let $\eta$ be a smooth closed $(1,1)$-form representative of $\Theta$. Then
\begin{equation*}
\sup_{u\in PSH(Y,\eta)}\nu(u,y)=\sigma(\Theta;y) \,\,\, \mbox{for any}\,\,\, y\in Y,\\
\end{equation*}
\end{lem}
\begin{proof}
For any $u\in PSH(Y,\eta)$ and any $y\in Y$,
$$
g^{*}(\eta_{u})-\nu(u,y)[E]
$$
is a closed and positive $(1,1)$-current, where $g:Z\to Y$ is the blow-up at $y$ and $E$ the exceptional divisor. Thus
$$
\sup_{u\in PSH(Y,\eta)}\nu(u,y)\leq \sigma(\Theta;y).
$$
On the other hand if $g^{*}\Theta-aE$ is pseudoeffective there exists a closed and positive $(1,1)$-current $T$ representative of $g^{*}\Theta-aE$. Therefore the current $T+a[E]$ is closed and positive with cohomology class $g^{*}\Theta$. But this implies that there exists a closed and positive current $S$ such that $g^{*}S=T+a[E]$ (see for instance Proposition $1.2.7.(ii)$ in \cite{BouTh}). Thus by the $\partial \bar{\partial}$-Lemma $S=\eta+dd^{c}u$ for $u\in PSH(Y,\eta)$, and by construction $\nu(u,y)=\inf_{z\in E}\nu(T+a[E],z)\geq a$ where we recall that the Lelong number of a closed and positive $(1,1)$-current is defined as the Lelong number of its potential once that a smooth form is fixed. Hence $\sup_{u\in PSH(Y,\eta)}\nu(u,y)\geq\sigma(\Theta;y)$ which concludes the proof. 
\end{proof}
Using the notation $\sigma(L;y):=\sigma\big(c_{1}(L);y\big)$, we can now state the following final estimate for the $\psi_{N,\delta}$-relative $\alpha$-invariant.
\begin{prop}
\label{prop:Final!}
Let $0<\delta<\min\{\epsilon(-K_{X};N),1\}$ and let $\psi\in \mathcal{M}^{+}$ the model type envelope with analytic singularity type formally encoded in $(\mathcal{I}_{S},\delta)$ where $S=\{p_{1},\dots,p_{N}\}$ is the set of points. Let also $L:=f^{*}(-K_{X})-\delta \mathbbm{E}$ the corresponding ample $\mathbbm{R}$-line bundle, where with obvious notations $f:Y\to X$ is the blow-up at $S$ and $\mathbbm{E}:=\sum_{j=1}^{N}E_{j}$ the sum of the exceptional divisors. Then, setting
\begin{gather*}
\sigma_{exc}(L):=\sup_{y\in\mathbbm{E}}\sigma(L;y),\\
\sigma_{gen}(L):=\sup_{y\in Y\setminus \mathbbm{E}}\sigma(L;y),\\
\epsilon_{exc}(L):=\inf_{y\in\mathbbm{E}}\epsilon(L;y),\\
\epsilon_{gen}(L):=\inf_{y\in Y\setminus \mathbbm{E}}\epsilon(L;y),
\end{gather*}
we have
\begin{gather}
\label{eqn:Final1}
\alpha_{\omega}(\psi)\geq  \min\Big\{\frac{1-\delta}{\sigma_{exc}(L)}, \frac{1}{\sigma_{gen}(L)}\Big\},\\
\label{eqn:Final2}
\alpha_{\omega}(\psi)\geq \min\Big\{\frac{(1-\delta)\epsilon_{exc}(L)^{n-1}}{V_{0}-\delta^{n}N}, \frac{\epsilon_{gen}(L)^{n-1}}{V_{0}-\delta^{n}N}\Big\}
\end{gather}
where $V_{0}:=\mbox{Vol}_{X}(-K_{X})=(-K_{X})^{n}$
\end{prop}
\begin{proof}
Proposition \ref{prop:Anal} yields a bijection between $PSH(X,\omega,\psi):=\{u\in PSH(X,\omega)\, : \, u\preccurlyeq \psi \} $ and $PSH(Y,\eta)$ where $\eta$ is a smooth closed $(1,1)$-form with cohomology class $c_{1}(L)$. Moreover denoting with $\tilde{u}\in PSH(X,\eta)$ the function corresponding to $u\in PSH(X,\omega,\psi)$, it follows by construction that, for any $\alpha>0$,
$$
\nu\big(\alpha(u-\psi)+\psi,p_{j}\big)=\delta +\alpha \inf_{y\in E_{j}}\nu(\tilde{u},y),
$$  
while $\nu\big(\alpha(u-\psi)+\psi,x\big)=\alpha\nu(\tilde{u},f^{-1}x)$ if $x\notin S$. Thus we get
\begin{gather}
\label{eqn:Ug1}
\delta+\alpha\sup_{y\in \mathbbm{E}}\sup_{\tilde{u}\in PSH(Y,\eta)}\nu(\tilde{u},y)\geq \sup_{j=1,\dots,N}\sup_{u\preccurlyeq \psi}\nu\big(\alpha(u-\psi)+\psi,p_{j}\big),\\
\label{eqn:Ug2}
\alpha\sup_{y\in Y\setminus \mathbbm{E}}\sup_{\tilde{u}\in PSH(Y,\eta)}\nu(\tilde{u},y)=\sup_{x\in X\setminus S}\sup_{u\preccurlyeq \psi}\nu\big(\alpha(u-\psi)+\psi,x\big).
\end{gather}
Then by Lemma \ref{lem:SecondEstimate} the left-hand sides in (\ref{eqn:Ug1}) and in (\ref{eqn:Ug2}) correspond respectively to $\delta+\alpha\sigma_{exc}(L)$ and $\alpha\sigma_{gen}(L)$. Therefore
$$
\sup_{x\in X}\sup_{u\preccurlyeq \psi}\nu \big(\alpha(u-\psi)+\psi,x\big)\leq \max\big\{\delta +\alpha\sigma_{exc}(L), \alpha \sigma_{gen}(L)\big\},
$$
which leads to $c_{\psi}(u)\geq \alpha$ for any $u\preccurlyeq \psi$ when $\alpha>0$ satisfies
$$
\alpha < \min \Big\{\frac{1-\delta}{\sigma_{exc}(L)}, \frac{1}{\sigma_{gen}(L)}\Big\}
$$
thanks to Theorem \ref{thm:Skoda}. Thus by Lemma \ref{lem:Exp} and Remark \ref{rem:AnSinTerm} we deduce that
$$
\alpha_{\omega}(\psi)\geq \min\Big\{\frac{1-\delta}{\sigma_{exc}(L)}, \frac{1}{\sigma_{gen}(L)}\Big\},
$$
i.e. (\ref{eqn:Final1}). Next, (\ref{eqn:Final2}) is a consequence of (\ref{eqn:Final1}) as
\begin{equation*}
\sigma(L;y)\epsilon(L;y)^{n-1}\leq \mbox{Vol}_{X}(L).
\end{equation*}
holds for any $y\in Y$. One easy way to check this last inequality is through the convexity of the Okounkov body of $L$ at $y$ with respect to an \emph{infinitesimal flag} (see \cite{LM09}, \cite{KL??}).
\end{proof}
\begin{rem}
\emph{As seen during the proof of Proposition \ref{prop:Final!} the lower bound in terms of the pseudoeffective thresholds is sharper than the one given by the Seshadri constants. However, giving upper bounds for the pseudoeffective threshold is usually harder than finding lower bounds for the Seshadri constant. Moreover the latter is much more studied in the literature as it is related to renowned conjectures in Algebraic Geometry (see \cite{BDRH}).}
\end{rem}
{\footnotesize
\bibliographystyle{acm}
\bibliography{main}
}
\end{document}